\documentclass[11pt]{article}

\title{\vspace{-2em}Multiperiod Martingale Transport}

\author{Marcel Nutz\thanks{
  Departments of Statistics and Mathematics, Columbia University, mnutz@columbia.edu. Research supported by an Alfred P.\ Sloan Fellowship and NSF Grants DMS-1512900 and DMS-1812661.
  } 
  \and 
  Florian Stebegg\thanks{Department of Statistics, Columbia University, florian.stebegg@columbia.edu.
  } 
  \and 
  Xiaowei Tan\thanks{Department of Mathematics, Columbia University, xt2161@columbia.edu.
  }}

\usepackage[utf8]{inputenc}
\usepackage[T1]{fontenc}

\usepackage{amsmath,amsfonts,amsthm,amssymb}
\usepackage{bm,bbm}
\usepackage{enumerate}
\usepackage{stmaryrd}
\usepackage{color}
\usepackage{tikz}
\usetikzlibrary{decorations.pathreplacing}

\theoremstyle{plain}
\newtheorem{proposition}{Proposition}[section]
\newtheorem{lemma}[proposition]{Lemma}
\newtheorem{theorem}[proposition]{Theorem}

\theoremstyle{definition}
\newtheorem{definition}[proposition]{Definition}
\newtheorem{remark}[proposition]{Remark}
\newtheorem{example}[proposition]{Example}

\renewenvironment{thebibliography}[1]{%
\begin{oldthebibliography}{#1}%
\setlength{\baselineskip}{.75em}
\linespread{.90}
\small
\setlength{\parskip}{0.3ex}%
\setlength{\itemsep}{.0em}%
}%
{%
\end{oldthebibliography}%
}

\newcommand{\R}{\mathbb{R}}

\newcommand{\B}{\mathfrak{B}}
\newcommand{\RE}{\bar{\mathbb{R}}}

\newcommand{\supp}{\mathrm{supp}}

\newcommand{\M}{\mathcal{M}}
\newcommand{\1}{\mathbf{1}}

\newcommand{\EE}{\mathbb{E}}
\newcommand{\D}{\mathcal{D}}
\newcommand{\U}{\mathcal{U}}

\newcommand{\F}{\mathfrak{F}}
\newcommand{\FF}{\mathbb{F}}
\newcommand{\V}{\mathcal{V}}
\newcommand{\conv}{\mathrm{conv}}
\newcommand{\bary}{\mathrm{bary}}
\newcommand{\conc}{\mathrm{conc}}
\renewcommand{\S}{\mathbf{S}}
\newcommand{\I}{\mathbf{I}}
\newcommand{\bC}{\mathbf{C}}

\newcommand{\bmu}{{\bm{\mu}}}
\newcommand{\bphi}{{\bm{\phi}}}
\newcommand{\bchi}{{\bm{\chi}}}

\newcommand{\bx}{{\bm{x}}}

\newcommand{\bX}{X}

\newcommand{\bk}{\bm{k}}
\newcommand{\bbk}{\mathbbm{k}}

\newcommand{\shadow}[2]{\mathcal{S}^{#2}(#1)}
\newcommand{\casts}[2]{\left\llbracket #1;#2 \right\rrbracket}

\numberwithin{equation}{section}

\newcommand{\MN}[1]{#1}
\newcommand{\MNN}[1]{#1}
\newcommand{\FS}[1]{#1}

\usepackage[pdfborder={0 0 0}]{hyperref}
\hypersetup{
  urlcolor = black,
  pdfauthor = {Marcel Nutz, Florian Stebegg, Xiaowei Tan},
  pdfkeywords = {Optimal Transport; Martingale Coupling; Duality},
  pdftitle = {Multiperiod Martingale Transport},
  pdfsubject = {Multiperiod Martingale Transport},
  pdfpagemode = UseNone
}

\begin{document}

\maketitle \vspace{-1.2em}

\begin{abstract}
Consider a multiperiod optimal transport problem where distributions $\mu_{0},\dots,\mu_{n}$ are prescribed and a transport corresponds to a scalar martingale $X$ with marginals $X_{t}\sim\mu_{t}$. We introduce particular couplings called left-monotone transports; they are characterized equivalently by a no-crossing property of their support, as simultaneous optimizers for a class of bivariate transport cost functions with a Spence--Mirrlees property, and by an order-theoretic minimality property. Left-monotone transports are unique if $\mu_{0}$ is atomless, but not in general. In the one-period case $n=1$, these transports reduce to the Left-Curtain coupling of Beiglb\"ock and Juillet. In the multiperiod case, the bivariate marginals for dates $(0,t)$ are of Left-Curtain type, if and only if $\mu_{0},\dots,\mu_{n}$ have a specific order property.
The general analysis of the transport problem also gives rise to a strong duality result and a description of its polar sets. Finally, we study a variant where the intermediate marginals $\mu_{1},\dots,\mu_{n-1}$ are not prescribed.
\end{abstract}

\vspace{0.9em}

{\small

}

\section{Introduction}

Let $\bmu = (\mu_0,\dots,\mu_n)$ be a vector of probability measures $\mu_{t}$ on the real line.
A measure $P$ on $\R^{n+1}$ whose marginals are given by $\bmu$ is called a coupling (or transport)
of $\bmu$, and the set of all such measures is denoted by $\Pi(\bmu)$. We shall be interested
in couplings $P$ that are martingales; that is, the identity $\bX = (X_0,\dots,X_n)$ on $\R^{n+1}$ is a martingale under $P$. Hence,
we will assume that all marginals have a finite first moment and denote by
$\M(\bmu)$ the set of martingale couplings. A classical result of Strassen \cite{Strassen.65}
shows that $\M(\bmu)$ is nonempty if and only if the marginals are in convex order,
denoted by $\mu_{t-1} \leq_c \mu_t$ and defined by the requirement that $\mu_{t-1}(\phi) \leq \mu_t(\phi)$
for any convex function $\phi$, where $\mu(\phi) := \int \phi \, d\mu$.

The first goal of this paper is to introduce and study a family of ``canonical'' couplings $P\in\M(\mu)$ that we call left-monotone. These couplings specialize to the Left-Curtain coupling of~\cite{BeiglbockJuillet.12} in the one-step case $n=1$ and share, broadly speaking, several properties reminiscent of the Hoeffding--Fr\'echet coupling of classical optimal transport. Indeed, left-monotone couplings will be characterized by order-theoretic minimality properties, as simultaneous optimal
transports for certain classes of reward (or cost) functions, and through
no-crossing conditions on their supports.

The second goal is to develop a strong duality theory for multiperiod martingale optimal transport, along the lines of~\cite{BeiglbockNutzTouzi.15} for the one-period martingale case and~\cite{Kellerer.84} for the classical optimal transport problem. That is, we introduce a suitable dual optimization problem and show the absence of a duality gap as well as the existence of dual optimizers for general transport reward (or cost) functions. 
The duality result is a crucial tool for the study of the left-monotone couplings.

We also develop similar results for a variant of our problem where the intermediate marginals $\mu_{1},\dots,\mu_{n-1}$ are not prescribed (Section~\ref{se:freeIntermediate}), but we shall focus on the full marginal case for the purpose of the Introduction.

\subsection{Left-Monotone Transports}

For the sake of orientation, let us first state the main result and then explain the terminology contained therein. The following is a streamlined version---the results in the body of the paper are stronger in some technical aspects.

\begin{theorem}\label{th:leftMonotoneIntro}
Let $\bmu = (\mu_0,\dots,\mu_n)$ be in convex order and
$P \in \M(\bmu)$ a martingale transport between these marginals. The following
are equivalent:
\begin{enumerate}
	\item $P$ is a simultaneous optimal transport for $f(X_0,X_t)$, $1\leq t \leq n$ whenever $f:\R^{2}\to\R$ is a smooth second-order Spence--Mirrlees function.
	\item $P$ is concentrated on a left-monotone set $\Gamma\subseteq\R^{n+1}$.
	\item $P$ transports $\mu_0|_{(-\infty,a]}$ to the obstructed shadow
	$\shadow{\mu_0|_{(-\infty,a]}}{\mu_1,\dots,\mu_t}$ in step $t$, for all $1 \leq t \leq n$
	and $a \in \R$.
\end{enumerate}
There exists $P \in \M(\bmu)$ satisfying (i)--(iii), and any such $P$ is called a left-monotone transport. If $\mu_{0}$ is atomless, then $P$ is unique.
\end{theorem}

Let us now discuss the items in the theorem.

\paragraph{(i) Optimal Transport.} This property characterizes $P$ as a simultaneous optimal transport. Given a function $f:\R^{n+1}\to\R$, we may consider the martingale optimal transport problem with reward $f$ (or cost $-f$),
\begin{equation}\label{eq:motIntro}
  \S_{\bmu}(f) = \sup_{P \in \M(\bmu)} P(f);
\end{equation}
recall that $P(f)=\EE^{P}[f(X_{0},\dots,X_{n})]$. 
A Lipschitz function $f\in C^{1,2}(\R^{2};\R)$ is called a \emph{smooth second-order Spence--Mirrlees} function if it satisfies the cross-derivative condition $f_{xyy}>0$; this has also been called the martingale Spence--Mirrlees condition in analogy to the classical Spence--Mirrlees condition $f_{xy}>0$. Given such a function of two variables and $1\leq t\leq n$, we may consider the $n$-step martingale optimal transport problem with reward $f(X_0,X_t)$. Characterization~(i) states that a left-monotone transport $P\in\M(\bmu)$ is an optimizer  simultaneously for the $n$ transport problems $f(X_0,X_t)$, $1\leq t\leq n$, for some (and then all) smooth second-order Spence--Mirrlees functions $f$.

In the one-step case, a corresponding result holds for the Left-Curtain coupling~\cite{BeiglbockJuillet.12}; here the simultaneous optimization becomes a single one. In view of the characterization in~(i), an immediate consequence is that if there exists $P\in\M(\bmu)$ such that all bivariate projections $P_{0t}=P\circ (X_{0},X_{t})^{-1}\in\M(\mu_{0},\mu_{t})$ are of Left-Curtain type, then $P$ is left-monotone. However, such a transport does not exist unless the marginals satisfy a very specific condition (see Proposition~\ref{pr:multiLeftCurtain}), and in general the bivariate projections of a left-monotone transport are \emph{not} of Left-Curtain type.

\paragraph{(ii) Geometry.} The second item characterizes $P$ through a geometric property of its support. A set $\Gamma \subseteq \R^{n+1}$ will be called \emph{left-monotone} if it has the following no-crossing property for all $1 \leq t \leq n$: Let $\bx=(x_{0},\dots,x_{t-1})$, $\bx'=(x'_{0},\dots,x'_{t-1}) \in \R^t$ and 
\begin{center}
$y^-,y^+,y' \in \R$~~with~~$y^- < y^+$
\end{center}
\MN{be such} that $(\bx,y^+), (\bx,y^-), (\bx',y')$ are in the projection of $\Gamma$ to the first $t+1$ coordinates. Then,
\begin{center}
  $y' \notin (y^-,y^+)$~~whenever~~$x_0 < x_0'$.
\end{center}
That is, if we consider two paths in $\Gamma$ starting at $x_0$ and coinciding up to $t-1$, and a third path starting at $x_0'$ to the right of $x_{0}$, then at time $t$ the third path cannot step in-between the first two---this is illustrated in Figure~\ref{fig:monotone}.
Item~(ii) states that a left-monotone transport $P\in\M(\bmu)$ can be characterized by the fact that it is concentrated on a left-monotone set $\Gamma$. (In Theorem~\ref{th:leftMonotoneMain} we shall state a stronger result: we can find a left-monotone set that carries all left-monotone transports at once.)

In the one-step case $n=1$, left-monotonicity coincides with the Left-Curtain property of~\cite{BeiglbockJuillet.12}. However, we emphasize that for $t>1$, our no-crossing condition differs from the Left-Curtain property of the bivariate projection $(X_{0},X_{t})(\Gamma)$ as the latter would not contain the restriction that the first two paths have to coincide up to $t-1$ (see also Example~\ref{ex:NotLeftCurtain}). This corresponds to the mentioned fact that the bivariate marginal $P_{0t}$ need not be of Left-Curtain type. On the other hand, the geometry of the projection $(X_{t-1},X_{t})(\Gamma)$ is also quite different from the Left-Curtain one, as our condition may rule out third paths crossing from the right \emph{and} left at $t-1$, depending on the starting point $x_{0}'$ rather than the location of $x_{t-1}'$.

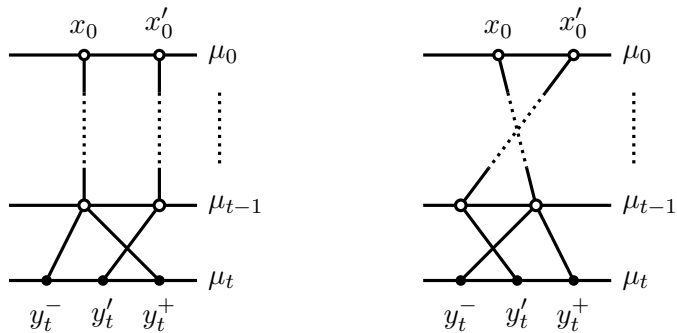
\begin{figure}
\begin{center}
\begin{tabular}{ccc}
\scalebox{1}{
\begin{tikzpicture}
\draw[very thick] (-1.5,3) -- (1,3) node[right] {$\mu_0$};
\draw[very thick] (-1.5,1) -- (1,1) node[right] {$\mu_{t-1}$};
\draw[very thick] (-1.5,0) -- (1,0) node[right] {$\mu_t$};

\draw[very thick,dotted] (1.3,2.5) -- (1.3,1.5) {};

\draw[very thick] (-0.5,3) node[minimum width=3.5pt,inner sep=0,draw,fill=white,circle,label=$x_0$] {} -- (-0.5,2.5) {};
\draw[very thick,dotted] (-0.5,2.5) -- (-0.5,1.5) {};
\draw[very thick] (-0.5,1.5) -- (-0.5,1) {};
\draw[very thick] (-0.5,1) -- (-1,0) node[minimum width=4pt,inner sep=0,fill,circle,label=below:$y_t^-$] {};
\draw[very thick] (-0.5,1) node[minimum width=4pt,inner sep=0,draw,fill=white,circle] {} -- (0.5,0) node[minimum width=4pt,inner sep=0,fill,circle,label=below:$y_t^+$] {};

\draw[very thick] (0.5,3) node[minimum width=3.5pt,inner sep=0,draw,fill=white,circle,label=$x_0'$] {} -- (0.5,2.5) {};
\draw[very thick,dotted] (0.5,2.5) -- (0.5,1.5) {};
\draw[very thick] (0.5,1.5) -- (0.5,1) {};
\draw[very thick] (0.5,1) node[minimum width=4pt,inner sep=0,draw,fill=white,circle] {} -- (-0.25,0) node[minimum width=4pt,inner sep=0,fill,circle,label=below:$y_t'$] {};
\end{tikzpicture}}
& \hspace{1cm} &
\scalebox{1}{
\begin{tikzpicture}
\draw[very thick] (-2,3) -- (0.5,3) node[right] {$\mu_0$};
\draw[very thick] (-2,1) -- (0.5,1) node[right] {$\mu_{t-1}$};
\draw[very thick] (-2,0) -- (0.5,0) node[right] {$\mu_t$};

\draw[very thick,dotted] (0.8,2.5) -- (0.8,1.5) {};

\draw[very thick] (-1,3) node[minimum width=3.5pt,inner sep=0,draw,fill=white,circle,label=$x_0$] {} -- (-0.875,2.5) {};
\draw[very thick,dotted] (-0.875,2.5) -- (-0.625,1.5) {};
\draw[very thick] (-0.625,1.5) -- (-0.5,1) {};
\draw[very thick] (-0.5,1) -- (-1.5,0) node[minimum width=4pt,inner sep=0,fill,circle,label=below:$y_t^-$] {};
\draw[very thick] (-0.5,1) node[minimum width=4pt,inner sep=0,draw,fill=white,circle] {} -- (0,0) node[minimum width=4pt,inner sep=0,fill,circle,label=below:$y_t^+$] {};

\draw[very thick] (0,3) node[minimum width=3.5pt,inner sep=0,draw,fill=white,circle,label=$x_0'$] {} -- (-0.375,2.5) {};
\draw[very thick,dotted] (-0.375,2.5) -- (-1.125,1.5) {};
\draw[very thick] (-1.125,1.5) -- (-1.5,1) {};
\draw[very thick] (-1.5,1) node[minimum width=4pt,inner sep=0,draw,fill=white,circle] {} -- (-0.75,0) node[minimum width=4pt,inner sep=0,fill,circle,label=below:$y_t'$] {};
\end{tikzpicture}}
\end{tabular}
\end{center}
\caption{Two examples of forbidden configurations in left-monotone sets.}
\label{fig:monotone}
\end{figure}

\paragraph{(iii) Convex Ordering.} This property characterizes left-monotone transports in an order-theoretic way and will be used in the existence proof. To explain the idea, suppose that~$\mu_{0}$ consists of finitely many atoms at $x_{1},\dots,x_{N}\in\R$. Then, for any fixed $t$, a coupling of $\mu_{0}$ and $\mu_{t}$ can be defined by specifying a ``destination'' measure for each atom. We consider all chains\footnote{\MN{Here $\mu_0|_{x_{i}}$ denotes a Dirac measure of mass $\mu_0(\{x_{i}\})$ at $x_{i}$.}} $\mu_0|_{x_{i}} \leq_c \theta_1 \leq_c \dots \leq_c \theta_t$ of measures $\theta_{s}$ in convex order that satisfy the marginals constraints $\theta_{s}\leq \mu_{s}$ for $s\leq t$. Of these chains, keep only the terminal measures $\theta_{t}$ and \FS{compare them}
according to the convex order. The \emph{obstructed shadow of $\mu_{0}|_{x_{1}}$ in $\mu_t$ through $\mu_1, \dots,\mu_{t-1}$}, denoted $\shadow{\mu_0|_{x_{i}}}{\mu_1,\dots,\mu_t}$, is defined as the unique least element\footnote{\FS{See Definition~\ref{def:obstructedShadow} and Lemma~\ref{le:shadowLeastElement} for details on this construction.}} among the $\theta_{t}$. A particular coupling of $\mu_{0}$ and $\mu_{t}$ is the one that successively maps the atoms $\mu_0|_{x_{i}}$ to their obstructed shadows in the remainder of $\mu_{t}$, starting with the left-most atom $x_{i}$ and continuing from left to right. In the case of general measures, we  consider the restrictions $\mu_{0}|_{(-\infty,a]}$ instead of successively mapping the atoms.
Characterization~(iii) then states that a left-monotone transport $P\in\M(\bmu)$ maps $\mu_{0}|_{(-\infty,a]}$ to its obstructed shadow at date $t$ for all $1\leq t\leq n$ and $a\in\R$. This shows in particular that the bivariate projections $P_{0t}=P\circ (X_{0},X_{t})^{-1}$ of a left-monotone coupling are uniquely determined. In the body of the text, we shall also give an alternative definition of the obstructed shadow by iterating unobstructed shadows through the marginals up to date~$t$; see Section~\ref{se:shadows}.

The above specializes to the construction of~\cite{BeiglbockJuillet.12} for the one-step case, which corresponds to the situation of $t=1$ where there are no intermediate marginals obstructing the shadow. When $t>1$, the obstruction by the intermediate marginals once again entails that $P_{0t}$ need not be of Left-Curtain type. More precisely, Characterization~(iii) gives rise to a sharp criterion (Proposition~\ref{pr:multiLeftCurtain}) on the marginals $\bmu$, describing exactly when this coincidence arises.

\paragraph{(Non-)Uniqueness.} We have seen above that for a left-monotone transport $P\in\M(\bmu)$ the bivariate projections $P_{0t}$, $1\leq t\leq n$ are uniquely determined. In particular, for $n=1$, we recover the result of~\cite{BeiglbockJuillet.12} that the \MNN{left-monotone} coupling is unique. For $n>1$, the situation turns out to be quite different depending on the nature of the first marginal. On the one extreme, we shall see that when $\mu_{0}$ is atomless, there is a unique left-monotone transport $P\in\M(\bmu)$. Moreover, $P$ has a degenerate structure reminiscent of Brenier's theorem: it can be disintegrated as $P=\mu_{0}\otimes\kappa_{1}\otimes \cdots\otimes\kappa_{n}$ where each one-step transport kernel $\kappa_{t}$ is concentrated on the graphs of two functions. On the other extreme, if $\mu_{0}$ is a Dirac mass, the typical case is that there are infinitely many left-monotone couplings---see Section~\ref{se:uniqueness} for a detailed discussion. We shall also show that left-monotone transports are not Markovian in general, even if uniqueness holds (Example~\ref{ex:NotMarkovian}).

\subsection{Duality}

The analysis of left-monotone transports is based on a duality result that we develop for general reward functions $f:\R^{n+1}\to(-\infty,\infty]$ with an integrable lower bound. Formally, the dual problem (in the sense of linear programming) for the transport problem  $\S_{\bmu}(f) = \sup_{P \in \M(\bmu)} P(f)$ is the minimization
$$
 \I_\bmu(f) := \inf_{(\bphi,H)} \sum_{t=0}^{n}\mu_{t}(\phi_{t})
$$
where the infimum is taken over vectors $\bphi=(\phi_{0},\dots,\phi_{n})$ of  real functions and predictable processes $H=(H_{1},\dots,H_{n})$ such that 
\begin{equation}\label{eq:dualIneqIntro}
  \sum_{t=0}^{n} \phi_{t}(X_{t}) + (H\cdot X)_{n}\geq f;
\end{equation}
here $(H \cdot X)_n := \sum_{t=1}^n H_t\, (X_t - X_{t-1})$ is the discrete-time integral. The desired result (Theorem~\ref{thm:duality}) states that there is no duality gap, i.e.\ $\I_\bmu(f)=\S_{\bmu}(f)$, and that the dual problem is attained whenever it is finite. From the analysis for the one-step case in~\cite{BeiglbockNutzTouzi.15} we know that this assertion fails for the above naive formulation of the dual, and requires several relaxations regarding the integrability of the functions $\phi_{t}$ and the domain $\V\subseteq \R^{n+1}$ where the inequality~\eqref{eq:dualIneqIntro} is required. Specifically, the inequality needs to be relaxed on sets that are $\M(\bmu)$-polar; i.e.\ not charged by any transport $P\in\M(\bmu)$. These sets are characterized in Theorem~\ref{thm:polarstruct} where we show that the $\M(\bmu)$-polar sets are precisely the (unions of) sets which project to a two-dimensional polar set of $\M(\mu_{t-1},\mu_{t})$ for some $1\leq t\leq n$.

The duality theorem gives rise to a monotonicity principle (Theorem~\ref{thm:monotonicityPrinciple}) that underpins the analysis \MN{of the} left-monotone couplings. Similarly to the cyclical monotonicity condition in classical transport, it allows one to study the geometry of the support of optimal transports for a given function~$f$.

\subsection{Background and Related Literature}

The martingale optimal transport problem~\eqref{eq:motIntro} was introduced in~\cite{BeiglbockHenryLaborderePenkner.11} with the dual problem as a motivation. Indeed, in financial mathematics the function $f$ is understood as the payoff of a derivative written on the underlying~$X$ and~\eqref{eq:dualIneqIntro} corresponds to superhedging $f$ by statically trading in European options $\phi_{t}(X_{t})$ and dynamically trading in the underlying according to the strategy $H$. The value $\I_\bmu(f)$ then corresponds to the lowest price of $f$ for which the seller can enter a model-free hedge $(\bphi,H)$ if the marginals $X_{t}\sim \mu_{t}$ are known from option market data.
In~\cite{BeiglbockHenryLaborderePenkner.11}, it was shown (with the above, ``naive'' formulation of the dual problem) that there is no duality gap if $f$ is sufficiently regular, whereas dual existence was shown to fail even in regular cases. The idea of model-free hedging as well as the connection to Skorokhod embeddings goes back to~\cite{Hobson.98}; we refer to \cite{BiaginiBouchardKardarasNutz.14, BouchardNutz.13, CoxObloj.11, Hobson.11,Obloj.04,Touzi.14} for further references. A specific multiperiod martingale optimal transport problem also arises in the study of the maximum maximum of a martingale given $n$ marginals~\cite{HenryLabordereOblojSpoidaTouzi.12}.

The one-step case $n=1$ has been studied in great detail. In particular, \cite{BeiglbockJuillet.12} introduced the Left-Curtain coupling and pioneered numerous ideas underlying Theorem~\ref{th:leftMonotoneIntro}, \cite{HenryLabordereTouzi.13} provided an explicit construction of that coupling, and \cite{Juillet.14} established the stability with respect to the marginals. Our duality results specialize to the ones of \cite{BeiglbockNutzTouzi.15} when $n=1$. Unsurprisingly, we shall exploit many arguments and results from these papers wherever possible. As indicated above, and as will be seen in the proofs below, the multistep case allows for a richer structure and necessitates novel ideas; for instance, the analysis of the polar sets (Theorem~\ref{thm:polarstruct}) is surprisingly involved. 
Other works in the one-step martingale case have studied reward functions $f$ such as forward start straddles
\cite{HobsonKlimmek.15,HobsonNeuberger.12} or Asian payoffs \cite{Stebegg.14}. We also refer to \cite{GhoussoubKimLim.15, DeMarchTouzi.17} for recent developments with multidimensional marginals.

One-step martingale optimal transport problems can alternately be studied as optimal Skorokhod embedding problems with marginal constraints; cf.\ \cite{BeiglbockCoxHuesman.14, BeiglbockCoxHuesmannPerkowskiPromel.15, BeiglbockHenryLabordereTouzi.15, BeiglbockHuesmannStebegg.16}. A multi-marginal extension \cite{BeiglbockCoxHuesman.17} of~\cite{BeiglbockCoxHuesman.14} is in preparation at the time of writing and the authors have brought to our attention that it will offer a version of Theorem~\ref{th:leftMonotoneIntro} in the Skorokhod picture, at least in the case where $\mu_{0}$ is atomless and some further conditions are satisfied. The Skorokhod embedding problem with multi-marginal constraint was also studied in~\cite{GuoTanTouzi.15a}.

A multi-step coupling quite different from ours can be obtained by composing in a Markovian fashion the Left-Curtain transport kernels from $\mu_{t-1}$ to $\mu_{t}$, $1\leq t\leq n$, as discussed in~\cite{HenryLabordereTouzi.13}. In~\cite{Juillet.15} the continuous-time limits of such couplings for $n\to\infty$ are studied to find solutions of the so-called Peacock problem~\cite{HirschProfetaRoynetteYor.11} where the marginals for a continuous-time martingale are prescribed; see also \cite{HenryLabordereTanTouzi.14} and \cite{KallbladTanTouzi.15} for other continuous-time results with full marginal constraint. Early contributions related to the continuous-time martingale transport problem include \cite{DolinskySoner.12,DolinskySoner.14, GalichonHenryLabordereTouzi.11, NeufeldNutz.12, SonerTouziZhang.2010dual, TanTouzi.11}.

The remainder of the paper is organized as follows. Section~\ref{se:preliminaries} fixes basic terminology and recalls the necessary results from the one-step case. In Section~\ref{se:polarStructure}, we characterize the polar structure of $\M(\bmu)$. Section~\ref{se:dualSpace} introduces and analyzes the space that is the domain of the dual problem in Section~\ref{se:duality}, where we state the duality theorem and the monotonicity principle.  Section~\ref{se:shadows} introduces left-monotone transports by the shadow construction and Section~\ref{se:geometry} develops the equivalent characterizations in terms of support and optimality properties. The (non-)uniqueness of left-monotone transports is discussed in Section~\ref{se:uniqueness}. We conclude with the analysis of the problem with unconstrained intermediate marginals in Section~\ref{se:freeIntermediate}.

\section{Preliminaries}\label{se:preliminaries}

Throughout this paper, $\mu_{t}, \mu, \nu$ denote finite measures on $\R$ with finite
first moment, the total mass not necessarily being normalized. Generalizing the notation from the Introduction to a vector $\bmu = (\mu_0,\dots,\mu_n)$ of such measures, we will write $\Pi(\bmu)$ for the set of couplings; that is, measures $P$ on $\R^{n+1}$ such that $P \circ X_t^{-1} = \mu_t$ for $0 \leq t \leq n$ where
$\bX = (X_0,\dots,X_n) : \R^{n+1} \to \R^{n+1}$ is the identity. Moreover, $\M(\bmu)$ is the
subset of all $P \in \Pi(\bmu)$ that are martingales, meaning that
\[\int X_s\1_A(X_0,\dots,X_s) dP = \int X_t\1_A(X_0,\dots,X_s) dP\]
for all $s \leq t$ and Borel sets $A \in \B(\R^{s+1})$. 

We denote by $\FF = \{\F_t\}_{0 \leq t \leq n}$ the canonical filtration $\F_t := \sigma(X_0,\dots,X_t)$. As usual, an $\FF$-predictable process $H = \{H_t\}_{1 \leq t \leq n}$
is a sequence of real functions on $\R^{n+1}$ such that $H_t$ is $\F_{t-1}$-measurable; i.e.\
$H_t = h_t(X_0,\dots,X_{t-1})$ for some Borel-measurable $h_t: \R^t \to \R$.
Given an $\FF$-predictable process $H$, the discrete stochastic integral $\{(H \cdot X)_t\}_{0 \leq t \leq n}$ is defined by 
\[(H \cdot X)_t := \sum_{s=1}^t H_s \cdot (X_s - X_{s-1}).\]
If $\bX$ is a martingale under some measure $P$, then $H \cdot X$
 is a generalized (not necessarily integrable) martingale in the sense 
of generalized conditional expectations; cf.\ \cite[Proposition~1.64]{JacodShiryaev.03}.

We say that $\bmu = (\mu_0,\dots,\mu_n)$ is in convex order if $\mu_{t-1} \leq_c \mu_{t}$ for all $1\leq t\leq n$; that is, $\mu_{t-1}(\phi) \leq \mu_{t}(\phi)$ for any convex function $\phi: \R \to \R$. This implies that
$\mu_{t-1}$ and $\mu_{t}$ have the same total mass. The order can also be characterized by the
potential functions
\[u_{\mu_t}:\R \to \R, \quad u_{\mu_t}(x):= \int |x-y| \, {\mu_t}(dy). \]
The following properties are elementary:
\begin{enumerate}[(i)]
\item $u_{\mu_t}$ is nonnegative and convex,
\item $\partial^+ u_{\mu_t}(x) - \partial^- u_{\mu_t}(x) = 2{\mu_t}(\{x\})$,
\item $\lim_{|x| \to \infty} u_{\mu_t}(x) = \infty \1_{{\mu_t} \neq 0}$,
\item $\lim_{|x| \to \infty} u_{\mu_t}(x) - {\mu_t}(\R)|x - \bary({\mu_t})| = 0$,
\end{enumerate}
where $\partial^+$ and  $\partial^-$ denote the right and left derivatives, respectively, and 
\MN{$\bary({\mu_t})=(\int x d \mu_t) / \mu_t(\R)$} is the barycenter.
We can therefore extend $u_{\mu_t}$ continuously to $\bar{\R} = [-\infty,\infty]$. \MN{The following
result of Strassen is classical (cf.~\cite{Strassen.65}; the last statement is obtained as e.g.\ in~\cite[Corollary~2.95]{FollmerSchied.11}).}

\begin{proposition}
\label{prop:StrassenResult}
Let $\bmu = (\mu_0,\dots,\mu_n)$ be finite measures on 
$\R$ with finite first moments and equal total mass.
The following are equivalent:
\begin{enumerate}[(i)]
\item $\mu_0 \leq_c \dots \leq_c \mu_n$,
\item $u_{\mu_0} \leq \dots \leq u_{\mu_n}$,
\item $\M(\bmu) \neq \emptyset$,
\item there exist stochastic kernels $\kappa_t(x_0,\dots,x_{t-1},dx_t)$ such that 
\[\int |x_t| \, \kappa_t(x_0,\dots,x_{t-1},dx_t) < \infty \text{ and } 
\int x_t \, \kappa_t(x_0,\dots,x_{t-1},dx_t) = x_{t-1}\] for all 
$(x_0,\dots,x_t) \in \R^t$ and $1\leq t \leq n$, and  
\[\mu_t = (\mu_0 \otimes \kappa_1 \otimes \dots \otimes \kappa_n) \circ (X_t)^{-1}\quad
\text{for all } 0 \leq t \leq n.\]
\end{enumerate}
\end{proposition}

\MN{All kernels will be stochastic (i.e.\ normalized) in what follows.} A kernel~$\kappa_{t}$ with the first property in~(iv) is called \emph{martingale kernel}.

\subsection{The One-Step Case}
\label{sse:onestep}

For the convenience of the reader, we summarize some results from \cite{BeiglbockJuillet.12} and~\cite{BeiglbockNutzTouzi.15} for the one-step problem ($n=1$) which will be used later on. In this section we write $(\mu,\nu)$ instead of $(\mu_{0},\mu_{1})$ for the given marginals in convex order.
 
\begin{definition}\label{de:componentSinglestep}
The pair $\mu \leq_c \nu$ is $\emph{irreducible}$ if the set $I = \{u_\mu < u_\nu\}$ is connected
and $\mu(I) = \mu(\R)$. In this situation, let $J$ be the union of $I$ and any endpoints of $I$ that
are atoms of $\nu$; then $(I,J)$ is the \emph{domain} of $\M(\mu,\nu)$.
\end{definition}

The first result is a decomposition of the transport problem into irreducible parts; cf.\ \cite[Theorem~8.4]{BeiglbockJuillet.12}.

\begin{proposition}
\label{prop:oneStepDecomposition}
Let $\mu \leq_c \nu$ and let $(I_k)_{1 \leq k \leq N}$ be the (open) components of
$\{u_\mu < u_\nu\}$, where $N \in \{0,1,\dots,\infty\}$. Set $I_0 = \R \backslash \cup_{k \geq 1} I_k$
and $\mu_k = \mu|_{I_k}$ for $k \geq 0$, so that $\mu = \sum_{k \geq 0} \mu_k$. Then, there
exists a unique decomposition $\nu = \sum_{k \geq 0} \nu_k$ such that
\[ \mu_0 = \nu_0 \quad \text{and} \quad \mu_k \leq_c \nu_k \quad \text{for all} \quad  k \geq 1,\]
and this decomposition satisfies $I_k = \{u_{\mu_k} < u_{\nu_k}\}$ for all $k \geq 1$. Moreover,
any $P \in \M(\mu,\nu)$ admits a unique decomposition $P = \sum_{k \geq 0} P_k$ such that
$P_k \in \M(\mu_k,\nu_k)$ for all $k \geq 0$.
\end{proposition}

We observe that the measure $P_0$ in Proposition \ref{prop:oneStepDecomposition} transports $\mu_0$ to itself and is concentrated on
$\Delta_0 := \Delta \cap I_0^2$ where $\Delta = \{(x,x) : x \in \R\}$ is the diagonal. Thus, the transport problem with index $k=0$ is not actually an irreducible one, but we shall nevertheless refer to $(I_0,I_0)$ as the domain of this problem. When we want to emphasize the distinction, we call $(I_0,I_0)$ the \emph{diagonal domain} and $(I_k,J_k)_{k\geq 1}$ the \emph{irreducible domains} of $\M(\mu,\nu)$.
Similarly, the sets $V_k := I_k \times J_k$, $k\geq1$ will be called the
\emph{irreducible components} and $V_0 := \Delta_0$ will be called
\MN{the} \emph{diagonal component} of $\M(\mu,\nu)$. This terminology refers to the following result of~\cite[Theorem~3.2]{BeiglbockNutzTouzi.15} which essentially states that the components are the only sets that can be charged by a martingale transport. We call a set $B\subseteq\R^{2}$ $\M(\mu,\nu)$-polar if it is $P$-null for all $P\in\M(\mu,\nu)$, \MN{where a nullset is, as usual, any set contained in a Borel set of zero measure.}

\begin{proposition}\label{prop:polarOneStep}
Let $\mu \leq_c \nu$ and let
 $B \subseteq \R^2$ be a Borel set. Then $B$ is $\M(\mu,\nu)$-polar if and only if there exist
a $\mu$-nullset $N_\mu$ and a $\nu$-nullset $N_\nu$ such that
\[B \subseteq (N_\mu \times \R) \cup (\R \times N_\nu) \cup \left(\bigcup_{k \geq 0} V_k \right)^c.\]
\end{proposition}

The following result of~\cite[Lemma~3.3]{BeiglbockNutzTouzi.15} will also be useful; it is the main ingredient in the proof of the preceding proposition.

\begin{lemma}\label{le:absContTransportOneStep}
Let $\mu \leq_c \nu$ be irreducible and let $\pi$ be a finite measure on~$\R^2$ whose marginals
$\pi_1,\pi_2$ satisfy\footnote{By $\pi_1 \leq \mu$ we mean that $\pi_1(A) \leq \mu(A)$  for every Borel set $A\subseteq \R$.} $\pi_1 \leq \mu$ and $\pi_2 \leq \nu$. Then, there exists
$P \in \M(\mu,\nu)$ such that $P$ dominates $\pi$ in the sense of absolute continuity.
\end{lemma}

\section{The Polar Structure}\label{se:polarStructure}

The goal of this section is to identify all obstructions to martingale transports imposed by the marginals $\bmu=(\mu_0,\dots,\mu_n)$, and thus, conversely, the sets that can indeed be charged. We recall that a subset $B$ of $\R^{n+1}$ is called $\M(\bmu)$-polar if it is a $P$-nullset for all $P\in\M(\bmu)$. The result for the one-step case in Proposition~\ref{prop:polarOneStep} already exhibits an obvious type of polar set $B \subseteq \R^{n+1}$:  if for some $t$ there is an $\M(\mu_{t-1},\mu_t)$-polar 
set $B' \subseteq \R^2$ such that $B \subseteq \R^{t-1} \times B' \times \R^{n-t}$, then $B$ must be $\M(\bmu)$-polar. The following shows that unions of such sets are in fact the only polar sets of $\M(\bmu)$.

\begin{theorem}[Polar Structure]
\label{thm:polarstruct}
Let $\bmu=(\mu_0,\dots,\mu_n)$ be in convex order. Then a Borel set $B \subseteq \R^{n+1}$ is
$\M(\bmu)$-polar if and only if there exist $\mu_t$-nullsets $N_{t}$ such that
\begin{equation} 
\label{eqn:polarsets}
B \;\;\subseteq \;\; \bigcup_{t=0}^n (X_t)^{-1}(N_{t}) \;\cup\; \bigcup_{t=1}^n (X_{t-1},X_t)^{-1}\left(\bigcup_{k \geq 0} V^t_k\right)^c
\end{equation}
where $(V_k^t)_{k\geq1}$ are the irreducible components of $\M(\mu_{t-1},\mu_t)$
and $V_0^t$ is the corresponding diagonal component.
\end{theorem}

Before stating the proof, we introduce some additional terminology. 
The second part of~\eqref{eqn:polarsets} can be expressed as
\begin{align}
\bigcup_{t=1}^n (X_{t-1},X_t)^{-1}\left(\bigcup_{k \geq 0} V_k^t \right)^c
&= \left(\bigcap_{t=1}^n \bigcup_{k \geq 0} (X_{t-1},X_t)^{-1}(V_k^t)\right)^c \nonumber\\ 
&= \left(\bigcup_{k_1,\dots,k_{n} \geq 0} \, \bigcap_{t=1}^n (X_{t-1},X_t)^{-1}(V_{k_t}^t)\right)^c. \label{eq:irredCompDeriv}
\end{align}
For every $\bk=(k_1,\dots,k_{n})$, the set
$$
   V_{\bk}=\bigcap_{t=1}^n (X_{t-1},X_t)^{-1}(V_{k_t}^t) \subseteq \R^{n+1}
$$
as occurring in the last expression of~\eqref{eq:irredCompDeriv} will be referred to as an \emph{irreducible component} of $\M(\bmu)$; these sets are disjoint since $V^{t}_{k}\cap V^{t}_{k'}=\emptyset$ for $k\neq k'$. Moreover, we call their union
$$
  \V= \cup_{\bk} V_{\bk}
$$
the \emph{effective domain} of $\M(\bmu)$.

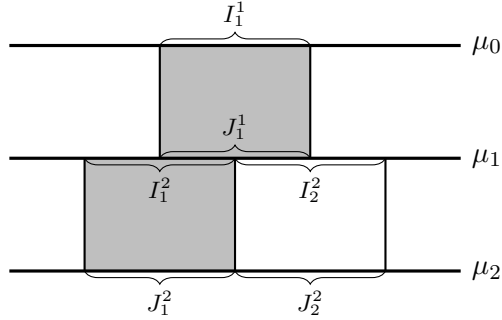
\begin{figure}
\begin{center}
\scalebox{1}{
\begin{tikzpicture}
\draw [fill=gray,gray,opacity=0.5] (-1,3) rectangle (1,1.5);
\draw [fill=gray,gray,opacity=0.5] (-2,1.5) rectangle (0,0);

\draw[very thick] (-3,3) -- (3,3) node[right] {$\mu_0$};
\draw[very thick] (-3,1.5) -- (3,1.5) node[right] {$\mu_1$};
\draw[very thick] (-3,0) -- (3,0) node[right] {$\mu_2$};

\draw[thick] (-1,3) -- (-1,1.5) {};
\draw[thick] (1,3) -- (1,1.5) {};
\draw[thick] (-2,1.5) -- (-2,0) {};
\draw[thick] (0,1.5) -- (0,0) {};
\draw[thick] (2,1.5) -- (2,0) {};

\draw [decorate,decoration={brace,amplitude=5pt}] 
(-1,3) -- (1,3)node [midway,yshift=12pt] {\footnotesize $I^1_1$};
\draw [decorate,decoration={brace,amplitude=5pt}] 
(-1,1.5) -- (1,1.5)node [midway,yshift=12pt] {\footnotesize $J^1_1$};

\draw [decorate,decoration={brace,mirror,amplitude=5pt}] 
(-2,1.5) -- (0,1.5)node [midway,yshift=-12pt] {\footnotesize $I^2_1$};
\draw [decorate,decoration={brace,mirror,amplitude=5pt}] 
(0,1.5) -- (2,1.5)node [midway,yshift=-12pt] {\footnotesize $I^2_2$};

\draw [decorate,decoration={brace,mirror,amplitude=5pt}] 
(-2,0) -- (0,0)node [midway,yshift=-12pt] {\footnotesize $J^2_1$};
\draw [decorate,decoration={brace,mirror,amplitude=5pt}] 
(0,0) -- (2,0)node [midway,yshift=-12pt] {\footnotesize $J^2_2$};
\end{tikzpicture}}
\end{center}
\caption{The shaded area represents $V_{\bk}$ for $\bk=(1,1)$.}
\label{fig:polarStructure}
\end{figure}

Roughly speaking, an irreducible component $V_{\bk}$ is a chain of irreducible components from the individual steps $(t-1,t)$.
In the one-step case considered in~\cite{BeiglbockJuillet.12,BeiglbockNutzTouzi.15}, it was possible and useful to decompose the transport problem into its irreducible components and study those separately to a large extent; cf.\ Proposition~\ref{prop:oneStepDecomposition}. This is impossible in the multistep case, as illustrated by the following example.

\begin{example}
\label{ex:polarStructure}
Consider the two-step martingale transport problem with marginals
$\mu_0 = \delta_0$, $\mu_1 = \frac{1}{2}(\delta_{-1}+\delta_1)$ and
$\mu_2 = \frac{1}{4}(\delta_{-2} + 2\delta_0 + \delta_2)$.
Then the irreducible components are given by
\begin{align*}
V_{00} &= \{(x,x,x) : x \notin (-2,2)\}\\
V_{01} &= \{(x,x) : x \in (-2,-1]\} \times [-2,0]\\
V_{02} &= \{(x,x) : x \in [1,2)\} \times [0,2]\\
V_{10} &= (-1,1) \times \{0\}\times \{0\}\\
V_{11} &= (-1,1) \times [-1,0) \times [-2,0]\\
V_{12} &= (-1,1) \times (0,1] \times [0,2].
\end{align*}
There is only one martingale transport $P\in\M(\bmu)$, given by
\[P = \frac{1}{4}(\delta_{(0,-1,-2)} + \delta_{(0,-1,0)} + \delta_{(0,1,0)} + \delta_{(0,1,2)}). \]
While $P$ is supported on $V_{11} \cup V_{12}$, it cannot be decomposed into two martingale parts that are supported on $V_{11}$ and $V_{12}$, respectively:
$V_{11}$ and $V_{12}$ are disjoint, but 
$P|_{V_{11}} = \frac{1}{4}(\delta_{(0,-1,-2)} + \delta_{(0,-1,0)})$
is not a martingale.
\end{example}

The main step in the proof of Theorem \ref{thm:polarstruct} will be the following lemma.

\begin{lemma}
\label{lem:polarstruct}
Let $V_{\bk}$ be an irreducible component of $\M(\bmu)$ and consider a measure $\pi$
concentrated on $V_{\bk}$ such that $\pi_t \leq \mu_t$ for $t = 0,\dots,n$. Then there exists a transport $P \in \M(\bmu)$ which dominates $\pi$ in the sense of absolute continuity.
\end{lemma}

Deferring the proof, we first show how this implies the theorem.

\begin{proof}[Proof of Theorem \ref{thm:polarstruct}]
Clearly $(X_t)^{-1}(N_{t})$ is $\M(\bmu)$-polar for $t=0,\dots,n$ and
$(X_{t-1},X_t)^{-1}\left(\bigcup_{k \geq 0} V_k^t\right)^c$ is $\M(\bmu)$-polar for
$t=1,\dots,n$. This shows that~\eqref{eqn:polarsets} is sufficient for $B \subseteq \R^{n+1}$
to be $\M(\bmu)$-polar.

\FS{Conversely, suppose that~\eqref{eqn:polarsets} does not hold; we show that $B$ is not
$\M(\bmu)$-polar. In view of~\eqref{eq:irredCompDeriv}, by passing to a subset of $B$ if necessary, we may assume that
\[B \subseteq \V = \bigcup_{\bk} V_{\bk} = \bigcup_{\bk} \bigcap_{t=1}^n(X_{t-1},X_t)^{-1}(V_{k_t}^t).\]

We may also assume that there are no $\mu_t$-nullsets $N_{t}$ such that
$B \subseteq \cup_{t=0}^{n}(X_t)^{-1}(N_{t})$. By a result of classical optimal transport \cite[Proposition~2.1]{BeiglbockGoldsternMareschSchachermayer.09}, this entails that $B$ is not $\Pi(\bmu)$-polar; i.e.\ we can find a measure $\rho \in \Pi(\bmu)$ such that $\rho(B) > 0$. 

We now write $B = \bigcup_{\bk} B \cap V_{\bk}$. As $\rho(B) = \sum_{\bk} \rho(B \cap V_{\bk}) > 0$, we can
find some $\bk$ such that $\rho(B \cap V_{\bk}) > 0$. But then $\pi := \rho|_{V_{\bk}}$ satisfies the assumptions of Lemma~\ref{lem:polarstruct} which yields $P \in \M(\bmu)$ such that $P \gg \pi$. In particular, $P(B) > 0$ and $B$ is not $\M(\bmu)$-polar.}
\end{proof}

\subsection{Proof of Lemma~\ref{lem:polarstruct}}

The reasoning for Lemma~\ref{lem:polarstruct} follows an induction on the number $n$ of time steps; its rigorous formulation requires a certain amount of control over subsequent steps of the transport problem. Thus, we first state a more quantitative version of (the core part of) the lemma that is tailored to the inductive argument.

\begin{definition}
Let $\bmu$ be in convex order and $\V$ the effective domain of $\M(\bmu)$. We say that
a finite measure $\pi$ has a \emph{compact support family} if there are 
disjoint compact product sets\footnote{By a compact product set we mean a set $K = A_0 \times \cdots \times A_n$ where
each $A_t \subseteq \R$ is compact.}
$K_1,\dots,K_m \subseteq \V$ with $\pi(\cup_i K_i) = \pi(\R^{n+1})$ such that
$K_i \subseteq V_{\bk_i}$ for some irreducible component $V_{\bk_i}$ for all $i=1,\dots,m$.
\end{definition}

\begin{definition}
\label{def:diagonallycompatible}
Let $\bmu$ be in convex order, $t\leq n$ and $\sigma \leq \mu_t$ a finite measure on~$\R$. 
If $t=n$, we 
say that $\sigma$ is \emph{diagonally compatible} (with $\bmu$) if there is a finite family of compact
sets $L_1,\dots,L_m \subseteq \R$ with $\sigma(\cup_i L_i) = \sigma(\R)$.
Whereas if $t < n$, we require in addition that for every $i$, 
either (a) $L_i \subseteq I_k$ for some irreducible component $(I_k,J_k)$ of $\M(\mu_t,\mu_{t+1})$ or (b)
$L_i \subseteq I_0$ and there is $t+1 \leq t' \leq n$ such that $L_i \subseteq I_0^s$ for the
diagonal components of $\M(\mu_s,\mu_{s+1})$ for all $t \leq s < t'$ and $L_i \subseteq I_k^{t'}$
for some (non-diagonal) irreducible component $(I_k^{t'},J_k^{t'})$ of $\M(\mu_{t'},\mu_{t'+1})$, where we set
$I_k^n = J_k^n= \R$ for notational convenience.
\end{definition}

\begin{lemma} \label{lem:diagonalcomponentcontinuation}
Let $t<n$ and let $L \subseteq I_0$ be a compact interval contained in the diagonal
component of $\M(\mu_t,\mu_{t+1})$ such that $\mu_t(L) > 0$. There exist a compact interval
$L' \subseteq L$ with $\mu_t(L') > 0$ and $t+1 \leq t' \leq n$ such that 
$L' \subseteq I_0^s$ for the diagonal component of $\M(\mu_s,\mu_{s+1})$ for all
$t \leq s < t'$ and $L' \subseteq I_k^{t'}$ for some (non-diagonal) irreducible
component $(I_k^{t'},J_k^{t'})$ of $\M(\mu_{t'},\mu_{t'+1})$, where we again set
$I_k^n = J_k^n = \R$ for notational convenience.
\end{lemma}

\begin{proof}
The statement is trivially satisfied for $t = n-1$ as we can just take $L' = L$. For $t < n-1$,
consider the family of irreducible components $(I_k^{t+1},J_k^{t+1})$ of $\M(\mu_{t+1},\mu_{t+2})$.
We distinguish three cases.

\FS{(i) First, consider the case where $L \cap I_k^{t+1} = \emptyset$ for all $k \geq 1$, then $L$ is contained in 
the diagonal component of $\M(\mu_{t+1},\mu_{t+2})$.

(ia) If $L = \{x\}$ consists of a single point with positive mass, then we can conclude by induction
from the result for $t+1$.

(ib) If no endpoint of $L$ is on the boundary of some component $I_k^t$, then observe that $\mu_t|_L = \mu_{t+1}|_L$.
We can find $L' \subseteq L$ from the statement of the lemma for $t+1$. Then $L'$ gives the result
as $\mu_t(L') = \mu_{t+1}(L') > 0$.

(ic) If $L$ contains more than one point, and also the endpoint of some component $I_k^t$. When
this endpoint $x$ has positive point mass, we can set $L'=\{x\}$ and conclude as in (ia).
If the endpoint has zero mass, we can find $\bar{L} \subseteq L$ compact with $\mu_t(\bar{L}) > 0$
that does not contain this endpoint and argue as in (ib). (Observe that there might be at most two
endpoints.)}

(ii) Next, let $k \geq 1$ be such that $\mu_{t+1}(L \cap I_k^{t+1}) > 0$ (and in particular
$L \cap I_k^{t+1} \neq \emptyset$). Then we can find a compact interval $L' \subseteq L \cap I_k^{t+1}$
such that \FS{$\mu_{t}(L') > 0$} and we directly see that $L'$ satisfies the statement of the lemma.

\FS{(iii) Finally, suppose that there is $k \geq 1$ with $L \cap I_k^{t+1}\neq \emptyset$ but $\mu_{t}(L \cap I_k^{t+1})=0$.
In particular this means that $L \not\subseteq I_k^{t+1}$. It furthermore means that $I_k^{t+1} \not\subseteq L$, as
otherwise $\mu_{t+1}(I_k^{t+1}) = \mu_t(I_k^{t+1}) = \mu_t(L \cap I_k^{t+1}) = 0$ which contradicts the definition
of $I_k^{t+1}$.}
As $L$ is a compact interval and $I_k^{t+1}$ is an open
interval, we have that $L \backslash I_k^{t+1}$ is a compact interval and
$\mu_t(L \backslash I_k^{t+1}) = \mu_t(L) > 0$. Notice that there can be at most two such components
$I_k^{t+1}$ for fixed $L$ and we will be in case~(i) after removing both of them if necessary.
\end{proof}

\begin{lemma}\label{lem:compatiblerestriction}
Let $t\leq n$ and let  $J \subseteq \R$ be an interval such that $\mu_t(J) > 0$. Then we can find a compact interval 
$K \subseteq J$ with $\mu_t(K) > 0$ such that $\mu_t|_K$ is diagonally compatible.
\end{lemma}

\begin{proof}
The case $t=n$ is trivial. Thus, let $t<n$. We consider the family $\{I_k\}_{k \geq 1}$ of open sets corresponding to the irreducible components
of $\M(\mu_t,\mu_{t+1})$ and distinguish two cases.

(i) There is some $k \geq 1$ such that $\mu_t(I_k \cap J) > 0$. In this case, we can choose
a compact interval $K \subseteq I_k \cap J$ such that $\mu_t(K) > 0$.

(ii) Now suppose that $\mu_t(I_k \cap J) = 0$ for all $k \geq 1$. Then we first notice that there
are at most two components $I_{k_1}$,$I_{k_2}$ so that $I_{k_i} \cap J \neq \emptyset$ and $J \backslash (I_{k_1} \cup I_{k_2})$
is still a nonempty interval with positive $\mu_t$-mass, since $I_k$ cannot be contained in $J$. We can therefore assume without loss of generality that
$J \subseteq I_0$ and is compact. Now we can apply Lemma \ref{lem:diagonalcomponentcontinuation} to find
a subinterval $K \subseteq J$ such that $\mu_t|_K$ is diagonally compatible.
\end{proof}

\begin{lemma}\label{lem:specialpolar}
Let $t \leq n$ and let $\pi$ be a measure on 
$\R^{t+1}$  that has a compact support family with respect to
$\mu_0,\dots,\mu_t$ and satisfies $\pi_s \leq \mu_s$ for $s \leq t$. In addition, suppose that $\pi_t$ is diagonally compatible.

Then there is a martingale measure $Q$ on $\R^{t+1}$ that dominates $\pi$ in the sense of absolute continuity and has a compact support
family with respect to $\mu_0,\dots,\mu_t$ and satisfies $Q_s \leq \mu_s$ for $s \leq t$. In addition, $Q_t$ can be chosen to be diagonally compatible. Finally, $Q$ can be chosen such that $dQ= g d\pi + d\sigma$ where the density $g$ is bounded and the measure $\sigma$ is singular with respect to~$\pi$.
\end{lemma}

\begin{proof}
We proceed by induction on $t$. For $t=0$ 
there is nothing to prove; we can set $Q = \pi$. 

Consider $t \geq 1$ and assume that the lemma has already been shown for $(t-1)$-step measures. 
We disintegrate 
\begin{equation}\label{eq:piDecomp}
  \pi = \pi' \otimes \kappa(x_0,\dots,x_{t-1},dx_t)
\end{equation}
and observe that $\pi'$ satisfies the conditions of the lemma. In particular, $\pi'_{t-1}$ must be
diagonally compatible: the compact sets that it is supported on are either contained in
irreducible components of $\M(\mu_{t-1},\mu_t)$ or in the diagonal component.
Any such compact subset of the diagonal component of $\M(\mu_{t-1},\mu_t)$ must
correspond to one of the finitely many compact sets in the support of 
$\pi_t$ so that they inherit the compatibility property from these sets.

By the induction assumption, we then find a martingale measure $Q'\gg \pi'$ on $\R^t$ with the stated properties. In particular, the marginal $Q'_{t-1}$ is diagonally compatible with $\bmu$.

Again, let $\{I_k\}_{k \geq 1}$ be the open intervals from the irreducible domains $(I_k,J_k)$ of $\M(\mu_{t-1},\mu_t)$ and let $I_0$ denote the corresponding diagonal domain. 
We shall construct a martingale kernel $\hat{\kappa}$ by suitably manipulating $\kappa$. Let us observe that since $\pi$ is concentrated on $\V$ and has a compact support family with respect to $\mu_0,\dots,\mu_t$, the following hold for $\pi'$-a.e.\ $\bx=(x_0,\dots,x_{t-1})\in\R^{t}$
and a finite family of compact sets $L_i$ with properties (a) or (b) from Definition \ref{def:diagonallycompatible}:
\begin{itemize}
  \item $\kappa(\bx,\cdot) = \delta_{x_{t-1}}$ whenever $x_{t-1} \in I_0$,
  \item $\kappa(\bx,\cdot)$ is concentrated on some $L_i$ with $L_i \subseteq J_k$ for $x_{t-1} \in I_k$ with $k\geq 1$ and  $Q'_{t-1}(I_k) > 0$.
\end{itemize}
By changing $\kappa$ on a $\pi'$-nullset, we may assume that these two properties 
hold for all $\bx\in\R^{t}$.

\emph{Step~1.} Next, we argue that we may change $Q'$ and $\kappa$ such that the marginal $(Q' \otimes \kappa)_t= (Q' \otimes \kappa)\circ X_{t}^{-1}$ satisfies
\begin{equation}\label{eq:kappaMarginalWlog}
  (Q' \otimes \kappa)_t \leq \mu_t.
\end{equation}
Indeed, recall that $dQ'= dQ'_{abs}+ d\sigma' = g' d\pi' + d\sigma'$ where the density $g'$ is bounded and $\sigma'$ is singular with respect to $\pi'$. Using the Lebesgue decomposition theorem, we find a Borel set $A \subseteq \R^t$ such that $\sigma'(A) = \sigma'(\R^t)$ and $\pi'(A) = 0$.
By scaling $Q'$ with a constant we may assume that $g'\leq 1/2$. As $\pi_{t}\leq\mu_{t}$, the marginal $(Q'_{abs}\otimes \kappa)_{t}$ is then bounded by $\frac12 \mu_{t}$, and it remains to bound $(\sigma'\otimes \kappa)_{t}$ in the same way.

Note that $Q'_{t-1}\leq\mu_{t-1}$ implies $\sigma'_{t-1}\leq\mu_{t-1}$. We may change $\kappa$ arbitrarily on the set $A$ without invalidating~\eqref{eq:piDecomp}. Indeed, for each irreducible component $(I_k,J_k)$ of $\M(\mu_{t-1},\mu_t)$ we choose and fix a compact interval
$K_k \subseteq J_k$ with $\mu_t(K_k)>0$ such that $\mu_t|_{K_k}$ is diagonally compatible; this is possible by Lemma \ref{lem:compatiblerestriction}.
For
$\bx=(x_0,\dots,x_{t-1}) \in A$ such that $x_{t-1} \in I_k$ we then define
\[ \kappa(\bx,\cdot) := \frac{1}{\mu_t(K_k)} \mu_t|_{K_k}.\]
Set $\epsilon_{k} = \mu_t(K_k)/ \mu_{t-1}(I_{k})$. 
Then
\[\epsilon :=  \inf_{k:\, Q'_{t-1}(I_k) > 0} \epsilon_{k}\wedge 1\]
is strictly positive because there are only finitely many $k$ with $Q'_{t-1}(I_k) > 0$ (this is the purpose of the induction assumption that $Q'_{t-1}$ is diagonally compatible). As $\sigma'_{t-1}\leq\mu_{t-1}$, we may scale $Q'$ once again to obtain  $\sigma'_{t-1}\leq \frac{\epsilon}{6} \mu_{t-1}$.  
We now have
\[(\sigma'|_{\R^{t-1}\times I_k} \otimes \kappa)_t = \sigma'_{t-1}(I_k)
\frac{1}{\mu_t(K_k)} \mu_t|_{K_k} \leq \frac{\epsilon}{6} \frac{\mu_{t-1}(I_k)}{\mu_t(K_k)} \mu_t|_{K_k} \leq \frac{1}{6}\mu_t|_{K_k}.\]
For the diagonal domain $I_{0}$ the corresponding inequality holds because we have $\kappa(\bx,\cdot)=\delta_{x_{t-1}}$ for $x_{t-1}\in I_{0}$ and $\sigma'_{t-1}|_{I_{0}}\leq \frac{1}{6}\mu_{t-1}|_{I_{0}}\leq\frac{1}{6}\mu_{t}|_{I_{0}}$. As a consequence, we have $(\sigma'\otimes \kappa)_{t}\leq \frac12 \mu_{t}$ as desired, so that we may assume~\eqref{eq:kappaMarginalWlog} in what follows.

\emph{Step~2.} We now construct a martingale kernel $\hat{\kappa}$ such that $Q=Q'\otimes \hat{\kappa}$ has the required properties. 
For a fixed irreducible component $(I_k,J_k)$ we have that $Q'_{t-1}|_{I_k} = Q'_{t-1}|_K$ for some compact $K \subseteq I_k$.
We can find compact intervals $B^-,B^+ \subseteq J_k$ with $\mu_t(B^-) > 0$ and $\mu_t(B^+)>0$ such that $B^-$ is
to the left of $K$ and $B^+$ is to the right of $K$, in the sense that $x < y < z$ for $x \in B^-$, $y \in K$ and $z \in B^+$. By Lemma~\ref{lem:compatiblerestriction}, we can further assume that we have $B^+ \subseteq I_k^t$ and $B^- \subseteq I_{k'}^t$ for some $k,k' \geq 0$, where $(I_l^t)_{l\geq 0}$ belong to the components of $\M(\mu_t,\mu_{t+1})$, and that $\mu_t|_{B^\pm}$ is diagonally compatible.

Next, we define two nonnegative functions $\bx\mapsto \varepsilon^-(\bx), \varepsilon^+(\bx)$ for $\bx=(x_{0},\dots,x_{t-1})\in\R^{t-1} \times K$ as follows:
\begin{itemize}
  \item for $\bx$ such that $\bary (\kappa(\bx,\cdot))<x_{t-1}$, let $\varepsilon^+$ be the unique number such that $\kappa(\bx,\cdot) + \varepsilon^+(\bx)\cdot \mu_t|_{B^+}$ has barycenter $x_{t-1}$,
  \item  for $\bx$ such that $\bary (\kappa(\bx,\cdot))>x_{t-1}$, let $\varepsilon^-$ be the unique number
  such that $\kappa(\bx,\cdot) + \varepsilon^-(\bx)\cdot\mu_t|_{B^-}$ has barycenter $x_{t-1}$,
  \item $\varepsilon^\pm(\bx)=0$ otherwise.
\end{itemize}
Observe that these numbers always exist because $B^-$ and $B^+$ have positive mass and positive distance from the points $x_{t-1}\in K$. We now define the martingale kernel $\hat{\kappa}$ by
\[\hat{\kappa}(\bx) := c(\varepsilon^-\cdot \mu_t|_{B^-} + \kappa + \varepsilon^+\cdot\mu_t|_{B^+})\]
where $0<c\leq 1$ is a normalizing constant such that $\hat{\kappa}$ is again a stochastic kernel. We also define $\hat{\kappa}(\bx)=\kappa(\bx)$ for $\bx$ on the diagonal domain.

For each $k\geq 1$, let $B^\pm_k$ denote the sets associated with $I_k$ as above. Once again, the number
\[C := \frac{1}{3}\inf_{k:\, Q'_{t-1}(I_k) > 0} \,[\mu_t(B^-_k) \wedge \mu_t(B^+_k)]\]
is strictly positive because there are only finitely many $k$ with $Q'_{t-1}(I_k) > 0$. We can now define 
\[Q := C \cdot (Q' \otimes \hat{\kappa}).\]
Then $Q$ is a martingale transport whose marginals satisfy $Q_s \leq Q'_{s} \leq \mu_s$ for $0\leq s\leq t-1$ whereas $Q_t \leq \mu_t$ by~\eqref{eq:kappaMarginalWlog}, the construction of $\hat{\kappa}$ and the choice of $C$; indeed, for every $x_{t-1} \in I_k^t$ we have
\begin{align*}
3C\hat{\kappa}(\bx) &\leq 3C\varepsilon^-\cdot \mu_t|_{B^-} + 3C\kappa + 3C\varepsilon^+\cdot\mu_t|_{B^+} \\
&\leq \mu_t|_{B^-} + \kappa + \mu_t|_{B^+} \leq 2\mu_t + \kappa.
\end{align*}

To see that $Q_t$ is diagonally compatible, observe that $Q_t$
is supported by a finite family of compact sets consisting of the following:
\begin{itemize}
  \item a finite family of compact sets $\bar{L}_i \subseteq I_0$ such that $Q'_{t-1}|_{\bar{L}_i}$ is diagonally compatible (from the induction hypothesis that $Q'_{t-1}$ is diagonally compatible),
  \item a finite family of compact sets $L_i \subseteq J_k$ for some $k \geq 1$ with $Q_{t-1}'(I_k) > 0$ such that $Q_t|_{L_i} \leq \mu_t|_{L_i}$ is diagonally compatible, and
  \item the sets $B^\pm_k$ for the finitely many $k$ such that $Q'_{t-1}(I_k) > 0$, where
  $Q_t|_{B^\pm_k} \leq \mu_t|_{B^\pm_k}$ is diagonally compatible.
\end{itemize}

It remains to check that $Q$ has the required decomposition with respect to $\pi$. 
Indeed, $\hat{\kappa}$ can be decomposed as
\[\hat{\kappa} = c\kappa+(1-c)\kappa^{\bot}\]
where $\kappa^{\bot}$ is singular to $\kappa$.
Recalling the decomposition $Q'= Q'_{abs}+ \sigma'$, we then have
\[
  Q' \otimes \hat{\kappa} = cQ'_{abs}\otimes \kappa \;+\; (1-c)Q'_{abs}\otimes\kappa^{\bot} \;+\; \sigma'\otimes \hat{\kappa}.
\]
The last two terms are singular with respect to $\pi=\pi'\otimes \kappa$, and the first term is absolutely continuous with bounded density.
\end{proof}

\begin{proof}[Proof of Lemma \ref{lem:polarstruct}]
Let $\pi$ be a measure with marginals $\pi_t \leq \mu_t$ for all $t$ which is concentrated on some irreducible component $V=V_{\bk}$ and thus, in particular, on the effective domain $\V$.

\emph{Step~1.} We first decompose $\pi = \sum_{m=1}^\infty \pi^m$ such that each $\pi^m$ satisfies the requirements of Lemma~\ref{lem:specialpolar} with $t=n$. 

Indeed, let $V = \cap_{t=1}^n(X_{t-1},X_t)^{-1}(V_{k_t}^t)$ and
suppose first that $k_t \neq 0$ for $1 \leq t \leq n$. Then, we can write $V$ as a product of nonempty intervals: $V = A_0 \times \dots \times A_n$
where $A_0 = I^1_{k_1}$, $A_n = J^n_{k_n}$ and $A_t = J^t_{k_t} \cap I^{t+1}_{k_{t+1}}$
for $1 < t < n$. Thus, we can choose increasing families of compact intervals $K_t^m$ such that 
$A_t = \cup_{m\geq 1} K_t^m$ for all $t$. Setting 
$\pi^1 := \pi|_{\prod_{t=0}^n K_t^1}$ and 
$\pi^m := \pi|_{\prod_{t=0}^{n} K_t^{m} \backslash \prod_{t=0}^{n} K_t^{m-1}}$
for $m > 1$ yields the required decomposition.

If $k_t = 0$ for one or more $1 \leq t \leq n$, we have
$V \subseteq A_0 \times \dots \times A_n$, where $A_t$ is defined as above when
 $k_t \neq0\neq k_{t+1}$ but we use $\R$ instead of $J^t_{k_t}$  when $k_t = 0$ and $\R$ instead of $I^{t+1}_{k_{t+1}}$ when $k_{t+1}=0$.
After these modifications, $\pi^m$ can be defined as above; recall that diagonal components are always closed.

\emph{Step~2.} For each of the measures $\pi^m$, Lemma~\ref{lem:specialpolar} yields a martingale measure $Q^m \gg \pi^m$ with the properties stated in the lemma. In particular, each $Q^{m}$ has a compact support family. We show below that there exist $P^m \in \M(\bmu)$ such that $P^m \gg Q^m$, and then $P := \sum 2^{-m} P^m$ satisfies $P \in \M(\bmu)$ and $P \gg \pi$ as desired.

To complete the proof, it remains to show that for fixed $m \geq 1$ there exist $0<\epsilon<1$ and $\bar{Q}^m \in \M(\bmu - \epsilon (Q^m_0,\dots,Q^m_n))$, as we may then conclude by setting $P^m := \epsilon Q^m + \bar{Q}^m \in \M(\bmu)$. By Proposition~\ref{prop:StrassenResult}, the set $\M(\bmu - \epsilon (Q^m_0,\dots,Q^m_n))$ is nonempty if the marginals are in convex order, or equivalently if the potential functions satisfy
\begin{equation}\label{eq:findEps}
u_{\mu_{t-1}} - \epsilon u_{Q^m_{t-1}} \leq u_{\mu_t} - \epsilon u_{Q^m_t}
\end{equation}
for $t =1,\dots,n$. Thus, it suffices to find $\epsilon>0$ with this property for fixed~$t$, and we have reduced to a question about a one-step martingale transport problem. Indeed, we have $u_{\mu_{t-1}}\leq u_{\mu_t}$ on $\R$. Since $Q^{m}$ has a compact support family and in particular is supported by $\V$, there is a finite collection of compact sets $K_{j}\subseteq\R$ such that each $K_{j}$ is contained in one of the intervals $I^{t-1}_{k_{j}}$ from the decomposition of $(\mu_{t-1},\mu_{t})$ into irreducible components, $Q^m$ transports mass from $K_{j}$ to itself for each $j$, and $Q^m$ is the identical Monge transport on the complement $(\cup_{j}K_{j})^{c}$. On each $K_{j}$, 
Steps~(a) and~(b) in the proof of \cite[Lemma~3.3]{BeiglbockNutzTouzi.15} yield $\epsilon>0$ such that~\eqref{eq:findEps} holds on $K_{j}$, and we can choose $\epsilon>0$ independently of~$j$ since there are finitely many $j$. On the other hand, \eqref{eq:findEps} trivially holds on $(\cup_{j}K_{j})^{c}$ since $u_{Q^m_{t-1}} = u_{Q^m_t}$ on that set. This completes the proof.
%
%
%
\end{proof}

\section{The Dual Space}\label{se:dualSpace}

In this section we introduce the domain of the dual optimization problem and show that it has a certain closedness property. The latter will be crucial for the duality theorem in the subsequent section.

We shall need a generalized notion of integrability for the elements of the dual space.
To this end, we first recall the integral for concave functions as detailed in \cite[Section~4.1]{BeiglbockNutzTouzi.15}.

\begin{definition}
\label{def:concaveintegral}
Let $\mu \leq_c \nu$ be irreducible with domain $(I,J)$ and let $\chi: J \to \R$ be a concave function. We define
\[(\mu-\nu)(\chi) := \frac{1}{2}\int_I (u_\mu - u_\nu)d\chi'' + \int_{J \backslash I} |\Delta \chi| d\nu
\in [0,\infty] \]
where $-\chi''$ is the (locally finite) second derivative measure of $-\chi$ on $I$ and $|\Delta \chi|$ is the
absolute magnitude of the jumps of $\chi$ at the boundary points~$J \backslash I$.
\end{definition}

\begin{remark}
\label{rem:concaveKernelIntegral}
As shown in \cite[Lemma~4.1]{BeiglbockNutzTouzi.15}, this integral is well-defined and satisfies
\[(\mu-\nu)(\chi) = \int_I \left[ \chi(x) - \int_J \chi(y) \, \kappa(x,dy)\right] \mu(dx)\]
for any $P = \mu \otimes \kappa \in \M(\mu,\nu)$. Moreover, it 
coincides with the difference $\mu(\chi)-\nu(\chi)$ of the usual integrals when 
$\chi \in L^1(\mu) \cap L^1(\nu)$.
\end{remark}

For later reference, we record two more properties of the integral.

\begin{lemma}
\label{lem:concaveRegularities}
Let $\mu \leq_c \nu$ be irreducible with domain $(I,J)$ and let $\chi : J \to \R$ be concave.
\begin{enumerate}[(i)]
\item Assume that $I$ has a finite right endpoint $r$ and $\chi(a) = \chi'(a) = 0$ for some $a \in I$.
Then $\chi \leq 0$ and $\chi \1_{[a,\infty)}$ is concave. If $\nu$ has an atom at~$r$, then
\[ \chi(r) \geq - \frac{C}{\nu(\{r\})} (\mu - \nu) (\chi \1_{[a,\infty)})\]
for a constant $C \geq 0$ depending only on $\mu,\nu$.
\item For $a,b \in \R$, the concave function $\bar{\chi}(x) := \chi(x) + ax + b$ satisfies
\[(\mu - \nu)(\bar{\chi}) = (\mu-\nu)(\chi).\]
\end{enumerate}
\end{lemma}

\begin{proof}
The first part is \cite[Remark 4.6]{BeiglbockNutzTouzi.15} and the second part follows directly from $\bar{\chi}'' = \chi''$ and $\Delta\bar{\chi} = \Delta\chi$.
\end{proof}

Let us now return to the multistep case with a vector $\bmu=(\mu_{0},\dots,\mu_{n})$ of measures in convex order and introduce $\bmu(\bphi) := \sum_{t=0}^n \mu_t(\phi_t)$ in cases where
we do not necessarily have $\phi_t \in L^1(\mu_t)$. As mentioned previously, \MN{in contrast to~\cite{BeiglbockNutzTouzi.15},} the multistep transport problem does not decompose into irreducible components, forcing us to directly give a global definition of the integral.

\begin{definition}\label{de:moderator}
Let $\bphi = (\phi_0,\dots,\phi_n)$ be a vector of Borel functions $\phi_t : \R \to \RE$.
A vector $\bchi = (\chi_1,\dots,\chi_n)$ of Borel functions $\chi_t:\R \to \R$ is 
called a \emph{concave moderator} for $\bphi$ if for $1\leq t\leq n$,
\begin{enumerate}[(i)]
\item $\chi_t|_J$ is concave for every domain $(I,J)$ of an irreducible component 
of $\M(\mu_{t-1},\mu_t)$,
\item $\chi_t|_{I_0} \equiv 0$ for the diagonal domain $I_0$ of $\M(\mu_{t-1},\mu_t)$,
\item $\phi_t - \chi_{t+1} + \chi_t \in L^1(\mu_t)$,
\end{enumerate}
where $\chi_{n+1} \equiv 0$. We also convene that $\chi_0 \equiv 0$.
The \emph{moderated integral} of $\bphi$ is then defined by
\begin{equation}
\label{eqn:moderatedIntegral}
\bmu(\bphi) := \sum_{t=0}^n \mu_t(\phi_t - \chi_{t+1} + \chi_t) + \sum_{t=1}^n \sum_{k \geq 1} (\mu_{t-1}-\mu_t)^k(\chi_t) \in (-\infty,\infty],
\end{equation}
where $(\mu_{t-1}-\mu_t)^k(\chi_t)$ denotes the integral of Definition~\ref{def:concaveintegral} on the $k$-th irreducible component of $\M(\mu_{t-1},\mu_t)$.
\end{definition}

\begin{remark}\label{rk:indepChi}
The moderated integral is independent of the choice of the moderator $\bchi$.
To see this, consider a second moderator \MN{$\tilde{\bchi}$} for $\bphi$;
then we have $(\tilde\chi_{t+1} - \chi_{t+1}) - (\tilde\chi_t - \chi_t) \in L^1(\mu_t)$. \MN{We may assume that~\eqref{eqn:moderatedIntegral} is finite for at least one of the moderators.} Using Remark~\ref{rem:concaveKernelIntegral} with arbitrary $\kappa_t$ such
that $\mu_{t-1} \otimes \kappa_t \in \M(\mu_{t-1},\mu_t)$ for $1 \leq t \leq n$, as well as Fubini's theorem for kernels,
\begin{align*}
\sum_{t=1}^n &\sum_{k \geq 1} (\mu_{t-1}-\mu_t)^k(\chi_t) - (\mu_{t-1}-\mu_t)^k(\tilde\chi_t) \\
&= \int \cdots \int \sum_{t=1}^n \chi_t(x_{t-1}) - \chi_t(x_t) \, 
\kappa_n(x_{n-1},dx_n) \cdots\kappa_1(x_0,dx_1) \mu_0(dx_0) \\
&\qquad - \int \cdots \int \sum_{t=1}^n \tilde\chi_t(x_{t-1}) - \tilde\chi_t(x_t) \, 
\kappa_n(x_{n-1},dx_n) \cdots\kappa_1(x_0,dx_1) \mu_0(dx_0) \\
&= \FS{\sum_{t=0}^n \mu_t((\chi_{t+1} - \tilde\chi_{t+1}) - (\chi_t - \tilde\chi_t)).}
\end{align*}
It now follows that~\eqref{eqn:moderatedIntegral}
yields the same value for both moderators.
\end{remark}

For later reference, we also record the following property.

\begin{remark}\label{rk:moderatorRestrictionSum}
If $\bchi$ is a concave moderator, Definition~\ref{de:moderator}\,(ii) implies that
\[\chi_t = \sum_{k \geq 1} \chi_t|_{I_k^t} = \sum_{k \geq 1} \chi_t|_{J_k^t}\]
where $(I_k^t,J_k^t)$ is
the $k$-th irreducible domain of $\M(\mu_{t-1},\mu_t)$.
\end{remark}

Next, we introduce the space of functions which have a finite integral in the moderated sense.

\begin{definition}
We denote by $L^c(\bmu)$ the space of all vectors $\bphi$ admitting a concave moderator $\bchi$ with
$\sum_{t=1}^n \sum_{k \geq 1} (\mu_{t-1}-\mu_t)^k(\chi_t) < \infty$.
\end{definition}

It follows that $\bmu(\bphi)$ is finite for $\bphi \in L^c(\bmu)$, and we have
$\bmu(\phi) = \sum_t \mu_t(\phi_t)$ for $\bphi \in \Pi_{t=0}^n L^1(\mu_t)$. The definition is also consistent with the expectation under martingale transports, in the following sense.

\begin{lemma}
\label{lem:dualPayoff}
Let $\bphi \in L^c(\bmu)$ and let $H=(H_1,\dots,H_n)$ be $\FF$-predictable. If
\[\sum_{t=0}^n \phi_t(X_t) + (H \cdot X)_n   \]
is bounded from below on the effective domain $\V$ of $\M(\bmu)$, then
$$
\bmu(\bphi) = P\left[\sum_{t=0}^n \phi_t(X_t) + (H \cdot X)_n\right],\quad P \in \M(\bmu).
$$
\end{lemma}

\begin{proof}
Let $P \in \M(\bmu)$, let $\bchi$ be a concave moderator for $\bphi$, and assume without loss of generality that 
$0$ is the lower bound. Using Remark~\ref{rk:moderatorRestrictionSum}, we have that $\sum_{t=0}^n \phi_t(X_t) + (H \cdot X)_n$ equals 
\[ \sum_{t=0}^n (\phi_t - \chi_{t+1} + \chi_t)(X_t) + 
\sum_{t=1}^n\sum_{k \geq 1} (\chi_t|_{I_k^t}(X_{t-1}) - \chi_t|_{J_k^t}(X_t))+ (H \cdot X)_n \geq 0.\]
By assumption, the functions
$(\phi_t - \chi_{t+1} + \chi_t)(X_t)$ are $P$-integrable. Therefore, the negative part
of the remaining expression must also be $P$-integrable.
Writing $P_t := P\circ (X_0,\dots,X_t)^{-1}$ and using that $(\chi_t|_{J_k^t})^+$ has linear growth, we see that for any disintegration
$P = P_{n-1} \otimes \kappa_n$,
\begin{align*}
\int & \Bigg[\sum_{t=1}^n\sum_{k \geq 1} (\chi_t|_{I_k^t}(X_{t-1}) - \chi_t|_{J_k^t}(X_t))+ (H \cdot X)_n\Bigg] \,
\kappa_n(X_0,\dots,X_{n-1},dX_n) \\
&= \sum_{t=1}^{n-1}\sum_{k \geq 1} (\chi_t|_{I_k^t}(X_{t-1}) - \chi_t|_{J_k^t}(X_t)) + (H \cdot X)_{n-1}\\
& \qquad + \sum_{k\geq 1} \int \left[\chi_n|_{I_k^n}(X_{n-1}) - \chi_n|_{J_k^n}(X_n)\right] \kappa_n(X_0,\dots,X_{n-1},dX_n).
\end{align*}
Iteratively integrating with kernels such that $P_t = P_{t-1} \otimes \kappa_t$ and observing that
we can apply Fubini's theorem to $\sum_{t=1}^n\sum_{k \geq 1} (\chi_t|_{I_k^t}(X_{t-1}) - \chi_t|_{J_k^t}(X_t))+ (H \cdot X)_n$ \MN{as its negative part is $P$-integrable,} we obtain

\[P\left[\sum_{t=1}^n\sum_{k \geq 1} (\chi_t|_{I_k^t}(X_{t-1}) - \chi_t|_{J_k^t}(X_t))+ (H \cdot X)_n\right]
= \sum_{t=1}^n \sum_{k \geq 1} (\mu_{t-1}-\mu_t)^k(\chi_t)\]
and the result follows.
\end{proof}

We can now define our dual space. It will be convenient to work with nonnegative
reward functions $f$ for the moment---we shall relax this constraint later on; cf.\ Remark~\ref{rk:relaxLowerBound}.

\begin{definition}
Let $f: \R^{n+1} \to [0,\infty]$. We denote by $\D_\bmu(f)$ the set of all pairs
$(\bphi,H)$ where $\bphi \in L^c(\bmu)$ and $H=(H_{1},\dots,H_{n})$ is an $\FF$-predictable process such that
\[\sum_{t=0}^n \phi_t(X_t) + (H \cdot X)_n \geq f \quad \text{on}\quad \V. \]
\end{definition}

By Lemma \ref{lem:dualPayoff}, the expectation of the left hand side under any $P\in \M(\bmu)$ is given by the moderated integral $\bmu(\bphi)$; this will be seen as the dual cost of $(\bphi,H)$ when we consider the dual problem $\inf_{(\bphi,H) \in \D_\bmu(f)} \bmu(\bphi)$ in Section~\ref{se:duality} below.

The following closedness property is the key result about the dual space.

\begin{proposition}
\label{prop:dualclosedregular}
Let $f^m: \R^{n+1} \to [0, \infty]$, $m\geq1$  be a sequence of functions such that 
\[f^m \to f \qquad \text{pointwise}\] 
and let $(\bphi^m, H^m) \in \D_\bmu(f^m)$ be such that $\sup_m \bmu(\bphi^m) < \infty$. Then there exist
$(\bphi,H) \in \D_\bmu(f)$ with
\[\bmu(\bphi) \leq \liminf_{m \to \infty} \bmu(\bphi^m).\]
\end{proposition}

\subsection{Proof of Proposition~\ref{prop:dualclosedregular}}

An attempt to prove Proposition~\ref{prop:dualclosedregular} directly along the lines of~\cite{BeiglbockNutzTouzi.15} runs into a technical issue in controlling the concave moderators. Roughly speaking, they do not allow sufficiently many normalizations; this is related to the aforementioned fact that the multistep problem cannot be decomposed into its components. We shall introduce a generalized dual space with families of functions indexed by the components, and prove a ``lifted'' version of Proposition~\ref{prop:dualclosedregular} in this larger space. Once that is achieved, we can infer the closedness result in the original space as well. (The reader willing to admit Proposition~\ref{prop:dualclosedregular} may skip this subsection without \MN{much} loss of continuity.)

\begin{definition}
\label{def:generalHedge}
Let $\bphi=\{\phi_t^k:\, 0\leq t\leq n,\, k\geq0\}$ be a family of Borel functions, consisting of one function $\phi_t^k: J_k^t \to \RE$ for each irreducible component
$(I_k^t,J_k^t)$ of $\M(\mu_{t-1},\mu_t)$ as indexed by $k\geq1$ and $1\leq t\leq n$, functions $\phi_t^0 : I_0^t \to \RE$ for 
the diagonal components $I_0^t$ indexed by $1\leq t\leq n$ , and a single function
$\phi_0^0 : \R \to \RE$ for $t=0$.
Similarly, let $\bchi=\{\chi_t^k:\, 1\leq t\leq n,\, k\geq0\}$ be a family of functions, consisting of one concave function $\chi_t^k : J_k^t \to \R$ for each irreducible component $(I_k^t,J_k^t)$ 
and Borel functions $\chi_t^0 : I_0^t \to \R$ for the diagonal components. We also convene that $\chi_0^0 \equiv 0$
and define the functions\footnote{The restriction to $I_k^{t}$ is important to avoid ``double counting'' in the sums. Note that the intervals $J$ may overlap at their endpoints.} $\chi_t := \sum_{k \geq 0}\chi_t^k|_{I_k^t}$  for $t = 1,\dots,n$, as well as
$\chi_{n+1} \equiv 0$.

We call $\bchi$ a \emph{concave moderator} for $\bphi$ if for all $t=0,\dots,n$ and $k \geq 0$,
\[\phi_t^k + \chi_t^k - \chi_{t+1} \in L^1(\mu_t^k)\]
and the sum $\sum_{k \geq 0} \mu_t^k(\phi_t^k + \chi_t^k - \chi_{t+1})$ converges in $(-\infty,\infty]$, where $\mu_t^k$ is the second marginal of the $k$-th 
irreducible component in the decomposition of $\M(\mu_{t-1},\mu_t)$ as in Proposition~\ref{prop:oneStepDecomposition} and $\mu_0^0 \equiv \mu_0$.
The generalized\footnote{\MNN{This integral is not related to the notion of a generalized martingale.}} moderated integral is then defined by 
\[\bmu(\bphi) := \sum_{t = 0}^n \sum_{k \geq 0} \mu_t^k(\phi_t^k + \chi_t^k - \chi_{t+1}) +  \sum_{t=1}^n \sum_{k \geq 1} (\mu_{t-1}-\mu_t)^k(\chi_t^k).\]
We denote by $L^{c,g}(\bmu)$ the set of all families $\bphi$ which admit a concave moderator $\bchi$ such that
\[\sum_{t = 0}^n \sum_{k \geq 0} |\mu_t^k(\phi_t^k + \chi_t^k - \chi_{t+1})| + \sum_{t=1}^n \sum_{k \geq 1} (\mu_{t-1}-\mu_t)^k(\chi_t^k) < \infty.\]
\end{definition}

For $\bphi\in L^{c,g}(\bmu)$, the value of $\bmu(\bphi)$ is independent of the choice of the moderator $\bchi$. This is shown similarly as in Remark~\ref{rk:indepChi}. We can now introduce the generalized dual space.

\begin{definition}
Let $f: \R^{n+1} \to [0,\infty]$. We denote by $\D^g_\bmu(f)$ the set of all pairs
$(\bphi,H)$ where $\bphi \in L^{c,g}(\bmu)$, $H=(H_{1},\dots,H_{n})$ is $\FF$-predictable, and
\[\sum_{t=0}^n \phi_t^{k_t}(x_t) + (H \cdot \bx)_n \geq f(\bx) \]
for all $\bx = (x_0,\dots,x_n)$ and $\bk = (k_0,\dots,k_n)$ such that $(x_{t-1},x_t) \in (I_{k_t}^t,J_{k_t}^t)$ for some (irreducible or diagonal) component%
\footnote{
Given an irreducible component $(I,J)$, the notation $(x,y) \in (I,J)$ means that $x \in I, y \in  J$, whereas for
a diagonal component $(I_0,I_0)$ it is to be understood as $x = y \in I_0$.
}  and $t = 1,\dots,n$.
\end{definition}

We observe that for any $\bx \in \V$ the corresponding $\bk = (k_0,\dots,k_n)$ is uniquely defined, where the index $k_0 \equiv 0$ exists purely for notational convenience.

For later reference, the following lemma elaborates on certain degrees of freedom in choosing elements of $\D^g_\bmu(f)$.

\begin{lemma}
\label{lem:modify}
Let $(\bphi,H)\in\D^g_\bmu(f)$ and let $\bchi$ be a corresponding concave moderator.
Let $1 \leq t \leq n$, let $(I_k^t,J_k^t)$ be the domain of an irreducible component of $\M(\mu_{t-1},\mu_t)$ and $c_1,c_2 \in \R$. Introduce new families $(\tilde \bphi,\tilde H)$ and $\tilde \bchi$ by either (i) or (ii):
\begin{enumerate}
\item Define
\begin{align*}
  \tilde{\phi}_t^k(y) &= \phi_t^k(y) - (c_1 y - c_2), \qquad \tilde \chi_t^k(y) =  \chi_t^k(y) + (c_1 y - c_2),\\
  \tilde{\phi}_{t-1}^{k'}(x) &= \phi_{t-1}^{k'}(x) + (c_1x - c_2)|_{I_k^t}, \qquad \tilde \chi_{t-1}^{k'} = \chi_{t-1}^{k'},\\
  \tilde{\phi}_s^{k'} &= \phi_s^{k'}, \qquad \tilde \chi_s^{k'} = \chi_s^{k'} \quad \mbox{for }s \notin \{t-1,t\}, \\
  \tilde{H}_t &= H_t + c_1|_{X_{t-1}^{-1}(I_k^t)}, \qquad  \tilde{H}_s = H_s\quad \mbox{for }s \neq t
\end{align*}
where $k'$ runs over all components of the corresponding step in the subscript.
\item Define
\begin{align*} 
\tilde{\phi}_t^0 = \phi_t^0 + \chi_t^0 - \chi_{t+1}^0|_{I_0^t}, \qquad \tilde \chi_t^0 = 0, \quad \mbox{ and}\\
\tilde \phi_t^k = \phi_t^k - \chi_{t+1}^0,\qquad  \tilde\chi_t^k = \chi_t^k \quad \text{ for } k \geq 1,\,\; t=0,\dots,n.
\end{align*}
\end{enumerate}
Then $(\tilde \bphi,\tilde H)\in \D^g_\bmu(f)$ and $\tilde \bchi$ is
a corresponding concave moderator. Moreover, we have
\[ \sum_{t=0}^n \phi_t^{k_t}(x_t) + (H \cdot \bx)_n = \sum_{t=0}^n \tilde \phi_t^{k_t}(x_t) + (\tilde H \cdot \bx)_n \quad \text{and}\]
\[\phi_t^k + \chi_t^k - \chi_{t+1}  = \tilde \phi_t^k + \tilde \chi_t^k - \tilde \chi_{t+1}  \quad \text{ for all } \;k \geq 1, \; t = 0,\dots,n,\]
as well as $\bmu(\bphi) = \bmu(\tilde \bphi)$.
\end{lemma}

\begin{proof}
(i) If $\bx$ is such that $(x_{t-1},x_t) \notin I_k^t \times J_k^t$, then
$\tilde\phi_t^{k_t}(x_t) = \phi_t^{k_t}(x_t)$ for $t = 0,\dots,n$ and $\tilde H(\bx) = H(\bx)$. 
Otherwise,
\begin{align*}
\tilde \phi_t^{k_t}(x_t) + \tilde \phi_{t-1}^{k_{t-1}}(x_{t-1}) + \tilde H_t(x_t - x_{t-1}) 
&= \phi_t^{k_t}(x_t) + \phi_{t-1}^{k_{t-1}}(x_{t-1}) \\&\quad\,+ H_t(x_t - x_{t-1}), \\
\tilde \phi_t^k + \tilde \chi_t^k - \tilde \chi_{t+1} &= \phi_t^k + \chi_t^k - \chi_{t+1}, \text{ and}\\ 
\tilde \phi_{t-1}^{k'} + \tilde \chi_{t-1}^{k'} - \tilde \chi_{t} &= \phi_{t-1}^{k'} + \chi_{t-1}^{k'} - \chi_{t}.
\end{align*}
Along with the fact that $(\mu_t - \mu_{t-1})^k(\chi_t^k)=(\mu_t - \mu_{t-1})^k(\tilde \chi_t^k)$, these identities imply the assertions.

(ii) Similarly as in (i), the terms in question coincide by construction.
\end{proof}

\begin{remark}
\label{rem:manymodify}
The modification of Lemma~\ref{lem:modify}\,(i) can be applied simultaneously for infinitely
many $k$'s without difficulties. In this case we set
\[\tilde{\phi}_{t-1}^{k'}(x) := \phi_{t-1}^{k'}(x) + \sum_{k\geq 1} (c_1^kx - c_2^k)|_{I_k^t},\]
as well as $\tilde \phi_t^k(y) = \phi_t^k(y) - (c_1^ky-c_2^k)$ and
$\tilde \chi_t^k(y) =  \chi_t^k(y) + (c_1^k y - c_2^k)$ for the components $k \geq 1$ in step $t$.
The pointwise equalities still hold as above and in particular, the
moderated integral does not change.
\end{remark}

\begin{remark}
\label{rem:liftdual}
Any $(\bphi,H) \in \D_\bmu(f)$ induces an element $(\bphi^g,H) \in \D_\bmu^g(f)$ with $\bmu(\bphi^{g})=\bmu(\bphi)$ by choosing some concave moderator $\bchi$ for $\bphi$ and setting 
\[\phi_t^k := \phi_t|_{J_k^t}, \qquad \chi_t^k := \chi_t|_{J_k^t}.\]
\end{remark}

We now show the analogue to Lemma \ref{lem:dualPayoff} for the generalized dual space.

\begin{lemma}
\label{lem:dualPayoffGeneral}
Let $\bphi \in L^{c,g}(\bmu)$ and let $H = (H_1,\dots,H_n)$ be $\FF$-predictable.~If
\[\sum_{t=0}^n \phi_t^{\bk_t(\bx)}(x_t) + (H\cdot\bx)_n\]
is bounded from below on the effective domain $\V$ of $\M(\bmu)$, then
\[\bmu(\phi) = P\left[\sum_{t=0}^n \phi_t^{\bk_t(\bx)}(x_t) + (H\cdot\bx)_n\right], \quad P \in \M(\bmu).\]
\end{lemma}

\begin{proof}
Let $P \in \M(\bmu)$, let $\bchi$ be a concave moderator for $\bphi$ such that
$\chi_t^0 \equiv 0$ and assume that $0$ is the lower bound. It is easy to see that
$\sum_{t=0}^n \phi_t^{\bk_t(\bx)}(x_t) + (H\cdot\bx)_n$ equals
\[\sum_{t=0}^n(\phi_t^{\bk_t(\bx)} - \chi_{t+1} + \chi_t^{\bk_t(\bx)})(x_t) + \sum_{t=1}^n(\chi_t^{\bk_t(\bx)}(x_{t-1}) - \chi_t^{\bk_t(\bx)}(x_t)) + (H \cdot \bx)_n \geq 0.\]

By assumption $\sum_{t=0}^n(\phi_t^{\bk_t(\bx)} - \chi_{t+1} + \chi_t^{\bk_t(\bx)})(x_t)$ is $P$-integrable.
Therefore, the negative part of the remaining expression must also be $P$-integrable. Writing 
$P_t := P \circ (X_0,\dots,X_t)^{-1}$ and using that $(\chi_t^k)^+$ has linear growth, we see that
for any disintegration $P = P_{n-1} \otimes \kappa_n$,

\begin{align*}
\int &\left[\sum_{t=1}^n(\chi_t^{\bk_t(\bx)}(x_{t-1}) - \chi_t^{\bk_t(\bx)}(x_t)) + (H \cdot \bx)_n\right]\kappa_n(x_0,\dots,x_{n-1},dx_n) \\
&= \sum_{t=1}^{n-1}(\chi_t^{\bk_t(\bx)}(x_{t-1}) - \chi_t^{\bk_t(\bx)}(x_t)) + (H \cdot \bx)_{n-1} \\
&\qquad + \int \left[(\chi_n^{\bk_n(\bx)}(x_{n-1}) - \chi_n^{\bk_n(\bx)}(x_n)) \right] \kappa_n(x_0,\dots,x_{n-1},dx_n).
\end{align*}

Iteratively integrating with kernels such that $P_t = P_{t-1}\otimes \kappa_t$ and observing that
we can apply Fubini's theorem to $\sum_{t=1}^n(\chi_t^{\bk_t(\bx)}(x_{t-1}) - \chi_t^{\bk_t(\bx)}(x_t)) + (H \cdot \bx)_n$
as its negative part is $P$-integrable, we obtain
\[P\left[\sum_{t=1}^n(\chi_t^{\bk_t(\bx)}(x_{t-1}) - \chi_t^{\bk_t(\bx)}(x_t)) + (H \cdot \bx)_n\right] = \sum_{t=1}^n \sum_{k \geq 1} (\mu_{t-1}-\mu_t)^k(\chi_t^k)\]
and the result follows.
\end{proof}





Next, we establish that lifting from $\D_\bmu(f)$ to $\D^g_\bmu(f)$ does not change the range of dual costs.

\begin{proposition}\label{prop:dualEquivalence}
Let $f: \R^{n+1} \to [0,\infty]$. We have
\[\{\bmu(\bphi^g) : (\bphi^g,H) \in \D^g_\bmu(f) \} =\{\bmu(\bphi) : (\bphi,H) \in \D_\bmu(f) \}.  \]
\end{proposition}

\begin{proof}
\FS{Remark \ref{rem:liftdual} shows the inclusion ``$\supseteq$.'' To show the reverse, we may apply Lemma \ref{lem:modify} (i) together with Remark \ref{rem:manymodify} to modify a given pair $(\bphi^g,H) \in \D^g_\bmu(f)$
such that $\phi_t^k(x) = 0$ for $x \in J_k^t \backslash I_k^t$, for all irreducible domains $(I_k^t,J_k^t)$ of $\M(\mu_{t-1},\mu_t)$ and $1 \leq t \leq n$. 
Here we have used that $x \in J_k^t \backslash I_k^t$ implies $\mu_k^t(\{x\}) > 0$, cf.\ Definition~\ref{de:componentSinglestep}, and therefore
$\phi^g \in L^{c,g}(\bmu)$ implies $\phi_t^k(x) \in \R$; that is, such endpoints can indeed be shifted to~$0$ by
adding affine functions to $\phi_t^k$.

Let $\bchi^g$ be a concave moderator for $\bphi^g$. Using Lemma~\ref{lem:concaveRegularities}\,(ii)
and again Lemma \ref{lem:modify} as above, we can modify $\chi_t^k$ to satisfy 
$\chi_t^k(x) = 0$ for $x \in J_k^t \backslash I_k^t$,
for all irreducible domains $(I_k^t,J_k^t)$ of $\M(\mu_{t-1},\mu_t)$ and $1 \leq t \leq n$. Here, the finiteness of $\chi_t^k$ at the endpoints follows from  Lemma \ref{lem:concaveRegularities}\,(i) and $(\mu_{t-1}-\mu_t)^k(\chi_t^k) < \infty$.}

Still denoting the modified dual element by $(\bphi^g,H)$, we
define $\bphi \in L^c(\bmu)$ and a corresponding concave moderator $\bchi$ by 
\[\phi_t(x) := \phi_t^k(x), \quad \chi_t(x) := \chi_t^k(x), \quad \text{ for } x \in J_k^t; \]
they are well-defined since $\phi_t^k$ and $\chi_t^k$ vanish at points that belong to more than one set $J_k^t$.
We have $\bmu(\bphi) = \bmu(\bphi^g)$ by construction and the result follows.
\end{proof}

\begin{definition}
Let $1\leq t \leq n$ and $x_t \in \R$. A sequence $\bx = (x_0,\dots,x_t)$ is a \emph{predecessor path} of
$x_t$ if there are indices $(k_0,\dots,k_t)$ such that
$(x_{s-1},x_s) \in (I_{k_s}^s,J_{k_s}^s)$ for some component (irreducible or diagonal) of
$\M(\mu_{s-1},\mu_s)$, for all $1\leq s \leq t$.
We write $\bbk(\bx)$ for the (unique) associated
sequence $(k_0,\dots,k_t)$ followed by the path $\bx$ in the above sense, and 
$\Psi^k_t(x_t)$ for the set of all predecessor paths with $k_t = k$.
\end{definition}

These notions will be useful in the next step towards the closedness result, which is to ``regularize'' the concave moderators. \FS{For concreteness in some of the expressions below, we convene that $\infty - \infty := \infty$.}
\FS{\begin{lemma}
\label{lem:moderatorChoice}
Let $(\bphi,H) \in \D^g_\bmu(0)$. There is a concave moderator
$\bchi$ of $\bphi$ such that
\begin{equation}
\label{eqn:specialModerator}
\phi^k_t + \chi^k_t - \chi_{t+1} \geq 0 \quad \text{on } J_k^t \quad \text{ for all} \quad  t=0,\dots,n, \quad k \geq 1, \text{ and}
\end{equation}
\begin{equation}
\label{eqn:specialModerator2}
\phi^0_t + \chi^0_t - \chi_{t+1} \geq 0 \quad \mu_t\text{-a.s. on } I_0^t \quad\text{ for all} \quad t=1,\dots,n.
\end{equation}
As a consequence,
\[\sum_{t=1}^n\sum_{k \geq 1}(\mu_{t-1}-\mu_t)^k(\chi_t^k) \leq \bmu(\bphi).\]
\end{lemma}

\begin{proof}
Fix $1\leq t \leq n$ and let $(I_k^t,J_k^t)$ be the domain of some component of
$\M(\mu_{t-1},\mu_t)$. We define $\bchi=(\chi_t^k)$ by $\chi_{0}^{0}=0$ and
\[\chi_t^k(x_t) = \inf_{\bx \in \Psi^k_t(x_t)}\left\{\sum_{s=0}^{t-1} \phi_s^{\bbk_s(\bx)}(x_s) + (H \cdot \bx)_t\right\};\]
then $\chi_t^k$ is concave on $J_k^t$ for $k \geq 1$ as an infimum of affine functions.

We first show that 
\[\{\chi_t^k = +\infty\} \subseteq \left\{ \phi_{t-1}^{k'} = +\infty\right\} \cup \{\chi_{t-1}^{k'} = +\infty\}.\]

In particular, such points only exist after a chain of diagonal components from a point where $\phi_t^k(x_t) = \infty$.
Suppose $\chi_t^k(x_t) = +\infty$ and $k \geq 1$, then the predecessor paths of $x_t$ agree with the predecessor
paths of all of $J^k_t$ up to $t-1$, but $\{\sum_{s=0}^{t-1} \phi_s^{\bbk_s(\bx)}(x_s) < \infty\}$
must hold $\M(\bmu)$-q.s. as $\bphi \in L^{c,g}(\bmu)$. We must therefore have $x_t \in I_t^0$. Then, by definition,
$\chi_t^0(x_t) = \chi_{t-1}^k(x_t) + \phi_{t-1}^k(x_t)$ and the claim follows.

Next, we verify that $\bchi$ satisfies~\eqref{eqn:specialModerator} and \eqref{eqn:specialModerator2}. 
For notational convenience we for now set $\chi_{n+1} \equiv \inf_{\bx \in \V} \left\{\sum_{s=0}^{n} \phi_s^{\bbk_s(\bx)}(x_s) + (H \cdot \bx)_{n}\right\} \geq 0$. Restricting the infimum in the definition of $\bchi$ to the set of paths $\bx$ with $x_{t+1} = x_t \in I_{k'}^{t+1} \cap J_{k}^t$
yields
\begin{align*}
\chi_{t+1}(x_t) &= \chi_{t+1}^{k'}(x_t) = \inf_{\bx \in \Psi^{k'}_{t+1}(x_t)}\left\{\sum_{s=0}^{t} \phi_s^{\bbk_s(\bx)}(x_s) + (H \cdot \bx)_{t+1}\right\} \\
&\leq \inf_{\bx \in \Psi^{k}_t(x_t)}\left\{\sum_{s=0}^{t-1} \phi_s^{\bbk_s(\bx)}(x_s) + (H \cdot \bx)_t\right\} + \phi_t^k(x_t) \\
&= \chi_t^k(x_t) + \phi_t^k(x_t).
\end{align*}
Since $\cup_{k' \geq 0} I_{k'}^{t+1}=\R$, this will imply~\eqref{eqn:specialModerator}
after we check that $\chi_t^k > -\infty$ for $k \geq 1$ and $\chi_t^0 > -\infty$ holds $\mu_{t}^0$-a.s.,
which also implies that $\chi_t > -\infty$ holds $\mu_{t-1}$-almost surely. We show this inductively
for $t \geq 1$.

Clearly $\chi_{n+1} \geq 0 > -\infty$. Now, for $t \leq n$ the induction hypothesis is that
$\chi_{t+1} > -\infty$ holds almost surely $\mu_{t}$. 

From $\bphi \in L^{c,g}$ and $\chi_{t+1} > -\infty$ $\mu_t$-a.s.\ we have that
\[ \phi_t^k < \infty, \quad \chi_{t+1} > -\infty \quad \text{hold } \mu_t^k\text{-a.s.}\]
As $\chi_t^k$ is concave and $J_t^k$ is the convex hull of the topological support of $\mu_t^k$ we then get $\chi_t^k > -\infty$ on all of $J_t^k$ from the previous inequality.

For $k=0$, the inequality yields $\{\chi_t^0 = -\infty\} \subseteq \{\chi_{t+1} = -\infty\} \cup \{\phi_t^0(x_t) = \infty\}$ and both of these sets are $\mu_t$ nullsets. Finally $\mu_{t-1}(\{\chi_t = -\infty\}) = 0$
as this is a subset of the diagonal component where $\mu_{t-1}$ is dominated by $\mu_t$.

Set 
$\bar{\phi}_t^k := \phi_t^k + \chi_t^k - \chi_{t+1}|_{J_k^t}$ for $0 \leq t \leq n$; then
$\bar{\phi}_t^k \geq 0$. Moreover, choose an arbitrary $P \in \M(\bmu)$ with disintegration
$P = \mu_0 \otimes \kappa_1 \otimes \dots \otimes \kappa_n$ for some stochastic kernels
$\kappa_t(x_0,\dots,x_{t-1},dx_t)$. From Lemma \ref{lem:dualPayoffGeneral} we know that
\[\bmu(\bphi) = P\left[\sum_{t=0}^n\phi_t^{\bbk_t(\bX)}(X_t) + (H \cdot X)_n \right] < \infty.\]
We can therefore apply Fubini's theorem for kernels as in the proof of Lemma~\ref{lem:dualPayoffGeneral} to the expression
\begin{align*}
0 &\leq \sum_{t=0}^n \phi_t^{\bbk_t(\bx)}(x_t) + (H \cdot \bx)_n \\
&= \sum_{t=0}^n\bar{\phi}_t^{\bbk_t(\bx)}(x_t) + \sum_{t=1}^n \left(\chi_t(x_{t-1}) - \chi_t^{\bbk_t(\bx)}(x_t)\right) + (H \cdot \bx)_n
\end{align*}
and obtain

\begin{align*}
P&\left[\sum_{t=0}^n \phi_t^{\bbk_t(\bX)}(X_t) + (H \cdot X)_n \right] = \sum_{t=0}^n\sum_{k \geq 0}\mu_t^k(\bar{\phi}_t^k) +
\sum_{t=1}^n \sum_{k \geq 1} (\mu_{t-1}-\mu_t)^k(\chi_t^k)
\end{align*}
which shows that the right hand side is finite, and therefore $\bchi$ is
a concave moderator for $\bphi$. Finally, the second claim follows from
$\mu_t^k(\bar{\phi}^{k}_t) \geq 0$.
\end{proof}}

The last tool for our closedness result is a compactness property for concave functions in the one-step case; cf.\ \cite[Proposition~5.5]{BeiglbockNutzTouzi.15}.

\begin{proposition}
\label{prop:moderatorConvergence}
Let $\mu \leq_c \nu$ be irreducible with domain $(I,J)$ and let $a \in I$ be the common barycenter of
$\mu$ and $\nu$. Let $\chi_m : J \to \R$ be concave functions such that\footnote{To be specific, let us convene that $\chi_m'$ is the left derivative---this is not important here.}
\[\chi_m(a) = \chi_m'(a) = 0 \quad \text{ and } \quad \sup_{m \geq 1}(\mu - \nu)(\chi_m) < \infty. \]
There exists a subsequence $\chi_{m_k}$ which converges pointwise on $J$ to a concave function
$\chi: J \to \R$, and $(\mu - \nu)(\chi) \leq \liminf_{k}(\mu - \nu)(\chi_{m_k})$.
\end{proposition}

We are now ready to state and prove the analogue of Proposition~\ref{prop:dualclosedregular} in the generalized
dual.

\begin{proposition}
\label{prop:dualclosed}
Let $f^m: \R^{n+1} \to [0, \infty]$, $m\geq1$ be a sequence of functions such that 
\[f^m \to f \qquad \text{pointwise}\] 
and let $(\bphi^m, H^m) \in \D^g_\bmu(f^m)$ be such that $\sup_m \bmu(\bphi^m) < \infty$. Then there exist
$(\bphi,H) \in \D^g_\bmu(f)$ with
\[\bmu(\bphi) \leq \liminf_{m \to \infty} \bmu(\bphi^m).\]
\end{proposition}

\begin{proof}
Since $(\bphi^m,H^m) \in \D^g_\bmu(f^m)$ and $f^m \geq 0$, we can introduce a sequence of concave moderators
$\bchi_m$ as in Lemma \ref{lem:moderatorChoice}. A normalization of $(\bphi^m,H^m)$ as in
Lemma \ref{lem:modify}\,(i) and~(ii), in the general form of Remark \ref{rem:manymodify}, allows us to assume without loss of generality that 
$\chi_{t,m}^0 \equiv 0$ and
$\chi_{t,m}^k(a_t^k) = ({\chi_{t,m}^k})'(a_t^k) = 0$, where $a_t^k$ is the barycenter of $\mu_t^k$---this modification is the main merit of lifting to the generalized dual space. While the generalized dual gives enough degrees of freedom to choose this normalization, the dual without the generalization does not. This is related to the possible overlap of the intervals $I,J$ at the different times $t$; see also Figure~\ref{fig:polarStructure} and the paragraph preceding Example~\ref{ex:polarStructure}.

By passing to a subsequence as in
Proposition~\ref{prop:moderatorConvergence} for each component and using a diagonal argument, 
we obtain pointwise limits $\chi_t^k: J_k^t \to \R$ for
$\chi_{t,m}^k$ after passing to another subsequence.

Since
$\phi_{t,m}^k + \chi_{t,m}^k - \chi_{t+1,m} \geq 0$ on $J_k^t$\footnote{Observe that this inequality will still hold
after modifying $\bphi$ and $\bchi$ as in Lemma \ref{lem:modify}.} and $\chi_{t,m}^k \to \chi_t^k$ as well as $
\chi_{t+1,m} \to \chi_{t+1}$, we can apply Komlos' lemma \MN{(in the form of \cite[Lemma~A1.1]{DelbaenSchachermayer.94} and its remark)} to find \MNN{convex combinations}
$\tilde{\phi}_{t,m}^k \in \conv\{\phi_{t,m}^k,\phi_{t,m+1}^k, \dots\}$ 
which converge $\mu_t^k$-a.s.\ for $0 \leq t \leq n$. We may assume without loss of generality that $\tilde{\phi}_{t,m}^k = \phi_{t,m}^k$. Thus,
we can set
\begin{align*}
\phi_t^k &:= \limsup \phi_{t,m}^k \quad \text{ on } J_k^t \quad \text{for } t=1,\dots,n, \\
\phi_0 &:= \liminf \phi_{0,m}
\end{align*}
to obtain
\[\phi_{t,m}^k \to \phi_t^k \quad \mu_t^k\text{-a.s.} \quad \text{ and }\quad
\phi_t^k + \chi_t^k - \chi_{t+1} \geq 0 \text{ on }J_k^t. \]
We can now apply Fatou's lemma and Proposition~\ref{prop:moderatorConvergence} to deduce that
\begin{align*}
\bmu&(\bphi) = \sum_{t = 0}^n \sum_{k \geq 0} \mu_t^k(\phi_t^k + \chi_t^k - \chi_{t+1}) +  \sum_{t=1}^n \sum_{k \geq 1} (\mu_{t-1}-\mu_t)^k(\chi_t^k) \\
&\leq \sum_{t = 0}^n \sum_{k \geq 0} \liminf \mu_t^k(\phi_{t,m}^k + \chi_{t,m}^k - \chi_{t+1,m}) \\
& \quad\; + \sum_{t=1}^n \sum_{k \geq 1} \liminf (\mu_{t-1}-\mu_t)^k(\chi_{t,m}^k) \\
& \leq \liminf \left[ \sum_{t = 0}^n \sum_{k \geq 0} \mu_t^k(\phi_{t,m}^k + \chi_{t,m}^k - \chi_{t+1,m}) +  \sum_{t=1}^n \sum_{k \geq 1} (\mu_{t-1}-\mu_t)^k(\chi_{t,m}^k) \right] \\
&= \liminf \bmu(\bphi^m) < \infty.
\end{align*}
In particular, we see that $\bphi \in L^{c,g}(\bmu)$ with concave moderator $\bchi$.

It remains to construct the predictable process $H=(H_{1},\dots,H_{n})$. With a mild abuse of notation, we shall identify $H_{t}(x_{0},\dots,x_{n})$ with the corresponding function of $(x_{0},\dots,x_{t-1})$ in this proof.

We first define for each $\bk = (k_0,\dots,k_t)$ and $\bx=(x_{0},\dots,x_{t})$ such that $\bk=\bbk(\bx)$, the functions $G_{t,m}^{\bk}$ and $G_{t}^{\bk}$ by 
\begin{align*}
  G_{t,m}^{\bk}(\bx) &:= \sum_{s=0}^t \phi_{s,m}^{k_s}(x_s) + \sum_{s=1}^t H_{s,m}(x_{0},\dots,x_{s-1})\cdot (x_s - x_{s-1}),\\
  G_{t}^{\bk}(\bx)&:= \liminf G_{t,m}^{\bk}(\bx).
\end{align*}
Given $\bk=(k_{0},\dots,k_{t})$, we write $\bk'=(k_{0},\dots,k_{t-1})$. We claim that there exists
an $\FF$-predictable process $H$ such that for all $1 \leq t \leq n$,
\begin{equation} 
\label{eqn:superhedgeIterative}
G_{t-1}^{\bk'}(x_{0},\dots,x_{t-1}) + \phi_t^{k_t}(x_t)+ H_t(x_{0},\dots,x_{t-1}) \cdot (x_t - x_{t-1}) \geq G_{t}^{\bk}(x_{0},\dots,x_{t}).
\end{equation}
Once this is established, the proposition follows by induction since $G_0^{(0)}(x_{0}) = \phi_0(x_{0})$ and $G_{n}^{\bk}(x_{0},\dots,x_{n}) \geq f(x_{0},\dots,x_{n})$.

To prove the claim, write $g^{\conc}$ for the concave hull of a function $g$ and observe that
\begin{align*}
\liminf[G_{t-1,m}^{\bk'}&(x_{0},\dots,x_{t-1}) + H_{t,m}(x_{0},\dots,x_{t-1}) \cdot (x_t - x_{t-1})] \\
 &\geq  \liminf[(G_{t,m}^{\bk}(x_{0},\dots,x_{t-1},\cdot) - \phi_{t,m}^{k_t}(\cdot))^{\conc}(x_t)] \\
 &\geq  [\liminf(G_{t,m}^{\bk}(x_{0},\dots,x_{t-1},\cdot) - \phi_{t,m}^{k_t}(\cdot)]^{\conc}(x_t) \\
 &\geq [G_{t}^{\bk}(x_{0},\dots,x_{t-1},\cdot) - \phi_t^{k_t}(\cdot)]^{\conc}(x_t)\\
 &=: \hat{\phi}_t^{\bk}(x_0,\dots,x_{t-1},x_{t}).
\end{align*}
By construction, $\hat{\phi}_t^{\bk}$ is concave in the last variable and satisfies
\[G_{t-1}^{\bk'}(x_{0},\dots,x_{t-1}) \geq \hat{\phi}_t^{\bk}(x_0,\dots,x_{t-1},x_{t-1}).\]
Let $\partial_t \hat{\phi}_t^{\bk}$ denote the \FS{left} partial derivative in the last variable
and set 
\[H_t^{\bk}(x_{0},\dots,x_{t-1}) := \partial_t \hat{\phi}_t^{\bk}(x_0,\dots,x_{t-1},x_{t-1})\]
\FS{for $\bk_t \geq 1$ and $H_t^{\bk}(x_{0},\dots,x_{t-1})=0$ for $\bk_t = 0$;} then we have
\begin{align*}
G_{t-1}^{\bk'}&(x_{0},\dots,x_{t-1}) + H_t^{\bk}(x_{0},\dots,x_{t-1}) \cdot (x_t - x_{t-1}) \\
&\geq \hat{\phi}_t^{\bk}(x_{0},\dots,x_{t-1},x_{t-1}) + H_t^{\bk}(x_{0},\dots,x_{t-1}) \cdot (x_t - x_{t-1}) \\
&\geq \hat{\phi}_t^{\bk}(x_0,\dots,x_{t-1},x_t)\\
&\geq G_{t}^{\bk}(x_{0},\dots,x_{t}) - \phi_t^{k_t}(x_{t}).
\end{align*}
Finally, for any $(x_{0},\dots,x_{t-1})\in\R^{t}$, we define $H_{t}(x_{0},\dots,x_{t-1})$ as
$$
  \begin{cases}
  H_t^{\bk}(x_{0},\dots,x_{t-1}),&\mbox{if $\bk=\bbk(x_{0},\dots,x_{t-1},x_{t})$ for some $x_{t}\in\R$}\\
  0,&\mbox{otherwise};
  \end{cases}
$$
this is well-defined since $\bbk(x_{0},\dots,x_{t})$ depends only on $(x_{0},\dots,x_{t-1})$. The predictable process $H$ satisfies~\eqref{eqn:superhedgeIterative} and thus the proof is complete.
\end{proof}

\begin{proof}[Proof of Proposition~\ref{prop:dualclosedregular}]
  In view of Remark~\ref{rem:liftdual} and Proposition~\ref{prop:dualEquivalence}, the result follows from Proposition~\ref{prop:dualclosed}.
\end{proof}

\section{Duality Theorem and Monotonicity Principle}\label{se:duality}

The first goal of this section is a duality result for the multistep martingale transport problem; it establishes the absence of a duality gap and the existence of optimizers in the dual problem. (As is well known, an optimizer for the primal problem only exists under additional conditions, such as continuity of $f$.) The second goal is a monotonicity principle describing the geometry of optimal transports; it will be a consequence of the duality result.

As above, we consider a fixed vector $\bmu=(\mu_{0},\dots,\mu_{n})$ of marginals in convex order. The primal
and dual problems as defined follows.

\begin{definition}
Let $f: \R^{n+1} \to [0,\infty]$. The \emph{primal problem} is
\[\S_\bmu(f) := \sup_{P \in \M(\bmu)} P(f) \in [0,\infty],\]
where $P(f)$ refers to the outer integral if $f$ is not measurable. The
\emph{dual problem} is
\[\I_\bmu(f) := \inf_{(\bphi,H) \in \D_\bmu(f)} \bmu(\bphi) \in [0,\infty]. \]
\end{definition}

We recall that a function $f: \R^{n+1}\to [0,\infty]$ is called \emph{upper semianalytic} if the sets $\{f\geq c\}$ are analytic for all $c\in\R$, where a subset of $\R^{n+1}$ is called analytic if it is the image of a Borel subset of a Polish space under a Borel mapping. Any Borel function is upper semianalytic and any upper semianalytic function is universally measurable; we refer to \cite[Chapter~7]{BertsekasShreve.78} for background. The following is the announced duality result.

\begin{theorem}[Duality]
\label{thm:duality}
Let $f:\R^{n+1} \to [0,\infty]$.
\begin{enumerate}[(i)]
\item If $f$ is upper semianalytic, then $\S_\bmu(f) = \I_\bmu(f) \in [0,\infty]$.
\item If $\I_\bmu(f) < \infty$, there exists a dual optimizer $(\bphi,H) \in \D_\bmu(f)$.
\end{enumerate}
\end{theorem} 

\begin{proof}
  Given our preceding results, much of the proof follows the lines of the corresponding result for the one-step case in~\cite[Theorem~6.2]{BeiglbockNutzTouzi.15}; therefore, we shall be brief. We mention that the present theorem is slightly more general than the cited one in terms of the measurability condition \FS{($f$ is upper semianalytic instead of Borel);} this is due to the global proof given here.
  
\emph{Step~1.} Using~Lemma \ref{lem:dualPayoff} we see that $\S_\bmu(f) \leq \I_\bmu(f)$ holds
for all upper semicontinuous $f:\R^{n+1} \to [0,\infty]$.

\emph{Step~2.} Using the de la Vall\'ee--Poussin theorem and our assumption that the marginals have a finite first moment, there exist increasing, superlinearly growing functions
$\zeta_{\mu_t} : \R^+ \to \R^+$ such that $x \mapsto \zeta_{\mu_t}(|x|)$ is $\mu_t$-integrable for all $0\leq t\leq n$. Define 
\[\zeta(x_{0},\dots,x_{n}) := 1 + \sum_{t=0}^n \zeta_{\mu_t}(|x_t|)\]
and let $C_\zeta$ be the vector space of all continuous functions $f$ such that $f/\zeta$
vanishes at infinity. Then, a Hahn--Banach separation argument can be used to
show that $\S_\bmu(f) \geq \I_\bmu(f)$ holds for all $f \in C_{\zeta}$; \FS{the details of the argument are the same as in the proof of \cite[Lemma~6.4]{BeiglbockNutzTouzi.15}}.

\emph{Step~3.} Let $f$ be bounded and upper semicontinuous; then there exists a sequence 
 of bounded continuous functions $f^m\in C_b(\R^{n+1})$ which decrease to $f$ pointwise. As $C_b(\R^{n+1}) \subseteq C_\zeta$, we have $\S_\bmu(f^m) = \I_\bmu(f^m)$ for all $m$ by the first two steps.
 
Let $\U$ be the set of all bounded, nonnegative, upper semicontinuous functions on $\R^{n+1}$.
We recall that a map $\bC:[0,\infty]^{\R^{n+1}} \to [0,\infty]$ is called
a $\U$-capacity if it is monotone, sequentially continuous upwards on $[0,\infty]^{\R^{n+1}}$
and sequentially continuous downwards on $\U$.
The functional
$f \mapsto \S_\bmu(f)$ is a $\U$-capacity; this follows from the weak
compactness of $\M(\bmu)$ and the arguments in~\cite[Propositions~1.21,\,1.26]{Kellerer.84}.

It follows that $\S_\bmu(f^m) \to \S_\bmu(f)$. By the monotonicity of $f \mapsto \I_\bmu(f)$
and Step~1 we obtain
\[\I_\bmu(f) \leq \lim \I_\bmu(f^m) = \lim \S_\bmu(f^m) = \S_\bmu(f) \leq \I_\bmu(f).\]

\emph{Step~4.} Since $\S_\bmu=\I_\bmu$ on $\U$ by Step~3, $\I_\bmu$ is sequentially downward
continuous on $\U$ like $\S_\bmu$. On the other hand, Proposition~\ref{prop:dualclosedregular} implies that it is sequentially upwards continuous on $[0,\infty]^{\R^{n+1}}$. As a result, $\I_\bmu$ is a $\U$-capacity.

\emph{Step~5.} Let $f:\R^{n+1} \to [0,\infty]$ be upper semianalytic. For any $\U$-capacity~$\bC$, Choquet's capacitability theorem shows that
\[\bC(f) = \sup\{\bC(g) : g \in \U, g \leq f\}.\]
As $\S_\bmu$ and $\I_\bmu$ are $\U$-capacities that coincide on $\U$, it follows that $\S_\bmu(f)=\I_\bmu(f)$. This completes the proof of~(i).

\emph{Step~6.} To see that the infimum $\I_\bmu(f)$ is attained if it is finite, we merely need to apply Proposition~\ref{prop:dualclosedregular} with the constant sequence $f^m = f$.
\end{proof}

\MN{We can easily relax the lower bound on $f$.}

\begin{remark}\label{rk:relaxLowerBound}
Let $f:\R^{n+1} \to (-\infty,\infty]$ and suppose there exist
$\bphi \in \prod_{t=0}^n L^1(\mu_t)$ and a predictable process $H$ such that
\[f \geq \sum_{t=0}^n\phi_t(X_t) + (H \cdot X)_n \quad \text{on} \quad \V.\]
Then we can apply Theorem~\ref{thm:duality} to
$\left[f - \sum_{t=0}^n\phi_t(X_t) - (H \cdot X)_n\right]^+$
and obtain the analogue of its assertion for $f$.
\end{remark}

The duality result gives rise to a monotonicity principle describing the support of optimal martingale transports, in the spirit of the cyclical monotonicity condition from classical transport theory. The following generalizes the results of \cite[Lemma~1.11]{BeiglbockJuillet.12} and \cite[Corollary~7.8]{BeiglbockNutzTouzi.15} for the one-step martingale transport problem.

\begin{theorem}[Monotonicity Principle]
\label{thm:monotonicityPrinciple}
Let $f: \R^{n+1} \to [0,\infty]$ be Borel and suppose that $\S_\bmu(f) < \infty$. There exists a Borel set 
$\Gamma \subseteq \R^{n+1}$
with the following properties.
\begin{enumerate}[(i)]
\item A measure $P \in \M(\bmu)$ is concentrated on $\Gamma$ if and only if it is optimal
for $\S_\bmu(f)$.
\item Let $\bar{\bmu}=(\bar{\mu}_{0},\dots,\bar{\mu}_{n})$ be another vector of marginals in convex order. If $\bar{P} \in \M(\bar{\bmu})$
is concentrated on $\Gamma$, then $\bar{P}$ is optimal for $\S_{\bar{\bmu}}(f)$.
\end{enumerate}
Indeed, if $(\bphi,H) \in \D_\bmu(f)$ is an optimizer for $\I_\bmu(f)$, then we can take
\[\Gamma := \left\{\sum_{t=0}^n \phi_t(X_t) + (H \cdot X)_n = f\right\} \cap \V.\]
\end{theorem}

\begin{proof}
As $\S_\bmu(f) < \infty$, Theorem~\ref{thm:duality} shows that $\I_\bmu(f) = \S_\bmu(f) < \infty$ and that there exists a dual optimizer $(\bphi,H) \in \D_\bmu(f)$. In particular, we can define $\Gamma$ as above.

(i) As $0\leq f$ and $P(f) \leq \S_\bmu(f) < \infty$ for all
$P \in \M(\bmu)$, we see that $f$ is $P$-integrable for all $P \in \M(\bmu)$. Since
$\sum_{t=0}^n \phi_t(X_t) + (H \cdot X)_n \geq 0$ on the effective domain $\V$, and
$P\left[\sum_{t=0}^n \phi_t(X_t) + (H \cdot X)_n\right] = \bmu(\bphi) = \I_\bmu(f) < \infty$ by
Lemma~\ref{lem:dualPayoff}, we
also obtain the $P$-integrability of $\sum_{t=0}^n \phi_t(X_t) + (H \cdot X)_n$. In particular,
\[0 \leq P\left[\sum_{t=0}^n \phi_t(X_t) + (H \cdot X)_n - f\right] = \bmu(\bphi) - P(f) = \S_\bmu(f) - P(f)\]
and equality holds if and only if $P$ is concentrated on $\Gamma$.

(ii) We may assume that $\bar{P}$ is a probability measure with $\bar{P}(f)<\infty$. As a first step, we show that the effective domain $\bar{\V}$ of
$\M(\bar{\bmu})$ is a subset of the effective domain $\V$ of $\M(\bmu)$. \FS{To that end, it is sufficient to show that if $1\leq t \leq n$ and $x\in\R$ are such that $u_{\mu_{t-1}}(x) = u_{\mu_t}(x)$, then $u_{\bar{\mu}_{t-1}}(x) = u_{\bar{\mu}_t}(x)$, and if moreover
$\partial^{+}u_{\mu_{t-1}}(x) = \partial^{+} u_{\mu_t}(x)$, then $\partial^{+}u_{\bar{\mu}_{t-1}}(x) = \partial^{+}u_{\bar{\mu}_t}(x)$, and similarly for the left derivative $\partial^{-}$ (cf.\ Proposition~\ref{prop:oneStepDecomposition}).
Indeed, for $t$ and $x$ such that $u_{\mu_{t-1}}(x) = u_{\mu_t}(x)$, our assumption that $\Gamma\subseteq \V$ implies
\[\Gamma \subseteq (X_{t-1},X_t)^{-1}\big( 
(-\infty,x]^2 \cup [x,\infty)^2\big). \]
Using also that $\EE^{\bar{P}}[X_{t}|\F_{t-1}]=X_{t-1}$ and that $\bar{P}$ is concentrated on $\Gamma$,
\begin{align*}
u_{\bar{\mu}_{t-1}}(x) &= \EE^{\bar{P}}[|X_{t-1} - x|] \\
&= \EE^{\bar{P}}[(X_{t-1} - x)\1_{X_{t-1} \geq x}] + \EE^{\bar{P}}[(x - X_{t-1})\1_{X_{t-1} \leq x}]\\
&= \EE^{\bar{P}}[(X_t - x)\1_{X_{t-1} \geq x}] + \EE^{\bar{P}}[(x - X_t)\1_{X_{t-1} \leq x}] \\
&= \EE^{\bar{P}}[|X_t - x|]= u_{\bar{\mu}_t}(x)
\end{align*}
as desired. If in addition $\partial^{+}u_{\mu_{t-1}}(x) = \partial^{+}u_{\mu_t}(x)$, then $\Gamma\subseteq \V$
implies
\[\Gamma \subseteq (X_{t-1},X_t)^{-1}\big( 
(-\infty,x]^2 \cup (x,\infty)^2\big). \]
As $\bar{P}$ is concentrated on $\Gamma$, it follows that
\begin{align*}
\partial^{+}u_{\bar{\mu}_{t-1}}(x) &= \bar P[X_{t-1}\leq x] - \bar P [X_{t-1} > x] \\
&= \bar P[X_t\leq x] - \bar P [X_t > x] = \partial^{+}u_{\bar{\mu}_{t}}(x)
\end{align*}
as desired. The same argument can be used for the left derivative
and we have shown that $\bar{\V}\subseteq \V$.}

In view of that inclusion, the inequality $\sum_{t=0}^n \phi_t(X_t) + (H \cdot X)_n \geq f$ holds
on~$\bar{\V}$. Since $\bar P$ is concentrated on $\Gamma$, 
\[
  \bar{P}\left[\sum_{t=0}^n \phi_t(X_t) + (H \cdot X)_n\right] = \bar{P}(f) <\infty.
\]
We may follow the arguments in the proof of Lemma~\ref{lem:moderatorChoice} to construct a moderator 
$\bchi$ and establish that $(\bphi,H) \in \D^g_{\bar \bmu}(f)$, where we are implicitly using the embedding 
detailed in Remark~\ref{rem:liftdual}. \FS{(Note that the proof of Lemma~\ref{lem:moderatorChoice} uses 
the condition $(\bphi,H) \in \D^g_{\bar \bmu}(0)$ only to establish 
$\bar{P}\left[\sum_{t=0}^n \phi_t(X_t) + (H \cdot X)_n\right] <\infty$. In the present situation the latter 
is known a priori and  the condition is not needed.)}
 Then, we can modify $\bchi$ as in the proof of Proposition~\ref{prop:dualEquivalence} to see that $(\bphi,H) \in \D_{\bar \bmu}(f)$.
As a result, we may apply Lemma~\ref{lem:dualPayoff} to obtain that
\[ \bar{P}(f) = \bar{P}\left[\sum_{t=0}^n \phi_t(X_t) + (H \cdot X)_n\right] = \bar{\bmu}(\bphi),\]
whereas for any other $P' \in \M(\bar{\bmu})$ we have
\[P'(f) \leq P'\left[\sum_{t=0}^n \phi_t(X_t) + (H \cdot X)_n\right] = \bar{\bmu}(\bphi)= \bar{P}(f). \]
This shows that $\bar{P}\in\M(\bar{\bmu})$ is optimal.
\end{proof}

\section{Left-Monotone Transports}\label{se:shadows}

In this section we define left-monotone transports through a shadow property and prove their existence.

\subsection{Preliminaries}

Before moving on to the $n$-step case, we recall the essential definitions and results regarding the one-step version of the left-monotone transport (also called the Left-Curtain coupling). The first notion is the so-called shadow, and it will be useful to define it for measures $\mu \leq_{pc} \nu$ in \emph{positive convex order}, meaning that $\mu(\phi) \leq \nu(\phi)$ for any nonnegative convex function $\phi$. Clearly, this order is weaker than the convex order $\mu \leq_{c} \nu$, and it is worth noting that $\mu$ may have a smaller mass than $\nu$. The following is the result of \cite[Lemma~4.6]{BeiglbockJuillet.12}. 

\begin{lemma}
\label{lem:shadowDefinition}
Let $\mu \leq_{pc} \nu$. Then the set
\[\casts{\mu}{\nu} := \{ \theta : \mu \leq_c \theta \leq \nu\}\]
is non-empty and contains a unique least element $\shadow{\mu}{\nu}$ for the convex order:
\[\shadow{\mu}{\nu} \leq_c \theta \text{ for all } \theta \in \casts{\mu}{\nu}.\]
The measure $\shadow{\mu}{\nu}$ is called the \emph{shadow} of $\mu$ in $\nu$.
\end{lemma}

It will be useful to have the following picture in mind: if $\mu$ is a Dirac measure, its shadow in $\nu$ is a measure $\theta$ of equal mass and barycenter, chosen such as to have minimal variance subject to the constraint $\theta\leq \nu$.


The second notion is a class of reward functions.

\begin{definition}
A Borel function $f: \R^2 \to \R$ is called \emph{second-order Spence--Mirrlees} if
$y \mapsto f(x',y) - f(x,y)$ is strictly convex for any $x < x'$.
\end{definition}

We note that if $f$ is sufficiently differentiable, this can be expressed as the cross-derivative condition $f_{xyy}>0$ which has also been called the martingale Spence--Mirrlees condition, in analogy to the classical Spence--Mirrlees condition $f_{xy}>0$.

In the one-step case, the left-monotone transport is unique and can be characterized as follows; cf.\ \cite[Theorems~4.18,\,4.21,\,6.1]{BeiglbockJuillet.12} where this transport is called the Left-Curtain coupling, as well as \cite[Theorem~1.2]{NutzStebegg.16} for the third equivalence in the stated generality.

\begin{proposition}
\label{prop:leftMonotoneTransport}
Let $\mu \leq_c \nu$ and $P \in \M(\mu,\nu)$. The following are equivalent:
\begin{enumerate}[(i)]
\item For all $x \in \R$ and $A \in \B(\R)$, 
\[P[(-\infty,x] \times A] = \shadow{\mu|_{(-\infty,x]}}{\nu}(A).\]
\item $P$ is concentrated on a Borel set $\Gamma \subseteq \R^2$ satisfying
\[(x,y^-),(x,y^+),(x',y') \in \Gamma,\;\; x < x' \quad \Rightarrow \quad y' \notin (y^-,y^+). \]
\item $P$ is an optimizer of $\S_{\mu,\nu}(f)$ for some (and then all) 
$f : \R^2 \to \R$ second-order Spence--Mirrlees such that there exist functions 
$a \in L^1(\mu)$, $b \in L^1(\nu)$
with $|f(x,y)| \leq a(x) + b(y)$. 

\end{enumerate}
There exists a unique measure $\bar{P} \in \M(\mu,\nu)$ satisfying (i)--(iii), and $\bar{P}$ is called the (one-step) left-monotone transport.
\end{proposition}

If $\mu$ is a discrete measure, the characterization in~(i) can be understood as follows: the left-monotone transport $\bar{P}$ processes the atoms of $\mu$ from left to right, mapping each one of them to its shadow in the remaining target measure.

Next, we record two more results about shadows that will be used below. The first one, cited from \cite[Theorem 3.1]{BeiglbockJuillet.16}, generalizes the above idea in the sense that the atoms are still mapped to their shadows but can be processed in any given order; in the general (non-discrete) case, such an order is defined by a coupling~$\pi$ from the uniform measure to $\mu$.

\begin{proposition}
\label{prop:generalShadowPlan}
Let $\mu \leq_c \nu$ and $\pi \in \Pi(\lambda,\mu)$ where $\lambda$ denotes the Lebesgue measure
on $[0,1]$. Then there exists a unique measure $Q \in \Pi(\lambda,\mu,\nu)$ on $\R^{3}$ such that $Q \circ (X_0,X_1)^{-1} = \pi$ and 
\[Q|_{[0,s] \times \R \times \R} \circ (X_1,X_2)^{-1} \in \M(\pi_s,\shadow{\pi_s}{\nu}),\quad s\in\R,\]
where $\pi_s := \pi|_{[0,s]\times \R} \circ (X_1)^{-1}$.
\end{proposition}

We shall also need the following facts about shadows.

\begin{lemma}
\label{lem:shadowProperties}
\begin{enumerate}[(i)]
\item Let $\mu_1,\mu_2,\nu$ be finite measures satisfying $\mu_1+\mu_2 \leq_{pc} \nu$. Then
$\mu_2 \leq_{pc} \nu - \shadow{\mu_1}{\nu}$ and $\shadow{\mu_1+\mu_2}{\nu} = \shadow{\mu_1}{\nu}
+ \shadow{\mu_2}{\nu - \shadow{\mu_1}{\nu}}$.
\item Let $\mu, \nu_1,\nu_2$ be finite measures such that $\mu \leq_{pc} \nu_1 \leq_c \nu_2$.
Then, it follows that $\shadow{\mu}{\nu_1} \leq_{pc} \nu_2$. Moreover, $\shadow{\shadow{\mu}{\nu_1}}{\nu_2} = \shadow{\mu}{\nu_2}$
if and only if $\shadow{\mu}{\nu_1} \leq_c \shadow{\mu}{\nu_2}$.
\end{enumerate}
\end{lemma}

\begin{proof}
Part~(i) is \cite[Theorem~4.8]{BeiglbockJuillet.12}.
To obtain the first statement in~(ii), we observe that $\shadow{\mu}{\nu_1} \leq \nu_1 \leq_c \nu_2$ and hence
\[\shadow{\mu}{\nu_1}(\phi) \leq \nu_1(\phi) \leq \nu_2(\phi)\]
for any nonnegative convex function $\phi$.
Turning to the second statement, the  ``only if'' implication follows directly from the definition of the shadow in Lemma~\ref{lem:shadowDefinition}.
To show the reverse implication, suppose that $\shadow{\mu}{\nu_1} \leq_c \shadow{\mu}{\nu_2}$. Then, we have
\[
\mu  \leq_c \shadow{\mu}{\nu_1} \leq_c \shadow{\shadow{\mu}{\nu_1}}{\nu_2} \leq \nu_2 \quad\text{and} \quad
\shadow{\mu}{\nu_1}  \leq_c \shadow{\mu}{\nu_2} \leq \nu_2.
\]
These inequalities imply that
\[\shadow{\shadow{\mu}{\nu_1}}{\nu_2} \in \casts{\mu}{\nu_2} \quad \text{and} \quad
\shadow{\mu}{\nu_2} \in \casts{\shadow{\mu}{\nu_1}}{\nu_2},\]
and now the minimality property of the shadow shows that
\[\FS{\shadow{\mu}{\nu_2} \leq_c \shadow{\shadow{\mu}{\nu_1}}{\nu_2} \quad \text{and} \quad 
\shadow{\shadow{\mu}{\nu_1}}{\nu_2} \leq_c \MNN{\shadow{\shadow{\mu}{\nu_2}}{\nu_{2}}}=\shadow{\mu}{\nu_2}}\]
as desired.
\end{proof}

\subsection{Construction of a Multistep Left-Monotone Transport}

Our next goal is to define and construct a multistep left-monotone transport. The following concept will be crucial.

\begin{definition}\label{def:obstructedShadow}
Let $\mu_0 \leq_{pc} \mu_1 \leq_c \dots \leq_c \mu_n$.
For $1\leq t\leq n$, the \emph{obstructed shadow of $\mu_0$ in $\mu_t$ through $\mu_1, \dots,\mu_{t-1}$} is iteratively defined by
\[\shadow{\mu_0}{\mu_1,\dots,\mu_t} := \shadow{\shadow{\mu_0}{\mu_1,\dots,\mu_{t-1}}}{\mu_t}.\]
\end{definition}

The obstructed shadow is well-defined due to Lemma~\ref{lem:shadowProperties}\,(ii). An alternative definition is provided by the following characterization.

\begin{lemma}
\label{le:shadowLeastElement}
Let $\mu_0 \leq_{pc} \mu_1 \leq_c \dots \leq_c \mu_n$ and $1\leq t\leq n$.  Then $\shadow{\mu_0}{\mu_1,\dots,\mu_t}$ is the unique least element of the set 
\begin{equation*}
\casts{\mu_0}{\mu_t}^{\mu_1,\dots,\mu_{t-1}} :=  \{\theta_t \leq \mu_t : \exists \theta_s\leq \mu_s, \, 1 \leq s \leq t-1, \, \mu_0 \leq_c \theta_1 \leq_c \dots \leq_c \theta_t\}
\end{equation*}
for the convex order; that is,
$\shadow{\mu_0}{\mu_1,\dots,\mu_t} \leq_c \theta$ for all elements $\theta$.
\end{lemma}

\begin{proof}
For $t=1$ this holds by the definition of the shadow in Lemma~\ref{lem:shadowDefinition}. For $t > 1$, we inductively assume that
$\shadow{\mu_0}{\mu_1,\dots,\mu_{t-1}}$ is the least element of
$\casts{\mu_0}{\mu_{t-1}}^{\mu_1,\dots,\mu_{t-2}}$.
Consider an arbitrary element $\theta_t \in \casts{\mu_0}{\mu_t}^{\mu_1,\dots,\mu_{t-1}}$ and
fix some 
\[
  \mu_0 \leq_c \theta_1 \leq_c \dots \leq_c \theta_{t-1} \leq_c \theta_t\quad\mbox{with}\quad \theta_s\leq \mu_s, \quad  1 \leq s \leq t-1.
\]
Then, $\theta_{t-1} \in \casts{\mu_0}{\mu_{t-1}}^{\mu_1,\dots,\mu_{t-2}}$ and in particular $\shadow{\mu_0}{\mu_1,\dots,\mu_{t-1}} \leq_c \theta_{t-1}$. Recall that
$\shadow{\mu_0}{\mu_1,\dots,\mu_t}$ is defined as the least element \MN{for $\leq_c$} of 
\begin{align*}
\casts{\shadow{\mu_0}{\mu_1,\dots,\mu_{t-1}}}{\mu_t} &= \{\theta \leq \mu_t :
\shadow{\mu_0}{\mu_1,\dots,\mu_{t-1}} \leq_c \theta\} \\
&\supseteq \{\theta \leq \mu_t :  \theta_{t-1} \leq_c \theta\} \ni \theta_t.
\end{align*}
Hence, $\shadow{\mu_0}{\mu_1,\dots,\mu_t} \leq_c \theta_t$, and as $\theta_t \in \casts{\mu_0}{\mu_t}^{\mu_1,\dots,\mu_{t-1}}$ was arbitrary,
this shows that $\shadow{\mu_0}{\mu_1,\dots,\mu_t}$ is a least element of $\casts{\mu_0}{\mu_t}^{\mu_1,\dots,\mu_{t-1}}$. The uniqueness of the least element follows from the general fact that $\theta^{1}_{t}\leq_{c}\theta^{2}_{t}$ and $\theta^{2}_{t}\leq_{c}\theta^{1}_{t}$ imply $\theta^{1}_{t}=\theta^{2}_{t}$.
\end{proof}

We can now state the main result of this section.
\begin{theorem}
\label{thm:multistepleft}
Let $\bmu=(\mu_{0},\dots,\mu_{n})$ be in convex order. Then there exists $P \in \M(\bmu)$ such that the bivariate projections
$P_{0t} := P \circ (X_0,X_t)^{-1}$ satisfy
\[
P_{0t}[(-\infty,x] \times A] = \shadow{\mu_0|_{(-\infty,x]}}{\mu_1,\dots,\mu_t}(A)
\quad \text{for} \quad x \in \R,\; A \in \B(\R),
\]
for all $1\leq t\leq n$. Any such $P \in \M(\bmu)$ is called a \emph{left-monotone transport}.
\end{theorem}

We observe that an $n$-step left-monotone transport is defined  purely in terms of its bivariate projections $P \circ (X_0,X_t)^{-1}$. In the one-step case, this completely determines the transport. For $n>1$, we shall see that there can be multiple (and then infinitely many) left-monotone transports; in fact, they form a convex compact set. This will be discussed in more detail in Section~\ref{se:uniqueness}, where it will also be shown that uniqueness does hold if $\mu_{0}$ is atomless.

\begin{proof}[Proof of Theorem \ref{thm:multistepleft}]
\emph{Step~1.}
We first construct measures $\pi_t \in \Pi(\lambda,\mu_t)$, $0 \leq t \leq n$ such that
\[
  \pi_t|_{[0,\mu_0((-\infty,x])]\times \R} \circ X_1^{-1} = \shadow{\mu_0|_{(-\infty,x]}}{\mu_1,\dots,\mu_t}
\]
for all $x \in \R$, as well as measures $Q_t \in \Pi(\lambda,\mu_{t-1},\mu_t)$, $1 \leq t \leq n$ such that
\begin{align}
Q_t|_{[0,\mu_0((-\infty,x])]\times \R \times \R} & \circ (X_1,X_2)^{-1} \in \nonumber \\
&\M\big(\shadow{\mu_0|_{(-\infty,x]}}{\mu_1,\dots,\mu_{t-1}},\shadow{\mu_0|_{(-\infty,x]}}{\mu_1,\dots,\mu_t}\big) \label{eqn:martingaleproperty}
\end{align}
for all $x \in \R$. Indeed, for $t=0$, we take $\pi_0 \in \Pi(\lambda,\mu_0)$ to be the quantile\footnote{The quantile coupling (or Fr\'echet--Hoeffding coupling) is given by the law of $(F^{-1}_{\lambda},F^{-1}_{\mu_{0}})$ under $\lambda$, where $F^{-1}_{\mu_{0}}$ is the inverse c.d.f.\ of $\mu_{0}$.}  coupling. Then, applying Proposition~\ref{prop:generalShadowPlan} to $\pi_{0}$ yields the measure $Q_{1}$, and we can define $\pi_1 := Q_1 \circ (X_0,X_2)^{-1}$. Proceeding inductively, applying Proposition~\ref{prop:generalShadowPlan} to $\pi_{t-1}$ yields $Q_{t}$ which in turn allows us to define $\pi_t := Q_t \circ (X_0,X_2)^{-1}$.

\emph{Step~2.} For $1\leq t\leq n$, consider a disintegration $Q_t = \pi_{t-1} \otimes \kappa_t$ of $Q_{t}$. \FS{By~\eqref{eqn:martingaleproperty}}, we may choose $\kappa_t(s,x_{t-1},dx_t)$ to be a martingale kernel; that is,
\[ \int x_t \, \kappa_t(s,x_{t-1},dx_t) = x_{t-1}\]
holds for all $(s,x_{t-1})\in\R^{2}$. We now define a measure $\pi \in \Pi(\lambda,\mu_0,\dots,\mu_n)$ on $\R^{n+2}$ via
\[\pi = \pi_0 \otimes \kappa_1 \otimes \dots \otimes \kappa_n.\]
Then, $\pi$ satisfies
\[ \pi \circ (X_0,X_t)^{-1} = \pi_{t-1} \quad\mbox{and}\quad \pi \circ (X_0,X_t,X_{t+1})^{-1} = Q_t\]
for $1\leq t\leq n$, and setting $P = \pi \circ (X_1,\dots,X_{n+1})^{-1}$ yields the theorem.
\end{proof}

The following result studies the bivariate projections $P_{0t}$ of a left-monotone transport and shows in particular that $P_{0t}$ may differ from the Left-Curtain coupling~\cite{BeiglbockJuillet.12} in $\M(\mu_{0},\mu_{t})$.

\begin{proposition}\label{pr:multiLeftCurtain}
  Let $\bmu=(\mu_{0},\dots,\mu_{n})$ be in convex order and let $P \in \M(\bmu)$ be a left-monotone transport. The following are equivalent:
  \begin{enumerate}
  \item The bivariate projection $P_{0t} = P \circ (X_0,X_t)^{-1}\in\M(\mu_{0},\mu_{t})$ is left-monotone for all $1\leq t\leq n$.
  \item The marginals $\bmu$ satisfy
  \begin{equation}
   \label{eqn:strongOrder}
   \shadow{\mu_0|_{(-\infty,x]}}{\mu_1} \leq_c \dots \leq_c \shadow{\mu_0|_{(-\infty,x]}}{\mu_n}
   \quad \text{for all}\quad x\in \R.
\end{equation}
  \end{enumerate}
\end{proposition}

\begin{proof}
  Given $\mu\leq\mu_{0}$, an iterative application of Lemma~\ref{lem:shadowProperties}\,(ii) shows that
  the obstructed shadows coincide with the ordinary shadows, i.e.\ $\shadow{\mu}{\mu_1,\dots,\mu_t} = \shadow{\mu}{\mu_t}$ for $1\leq t\leq n$, if and only if
$\shadow{\mu}{\mu_1} \leq_c \dots \leq_c \shadow{\mu}{\mu_n}$.
The proposition follows by applying this observation to $\mu=\mu_0|_{(-\infty,x]}$.
\end{proof}

The following example illustrates the proposition and shows that~\eqref{eqn:strongOrder} may indeed fail.

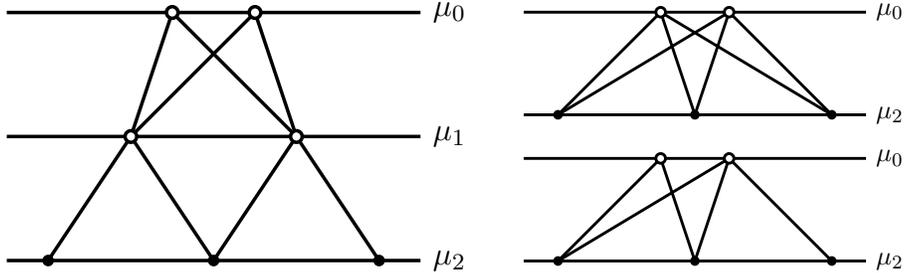
\begin{figure}
\hspace{-2.3mm}
\begin{tabular}{cc}
\raisebox{-0.46\height}{
\scalebox{1.1}{
\begin{tikzpicture}
\draw[very thick] (-2.5,3)  -- (2.5,3) node[right] {$\mu_0$};
\draw[very thick] (-2.5,1.5) -- (2.5,1.5) node[right] {$\mu_1$};
\draw[very thick] (-2.5,0) -- (2.5,0) node[right] {$\mu_2$};

\draw[very thick] (-0.5,3) -- (-1,1.5) {};
\draw[very thick] (-0.5,3) node[minimum width=4pt,inner sep=0,draw,fill=white,circle] {} -- (1,1.5) {};
\draw[very thick] (0.5,3) -- (-1,1.5) {};
\draw[very thick] (0.5,3) node[minimum width=4pt,inner sep=0,draw,fill=white,circle] {} -- (1,1.5) {};

\draw[very thick] (1,1.5) -- (2,0) node[minimum width=4pt,inner sep=0,fill,circle] {};
\draw[very thick] (1,1.5) node[minimum width=4pt,inner sep=0,draw,fill=white,circle] {} -- (0,0) {};
\draw[very thick] (-1,1.5) -- (-2,0) node[minimum width=4pt,inner sep=0,fill,circle] {};
\draw[very thick] (-1,1.5) node[minimum width=4pt,inner sep=0,draw,fill=white,circle] {} -- (0,0) node[minimum width=4pt,inner sep=0,fill,circle] {};
\end{tikzpicture}}}
& \scalebox{.91}{\begin{tabular}{c}
\begin{tikzpicture}
\draw[very thick] (-2.5,3) -- (2.5,3) node[right] {$\mu_0$};
\draw[very thick] (-2.5,1.5) -- (2.5,1.5) node[right] {$\mu_2$};

\draw[very thick] (-0.5,3) -- (-2,1.5) {};
\draw[very thick] (-0.5,3) -- (0,1.5) {};
\draw[very thick] (-0.5,3) node[minimum width=4pt,inner sep=0,draw,fill=white,circle] {} -- (2,1.5) {};
\draw[very thick] (0.5,3) -- (-2,1.5) node[minimum width=4pt,inner sep=0,fill,circle] {};
\draw[very thick] (0.5,3) -- (0,1.5) node[minimum width=4pt,inner sep=0,fill,circle] {};
\draw[very thick] (0.5,3) node[minimum width=4pt,inner sep=0,draw,fill=white,circle] {} -- (2,1.5) node[minimum width=4pt,inner sep=0,fill,circle] {};
\end{tikzpicture} \\
\begin{tikzpicture}
\draw[very thick] (-2.5,3) -- (2.5,3) node[right] {$\mu_0$};
\draw[very thick] (-2.5,1.5) -- (2.5,1.5) node[right] {$\mu_2$};

\draw[very thick] (-0.5,3) -- (-2,1.5) {};
\draw[very thick] (-0.5,3) node[minimum width=4pt,inner sep=0,draw,fill=white,circle] {} -- (0,1.5) {};
\draw[very thick] (0.5,3) -- (-2,1.5) node[minimum width=4pt,inner sep=0,fill,circle] {};
\draw[very thick] (0.5,3) -- (0,1.5) node[minimum width=4pt,inner sep=0,fill,circle] {};
\draw[very thick] (0.5,3) node[minimum width=4pt,inner sep=0,draw,fill=white,circle] {} -- (2,1.5) node[minimum width=4pt,inner sep=0,fill,circle] {};
\end{tikzpicture}
\end{tabular}}
\end{tabular}
\caption{The left panel shows the support of the left-monotone transport~$P$ from Example~\ref{ex:NotLeftCurtain}. The right panel shows the support of $P_{02}$ (top) and the support of the left-monotone transport in $\M(\mu_{0},\mu_{2})$ (bottom). The elements of the support are represented by the diagonal lines.}
\label{fig:notStrictlyLeftMonotone}
\end{figure}

\begin{example}\label{ex:NotLeftCurtain}
Consider the marginals
\[\mu_0 = \frac{1}{2}\delta_{-1} + \frac{1}{2}\delta_1, \quad
\mu_1 = \frac{1}{2}\delta_{-2} + \frac{1}{2}\delta_2, \quad 
\mu_2 = \frac{1}{4}\delta_{-4} + \frac{1}{2} \delta_0 + \frac{1}{4}\delta_4. \]
Then the set $\M(\bmu)$ consists of a single transport $P$; cf.\ the left panel of
Figure~\ref{fig:notStrictlyLeftMonotone}. Thus, $P$ is necessarily left-monotone. Similarly, $P_{01}=P \circ (X_0,X_1)^{-1}$ is the unique element of $\M(\mu_0,\mu_1)$.
However, $P_{02} = P \circ (X_0,X_2)^{-1}$ is given by
\[ \frac{3}{16} \delta_{(-1,-4)} + \frac{1}{4} \delta_{(-1,0)} + \frac{1}{16}\delta_{(-1,4)}
+ \frac{1}{16} \delta_{(1,-4)} + \frac{1}{4}\delta_{(1,0)} + \frac{3}{16}\delta_{(1,4)} \]
whereas the unique left-monotone transport in $\M(\mu_0,\mu_2)$ can be found to be
\[\frac{1}{8}\delta_{(-1,-4)} + \frac{3}{8}\delta_{(-1,0)} + \frac{1}{8}\delta_{(1,-4)}
+ \frac{1}{8}\delta_{(1,0)} + \frac{1}{4}\delta_{(1,4)}.\]
Therefore, there exists no transport $P\in\M(\bmu)$ such that
both $P_{01}$ and $P_{02}$ are left-monotone, and Proposition~\ref{pr:multiLeftCurtain} shows that~\eqref{eqn:strongOrder} fails.
\end{example}

\begin{remark}\label{rk:rightMonotone}
  Of course, all our results on left-monotone transports have ``right-monotone'' analogues, obtained by reversing the orientation on the real line \MN{(i.e.\ replacing $x\mapsto -x$ everywhere)}.
\end{remark}
\section{Geometry and Optimality Properties}\label{se:geometry}

In this section we introduce the optimality properties for transports and the geometric properties of their supports that were announced in the Introduction, and prove that they equivalently characterize left-monotone transports.

\subsection{Geometry of Optimal Transports for Reward Functions of Spence--Mirrlees Type}

The first goal is to show that optimal transports for specific reward functions are concentrated on sets $\Gamma\subseteq\R^{n+1}$ satisfying certain no-crossing conditions that we introduce next.
Given $1\leq t\leq n$, we write
\[\Gamma^t = \{ (x_0,\dots,x_t) \in \R^{t+1} : (x_0,\dots,x_n) \in \Gamma \text{ for some } (x_{t+1},\dots,x_n) \in \R^{n-t}\}\]
for the projection of $\Gamma$ onto the first $t+1$ coordinates.

\begin{definition}
\label{def:leftMonotoneSupport}
Let $\Gamma \subseteq \R^{n+1}$ and $1 \leq t \leq n$. Consider $\bx=(x_{0},\dots,x_{t-1})$, $\bx'=(x'_{0},\dots,x'_{t-1}) \in \R^t$ and $y^+,y^-,y' \in \R$ with $y^- < y^+$
such that $(\bx,y^+), (\bx,y^-), (\bx',y') \in \Gamma^t$. Then, the projection
\begin{center}
  $\Gamma^t$ is \emph{left-monotone}~~~if~~~$y' \notin (y^-,y^+)$ whenever
$x_0 < x_0'$.
\end{center}
The set $\Gamma$ is left-monotone\footnote{\label{foot:abuse}\FS{This terminology for $\Gamma$ is abusive since $\Gamma=\Gamma^{n}$ is in fact a projection itself}---it will be clear from the context what is meant.} if $\Gamma^t$ is left-monotone for all
$1 \leq t \leq n$.
\end{definition}

We also need the following notion.

\begin{definition}
\label{def:nondegenerate}
Let $\Gamma \subseteq \R^{n+1}$ and $1 \leq t \leq n$. The projection $\Gamma^t$ is \emph{nondegenerate} if for all $\bx=(x_{0},\dots,x_{t-1}) \in \R^t$ and $y \in \R$
such that $(\bx,y) \in \Gamma^t$, the following hold:
\begin{enumerate}[(i)]
\item if $y > x_{t-1}$, there exists $y' < x_{t-1}$ such that
$(\bx,y') \in \Gamma^t$;
\item if $y < x_{t-1}$, there exists $y' > x_{t-1}$ such that 
$(\bx,y') \in \Gamma^t$.
\end{enumerate}
The set $\Gamma$ is called nondegenerate\footnote{Footnote~\ref{foot:abuse} applies here as well.} if $\Gamma^t$ is nondegenerate for all $1 \leq t \leq n$.
\end{definition}

Broadly speaking, this definition says that for any path \FS{to the right} in~$\Gamma$ there exists a path \FS{to the left}, and vice versa. For a set supporting a martingale, nondegeneracy is not a restriction, in the following sense.

\begin{remark}
\label{rem:nondegenerate}
Let $\bmu$ be in convex order, $\V$ its effective domain and $\Gamma \subseteq \V$.

(i) There exists a nondegenerate, universally measurable set $\Gamma' \subseteq \Gamma$ such that $P(\Gamma') = 1$ for all $P \in \M(\bmu)$ with $P(\Gamma) = 1$.

(ii) Fix $P \in \M(\bmu)$ with $P(\Gamma) = 1$. There exists a nondegenerate, Borel-measurable set $\Gamma'_{P} \subseteq \Gamma$ such that \FS{$P(\Gamma'_P) = 1$}.
\end{remark}

\begin{proof}
Let $N_{t}$ be the set of all \FS{$\bx \in \Gamma^t$} such that~(i) or~(ii) of Definition~\ref{def:nondegenerate}
fail. If $P$ is a martingale with $P(\Gamma)=1$, we see that $N_{t} \times \R^{n-t+1}$ is $P$-null. Moreover, $N_{t}$ is universally measurable \FS{(as the projection of a Borel set)} and we can set
$$\Gamma' := \Gamma \backslash  \bigcup_{t=1}^{n}(N_{t} \times \R^{n-t+1} )$$
to prove~(i). Turning to~(ii), universal measurability implies that there exists a Borel set $N'_{t} \supseteq N_{t}$ such that $N'_{t} \backslash N_{t}$ is $P_{t-1}$-null, where $P_{t-1}=P\circ (X_{0},\dots,X_{t-1})^{-1}$.
We can then set $\Gamma'_{P} := \Gamma \backslash  \cup_{t=1}^{n}(N'_{t} \times \R^{n-t+1} )$.
\end{proof}

Next, we introduce a notion of competitors along the lines of \cite[Definition~1.10]{BeiglbockJuillet.12}.

\begin{definition}
Let $\pi$ be a finite measure on $\R^{t+1}$ whose marginals have finite first moments and consider a disintegration $\pi = \pi_t \otimes \kappa$, where $\pi_t$ is the projection of $\pi$ onto the first $t$ coordinates.
A measure $\pi' = \pi_t \otimes \kappa'$ is a \emph{$t$-competitor} of $\pi$ if it has the same last
marginal and 
\[\bary(\kappa(\bx,\cdot)) = \bary(\kappa'(\bx,\cdot))
\quad \text{for } \pi_t\text{-a.e.}\quad \bx =(x_0,\dots,x_{t-1}).\]
\end{definition}

Using these definitions, we now formulate a variant of the monotonicity principle stated
in Theorem~\ref{thm:monotonicityPrinciple} \FS{(i)} that will \MN{be} convenient to infer the geometry
of~$\Gamma$.

\begin{lemma}
\label{lem:variationalPrinciple}
Let $\bmu = (\mu_0,\dots,\mu_n)$ be in convex order, $1\leq t\leq n$ and let $\bar{f}: \R^{t+1} \to [0,\infty)$ be Borel. Consider $f(X_0,\dots,X_n) := \bar{f}(X_0,\dots,X_t)$
and suppose that $\I_\bmu(f) < \infty$. Let $(\bphi,H) \in \D_\bmu(f)$ be
an optimizer for $\I_\bmu(f)$ with the property that $\phi_{s}\equiv H_{s}\equiv 0$ for $s=t+1,\dots,n$
and define the set
\[\Gamma := \left\{\sum_{t=0}^n \phi_t(X_t) + (H \cdot X)_n = f\right\} \cap \V.\]
Let $\pi$ be a finitely supported probability on $\R^{t+1}$ which is concentrated on
$\Gamma^t$. Then $\pi(\bar{f}) \geq \pi'(\bar{f})$ for any $t$-competitor $\pi'$ of $\pi$ that is
concentrated on $\V^t$.
\end{lemma}

\begin{proof}
Recall that the projections $\pi_{t}$ and $\pi'_{t}$ onto the first $t$ coordinates coincide.
Thus,
\begin{align*}
\pi[H_t\cdot(X_t - X_{t-1})] &=
\int H_t \cdot(\bary(\kappa(X_0,\dots,X_{t-1},\cdot) - X_{t-1}) d\pi_t \\
&= \int H_t \cdot(\bary(\kappa'(X_0,\dots,X_{t-1},\cdot) - X_{t-1}) d\pi_t'\\
&= \pi'[H_t\cdot(X_t-X_{t-1})].
\end{align*}
Using also that the last marginals coincide, we deduce that
\begin{align*}
\pi[\bar{f}] &= \pi\left[\sum_{s=0}^t\phi_s(X_s) + (H \cdot X)_t\right]
= \pi'\left[\sum_{s=0}^t\phi_s(X_s) + (H \cdot X)_t\right] \geq \pi'[\bar{f}].
\end{align*}
\end{proof}

Next, we formulate an intermediate result relating optimality for Spence--Mirrlees reward functions to left-monotonicity of the support.

\begin{lemma}
\label{lem:SpenceMirrleesToLeftMonotone}
Let $1\leq t\leq n$ and let $\Gamma \subseteq \V$ be a subset such that $\Gamma^t$ is nondegenerate. Moreover, let $f: \R^{t+1} \to \R$ be of the form $f(X_0,\dots,X_t) = \bar{f}(X_0,X_t)$ for a second-order Spence--Mirrlees function $\bar{f}$. Assume that for any finitely supported probability $\pi$ that is 
concentrated on $\Gamma^t$ and any $t$-competitor $\pi'$ of $\pi$ that is concentrated on~$\V^t$, we have
$\pi(f) \geq \pi'(f)$. Then, the projection $\Gamma^t$ is left-monotone.
\end{lemma}

\begin{proof}
Consider $(\bx,y_1),(\bx,y_2),(\bx',y') \in \Gamma^t$ satisfying $x_0 < x_0'$ and suppose
for contradiction that $y_1 < y' < y_2$. We define $\lambda = \frac{y_2 - y'}{y_2 - y_1}$ and
\begin{align*}
\pi &= \frac{\lambda}{2} \delta_{(\bx,y_1)} + \frac{1-\lambda}{2} \delta_{(\bx,y_2)} + \frac{1}{2}\delta_{(\bx',y')} \\
\pi' &= \frac{\lambda}{2} \delta_{(\bx',y_1)} + \frac{1-\lambda}{2} \delta_{(\bx',y_2)} + \frac{1}{2}\delta_{(\bx,y')}.
\end{align*}
Then $\pi$ and $\pi'$ have the same projection $\pi_t=\pi'_t$ on the first $t$ marginals and their last marginals also coincide. Moreover, disintegrating
$\pi = \pi_t \otimes \kappa$ and $\pi' = \pi_t \otimes \kappa'$, the measures
$\kappa(\bx), \kappa(\bx'), \kappa'(\bx), \kappa(\bx')$ all have barycenter~$y'$. Therefore,
$\pi$ and $\pi'$ are $t$-competitors. We must also have that $\pi'$ is concentrated on
$\V^t$, by the shape of $\V$. Now our assumption implies that $\pi(f) \geq \pi'(f)$, but the second-order
Spence--Mirrlees property of $\bar{f}$ implies that $\pi(f) < \pi'(f)$.
\end{proof}

\subsection{Geometry of Left-Monotone Transports}

Next, we establish that transports with left-monotone support are indeed left-monotone in the sense of Theorem~\ref{thm:multistepleft}.

\begin{theorem}
\label{thm:leftMonotoneGeometry}
Let $\bmu=(\mu_{0},\dots,\mu_{n})$ be in convex order and let $P \in \M(\bmu)$ be concentrated on a nondegenerate, left-monotone set~$\Gamma\subseteq\R^{n+1}$. Then~$P$ is left-monotone.
\end{theorem}

Before stating the proof of the theorem, we record two auxiliary results about measures on the real line. The first one is a direct consequence of Proposition~\ref{prop:StrassenResult}.

\begin{lemma}
\label{lem:convexSupportConstraints}
Let $a < b$ and $\mu \leq_c \nu$. If $\nu$ is concentrated on
$(-\infty,a]$, then so is $\mu$, and moreover $\nu(\{a\}) \geq \mu(\{a\})$. The analogue holds
for $[b,\infty)$.
\end{lemma}

The second result is~\cite[Lemma~5.2]{BeiglbockJuillet.12}.

\begin{lemma}
\label{lem:differenceWitness}
Let $\sigma$ be a nontrivial signed measure on $\R$ with $\sigma(\R) = 0$ and
let $\sigma = \sigma^+ - \sigma^-$ be its Hahn decomposition. There exist
$a \in \supp(\sigma^+)$ and $b>a$ such that $\int (b-y)^+\1_{[a,\infty)}\, d\sigma(y) > 0$.
\end{lemma}

We can now give the proof of the theorem; \MN{it is inspired by~\cite[Theorem~5.3]{BeiglbockJuillet.12} which corresponds to the case $n=1$}.

\begin{proof}[Proof of Theorem \ref{thm:leftMonotoneGeometry}]
Since the case $n=1$ is covered by Proposition~\ref{prop:leftMonotoneTransport}, we may assume that the theorem has been proved for transports with $n-1$ steps and focus on the induction argument.

For every $x\in \R$ we denote by $\mu^t_x$ the marginal $(P|_{(-\infty,x] \times \R^n})\circ X_{t}^{-1}$.
In particular, we then have $\mu^0_x = \mu_0|_{(-\infty,x]}$ and $\mu_x^t$ is the image of
$\mu^0_x$ under~$P$ after~$t$ steps. For the sake of brevity, we also set
$\nu^t_x := \shadow{\mu^0_x}{\mu_1,\dots,\mu_t}$. By definition, $P$ is left-monotone
if $\mu^t_x = \nu^t_x$ for all $x\in \R$ and $t \leq n$, and by the induction hypothesis,
we may assume that this holds for $t \leq n-1$.

We argue by contradiction and assume that there exists $x \in \R$ such that
$\mu^n_x \neq \nu^n_x$. Then, the signed measure
\[\sigma := \nu^n_x - \mu^n_x\]
is nontrivial and we can find $a < b$ with $a \in \supp(\sigma^+)$ as in Lemma \ref{lem:differenceWitness}.
Observe that $\sigma^+ \leq \mu_n - \mu^n_x$ where $\mu_n - \mu^n_x$ is the image of
$\mu_n|_{(x,\infty)}$ under~$P$. Hence, $a \in \supp(\mu_n - \mu^n_x)$ and as $P$ is concentrated on $\Gamma$, we conclude that there exists a sequence of points
\begin{equation}
\label{eqn:witnessApproaching}
\bx^m = (x_0^m,\dots,x_n^m)\in \Gamma \quad \text{with } x < x_0^m \text{ and }x_n^m \to a.
\end{equation}
Moreover, by the characterization of the obstructed shadow in Lemma~\ref{le:shadowLeastElement}, we must have
\[\nu^n_x \leq_c \mu^n_x \]
as $\mu^n_x \in \casts{\mu^0_x}{\mu_n}^{\mu_1,\dots,\mu_{n-1}}$ due to the fact that $\mu^n_x$ is the image of $\mu^0_x$ under a martingale transport.

\emph{Step~1.} We claim that for all $\bx = (x_0,\dots,x_{n-1})$ with $x_0 \leq x$ and
$x_{n-1} \leq a$, it holds that
\[\Gamma_{\bx} \cap (a,\infty) = \emptyset,\]
where $\Gamma_{\bx}=\{y\in \R:\, (\bx,y)\in \Gamma\}$ is the section of $\Gamma$ at $\bx$.
By way of contradiction, assume that for some $\bx$ with $x_0 \leq x$ and $x_{n-1} \leq a$ we have
$\Gamma_{\bx} \cap (a,\infty) \neq \emptyset$, then in particular
$\Gamma_\bx \cap (x_{n-1},\infty) \neq \emptyset$. In view of the nondegeneracy of $\Gamma$, we conclude that $\Gamma_\bx \cap (-\infty,x_{n-1}) \neq \emptyset$ and hence that
$\Gamma_{\bx} \cap (-\infty,a) \neq \emptyset$. This yields a contradiction to the
left-monotonicity of~$\Gamma$ by using~$\bx^m$ from \eqref{eqn:witnessApproaching}
for $\bx'$ in Definition \ref{def:leftMonotoneSupport} for large enough $m$, and the proof of the claim is complete.

\emph{Step~2.} Similarly, we can show that for all $\bx = (x_0,\dots,x_{n-1})$ with $x_0 \leq x$
and $x_{n-1} \geq a$,
\[\Gamma_{\bx} \cap (-\infty,a) = \emptyset.\]

\emph{Step~3.} Next, we consider the marginals 
$$
  \mu^t_{x,a} := \big(P|_{(-\infty,x]\times \R^{n-2} \times (-\infty,a] \times \R} \big) \circ X_{t}^{-1}.
$$
Then, in particular, $\mu^{n-1}_{x,a} = \mu^{n-1}_x |_{(-\infty,a]}$ and
$\mu^n_{x,a}$ is the image of $\mu^{n-1}_{x,a}$ under the last step of $P$. Step~1 of the proof thus implies that $\mu^n_{x,a}$ is concentrated on $(-\infty,a]$.
We also write 
\[\FS{\nu^n_{x,a} := \shadow{\mu^{n-1}_x|_{(-\infty,a]}}{\mu_n}.}\]
We have $\mu^{n-1}_{x,a} \leq_c \mu^n_{x,a}$ as $\M(\mu^{n-1}_{x,a},\mu^n_{x,a}) \neq \emptyset$, and $\mu^n_{x,a} \leq \mu^n_x\FS{ \leq \mu_n}$. Therefore,
\begin{equation}
\label{eqn:leftPartOrder}
\nu^n_{x,a} \leq_c \mu^n_{x,a}
\end{equation}
by the minimality of the shadow.
Next, we show that
\begin{equation}
\label{eqn:rightPartOrder}
\nu^n_x - \nu^n_{x,a} \leq_c \mu^n_x - \mu^n_{x,a}.
\end{equation}
Observe that $\mu^n_x - \mu^n_{x,a}$ is the image of $\mu^{n-1}_x|_{(a,\infty)}$ under $P$ and therefore
concentrated on $[a,\infty)$ by Step~2. Using this observation, that $\mu^n_{x,a}$ is concentrated on $(-\infty,a]$ as mentioned above, and the fact that $\nu^n_{x,a}(\{a\}) \leq \mu^n_{x,a}(\{a\})$
 as a consequence of~\eqref{eqn:leftPartOrder} and Lemma~\ref{lem:convexSupportConstraints},
we have
\[\mu^n_x - \mu^n_{x,a} = (\mu^n_x - \mu^n_{x,a})|_{[a,\infty)}
\leq (\mu_n - \mu^n_{x,a})|_{[a,\infty)} \leq (\mu_n - \nu^n_{x,a})|_{[a,\infty)}
\leq \mu_n - \nu^n_{x,a}.\]
We also have $\mu^{n-1}_x|_{(a,\infty)} \leq_c \mu^n_x - \mu^n_{x,a}$ since the latter measure is the image of the former under $P$. Together with the preceding display, we have established that
\[\mu^n_x - \mu^n_{x,a} \in \casts{\mu^{n-1}_x|_{(a,\infty)}}{\mu_n - \nu^n_{x,a}}.\]
On the other hand,
\[\nu^n_x - \nu^n_{x,a}
= \shadow{\mu^{n-1}_x|_{(a,\infty)}}{\mu_n - \nu^n_{x,a}}\]
from the additivity property of the
shadow in Lemma~\ref{lem:shadowProperties}\,(i),
and therefore~\eqref{eqn:rightPartOrder} follows by the minimality of the shadow.

\emph{Step~4.} Recall from Step~3 that $\mu^n_{x,a}$ is concentrated on $(-\infty,a]$ and that
$\mu^n_x - \mu^n_{x,a}$ is concentrated on $[a,\infty)$. Therefore,
$\nu^n_{x,a}$ is concentrated on $(-\infty,a]$ and
$\nu^n_x - \nu^n_{x,a}$ is concentrated on $[a,\infty)$, by Lemma~\ref{lem:convexSupportConstraints}.
Moreover, we have
$\nu^n_{x,a}(\{a\}) \leq \mu^n_{x,a}(\{a\})$ by the same lemma, and finally, the function
$y \mapsto (b-y)^+\1_{[a,\infty)}(y)$ is convex on $[a,\infty)$ as $a<b$. Using these facts and~\eqref{eqn:rightPartOrder},
\begin{align*}
\int (b-y)^+&\1_{[a,\infty)}(y) \nu^n_x(dy) \\
&= \int (b-y)^+\1_{[a,\infty)}(y) (\nu^n_x-\nu^n_{x,a})(dy) + (b-a) \nu^n_{x,a}(\{a\}) \\
&\leq \int (b-y)^+\1_{[a,\infty)}(y) (\mu^n_x-\mu^n_{x,a})(dy) + (b-a) \mu^n_{x,a}(\{a\}) \\
&= \int (b-y)^+\1_{[a,\infty)}(y) \mu^n_x(dy).
\end{align*}
This contradicts the choice of $a$ and $b$, cf.\ Lemma~\ref{lem:differenceWitness}, and thus completes the proof.
\end{proof}

\subsection{Optimality Properties}

In this section we relate left-monotone transports and left-monotone sets to the optimal transport problem for Spence--Mirrlees functions.

\begin{theorem}
\label{thm:optimizerLeftMonotonicity}
For $1 \leq t \leq n$, let $f_t: \R^2 \to \R$  be second-order Spence--Mirrlees functions
such that $|f_t(x,y)| \leq a_{0}(x) + a_{t}(y)$ for some $a_{0} \in L^1(\mu_0)$ and $a_{t} \in L^1(\mu_t)$. There exists a universally measurable, nondegenerate, left-monotone set $\Gamma'\subseteq \R^{n+1}$ such that any simultaneous optimizer $P\in \M(\bmu)$ for $\S_\bmu(f_t(X_0,X_t)),$ $1 \leq t \leq n$ is concentrated on $\Gamma'$. In particular, any such $P$ is left-monotone.
\end{theorem}

\begin{proof}
The last assertion follows by an application of Theorem~\ref{thm:leftMonotoneGeometry}, so we may focus on finding $\Gamma'$. For each $1 \leq t \leq n$, we use Theorem~\ref{thm:duality} and Remark~\ref{rk:relaxLowerBound} to find a dual optimizer $(\bphi,H) \in \D_\bmu(\FS{f_t})$ for $\I_\bmu(f_t(X_0,X_t))$
and define the Borel set
\[\Gamma(t) := \left\{\sum_{s=0}^n \phi_s(X_s) + (H \cdot X)_n = \FS{f_t}\right\} \cap \V.\]
Here, we may choose a dual optimizer such that $\phi_{s}\equiv H_{s}\equiv 0$ for $s=t+1,\dots,n$. (This can be seen by applying Theorem~\ref{thm:duality} to the transport problem involving only the marginals $(\mu_{0},\dots,\mu_{t})$ and taking the corresponding dual optimizer.) Theorem~\ref{thm:monotonicityPrinciple} shows that any simultaneous optimizer $P\in \M(\bmu)$ is concentrated on $\Gamma(t)$ for all $t$, and hence also on the Borel set
$$
  \Gamma:=\bigcap_{t=1}^{n}\Gamma(t).
$$
Using Remark~\ref{rem:nondegenerate}\,(i), we find a universally measurable, nondegenerate subset $\Gamma'\subseteq \Gamma$ with the same property. Since the projection $(\Gamma')^{t}$ is contained in the projection $(\Gamma(t))^{t}$, Lemma~\ref{lem:variationalPrinciple} and Lemma~\ref{lem:SpenceMirrleesToLeftMonotone} yield that $(\Gamma')^{t}$ is left-monotone for all $t$; that is, $\Gamma'$ is left-monotone.
\end{proof}

\begin{remark}\label{rk:BorelVersionOfGamma}
  In Theorem~\ref{thm:optimizerLeftMonotonicity}, if we only wish to find a nondegenerate, left-monotone set $\Gamma'_{P}\subseteq \R^{n+1}$ such that a given simultaneous optimizer $P\in \M(\bmu)$ is concentrated on $\Gamma'_{P}$, then we may choose $\Gamma'_{P}$ to be Borel instead of universally measurable. This follows by replacing the application of Remark~\ref{rem:nondegenerate}\,(i) by Remark~\ref{rem:nondegenerate}\,(ii) in the proof.
\end{remark}

The following is a converse to Theorem~\ref{thm:optimizerLeftMonotonicity}.

\begin{theorem}\label{th:leftMonotoneOptimalitySpenceMirrlees}
  Given $1\leq t\leq n$, let $f\in C^{1,2}(\R^{2})$ be such that $f_{xyy}\geq0$ and suppose that the following integrability condition holds:
  \begin{equation}\label{eq:smoothSMint}
  \begin{cases}
    f(X_{0},X_{t}), \quad f(0,X_{t}),  \quad f(X_{0},0), \quad \bar{h}(X_{0})X_{0}, \quad \bar{h}(X_{0})X_{t}\\
      \mbox{are $P$-integrable for all}\;\; P\in\M(\bmu),
  \end{cases}
  \end{equation}
  where $\bar{h}(x):=\partial_{y}|_{y=0} [f(x,y) - f(0,y)]$. 
  Then every left-monotone transport $P\in \M(\bmu)$ is an optimizer for $\S_\bmu(f)$. 
\end{theorem}

The integrability condition clearly holds when $f$ is Lipschitz continuous; in particular, a smooth second-order Spence--Mirrlees function (as defined in the Introduction) satisfies the assumptions of the theorem for any $\bmu$.

The proof will be given by an approximation based on the following building blocks for Spence--Mirrlees functions; the construction is novel and may be of independent interest.

\begin{lemma}\label{le:optimalityProductFun}
Let $1 \leq t \leq n$ and let $f(X_0,\dots,X_n) := \1_{(-\infty,a]}(X_0)\varphi(X_t)$ for a concave function $\varphi$ and $a\in\R$. Then every left-monotone transport $P\in \M(\bmu)$ is an optimizer for $\S_\bmu(f)$.
\end{lemma}

\begin{proof}
  In view of Lemma~\ref{le:shadowLeastElement}, this follows directly by applying the defining shadow property from Theorem~\ref{thm:multistepleft} with $x=a$.
\end{proof}

The integrability condition~\eqref{eq:smoothSMint} implies that setting
$$
  g(x,y):=f(x,0)+f(0,y)-f(0,0)+\bar{h}(x)y,
$$
the three terms constituting
$$
  g(X_{0},X_{t}) = [f(X_{0},0)+\bar{h}(X_{0})X_{0}] +  [f(0,X_{t})-f(0,0)] + [\bar{h}(X_{0})(X_{t}-X_{0})]
$$
are $P$-integrable and $P[g(X_{0},X_{t})]$ is constant over $P\in\M(\bmu)$.
By replacing $f$ with $f-g$, we may thus assume without loss of generality that
\begin{equation}\label{eq:normalized}
  f(x,0)=f(0,y)= f_{y}(x,0)=0 \quad \mbox{for all}\quad  (x,y)\in\R^{2}.
\end{equation}
After this normalization, integration by parts yields the representation
\begin{equation}\label{eq:intByParts}
  f(x,y) = \int_0^y \int_0^x (y-t) f_{xyy}(s,t)\,ds\,dt.
\end{equation}

\begin{lemma}\label{le:optimalityRestrictedFun}
  Theorem~\ref{th:leftMonotoneOptimalitySpenceMirrlees} holds under the following additional condition: there exists a constant $c>0$ such that 
  \begin{align*}
  &\mbox{$x\mapsto f(x,y)$ is constant on $\{x>c\}$ and on $\{x<-c\}$},\\
  &\mbox{$y\mapsto f(x,y)$ is affine on $\{y>c\}$ and on $\{y<-c\}$.}
  \end{align*}
\end{lemma}

\begin{proof}
  \MN{Integration by parts implies that} for all $(x,y)\in\R^{2}$, we have the representation
  \begin{align*}
  f(x,y) = \,&- \int_{-c}^c \int_{-c}^c \1_{(-\infty,s]}(x)(y-t)^+f_{xyy}(s,t)\,ds\,dt \\
  & + [f(x,-c)-(-c)f_y(x,-c)] \\
  & + [f(c,y) - f(c,-c) - f_y(c,-c)(y - (-c))] \\
  & + f_y(x,-c)y.
\end{align*} 
The last three terms are of the form $g(x,y)=\tilde{\phi}(x) + \tilde{\psi}(y)+\tilde{h}(x)y$ and of linear growth due to the additional condition. Hence, as above, $P'[g(X_{0},X_{t})]=C$ is constant for $P'\in\M(\bmu)$. If $P\in\M(\bmu)$ is left-monotone and $P'\in\M(\bmu)$ is arbitrary, Fubini's theorem and Lemma~\ref{le:optimalityProductFun} yield that
  \begin{align*}
 P[f] &= - \int_{-c}^c \int_{-c}^c P[\1_{(-\infty,s]}(x)(y-t)^+]f_{xyy}(s,t) \,ds\,dt + C \\
      & \geq - \int_{-c}^c \int_{-c}^c P'[\1_{(-\infty,s]}(x)(y-t)^+]f_{xyy}(s,t) \,ds\,dt + C\\
      &= P'[f],
  \end{align*}
where $P,P'$ are understood to integrate with respect to $(x,y)$ and the application of Fubini's theorem is justified by the nonnegativity of the integrand.
\end{proof}

\begin{proof}[Proof of Theorem~\ref{th:leftMonotoneOptimalitySpenceMirrlees}]
  Let $f$ be as in the theorem. We shall construct functions $f^{m}$, $m\geq1$ satisfying the assumption of Lemma~\ref{le:optimalityRestrictedFun} as well as $P[f^{m}]\to P[f]$ for all $P\in\M(\bmu)$. Once this is achieved, the theorem follows from the lemma.
  
  Indeed, we may assume that $f$ is normalized as in~\eqref{eq:normalized}. Let $m\geq 1$ and let $\rho_{m}:\R\to[0,1]$ be a smooth function such that $\rho_{m}=1$ on $[-m,m]$ and $\rho_{m}=0$ on $[-m-1,m+1]^{c}$. In view of~\eqref{eq:intByParts}, we define $f^{m}$ by
  $$
    f^{m}(x,y) = \int_0^y \int_0^x (y-t) f_{xyy}(s,t) \rho_{m}(s)\rho_{m}(t)\,ds\,dt.
  $$
  It then follows that $f^{m}$ satisfies the assumptions of Lemma~\ref{le:optimalityRestrictedFun} with the constant $c=m+1$. Moreover, we have 
  $$
   0\leq f^{m}(x,y)\leq f^{m+1}(x,y) \leq f(x,y)\quad\mbox{for}\quad x\geq 0
  $$
  and the opposite inequalities for $x\leq 0$, as well as $f^{m}(x,y)\to f(x,y)$ for all~$(x,y)$.
  
  Let $P\in\M(\bmu)$. Since $f$ is $P$-integrable, applying monotone convergence separately on $\{x\geq0\}$ and $\{x\leq0\}$ yields that $P[f^{m}]\to P[f]$, and the proof is complete.
\end{proof}

\begin{remark}\label{rk:smoothSMexample}
The function
$$
  \bar{f}(x,y) := \tanh(x)\sqrt{1+y^2}
$$
satisfies the conditions of Theorem~\ref{th:leftMonotoneOptimalitySpenceMirrlees} for all marginals $\bmu$ in convex order, since the latter are assumed to have a finite first moment.
\end{remark}

We can now collect the preceding results to obtain, in particular, the equivalences stated in Theorem~\ref{th:leftMonotoneIntro}.

\begin{theorem}\label{th:leftMonotoneMain}
Let $\bmu = (\mu_0,\dots,\mu_n)$ be in convex order. There exists a left-monotone, nondegenerate, universally measurable set $\Gamma\subseteq\R^{n+1}$ such that for any $P \in \M(\bmu)$, the following are equivalent:
\begin{enumerate}
	\item $P$ is an optimizer for $\S_{\bmu}(f(X_0,X_t))$ whenever $f$ is a smooth second-order Spence--Mirrlees function and $1\leq t \leq n$,
	\item $P$ is concentrated on $\Gamma$,
	\item[(ii')] $P$ is concentrated on a left-monotone set,
	\item $P$ is left-monotone; i.e.\ $P_{0t}$ transports $\mu_0|_{(-\infty,a]}$ to 
	$\shadow{\mu_0|_{(-\infty,a]}}{\mu_1,\dots,\mu_t}$ for all $1 \leq t \leq n$
	and $a \in \R$.
\end{enumerate}
Moreover, there exists $P \in \M(\bmu)$ satisfying (i)--(iii).
\end{theorem}

\begin{proof}
  Let $\Gamma$ be the set provided by Theorem~\ref{thm:optimizerLeftMonotonicity} for the function $f_{t}=\bar f$ of Remark~\ref{rk:smoothSMexample}. Given $P\in \M(\bmu)$, Theorem~\ref{thm:optimizerLeftMonotonicity} shows that~(i) implies~(ii) which trivially implies (ii'). Theorem~\ref{thm:leftMonotoneGeometry} and Remark~\ref{rem:nondegenerate} show that~(ii') implies~(iii),
  and Theorem~\ref{th:leftMonotoneOptimalitySpenceMirrlees} shows that~(iii) implies~(i). Finally, the existence of a left-monotone transport was stated in Theorem~\ref{thm:multistepleft}.
\end{proof}

We conclude this section with an example showing that left-monotone transports are not Markovian in general, even if they are unique and~\eqref{eqn:strongOrder} holds for $\bmu$.

\begin{example}\label{ex:NotMarkovian} 
  Consider the marginals
  $$
    \mu_0 = \frac{1}{2}\delta_0+\frac{1}{2}\delta_1,\quad \mu_1 = \frac{3}{4}\delta_{0} + \frac{1}{4}\delta_2,\quad \mu_2 = \frac{1}{8} \delta_{-1} + \frac{1}{2}\delta_0 + \frac{1}{8}\delta_1+\frac{1}{4}\delta_2.
  $$
  The transport $P\in\M(\bmu)$ given by
  $$
    P= \frac{1}{2}\delta_{(0,0,0)}+\frac{1}{8}\delta_{(1,0,-1)} + \frac{1}{8}\delta_{(1,0,1)}+\frac{1}{4}\delta_{(1,2,2)}
  $$
  is left-monotone because its support is left-monotone (Figure~\ref{fig:nonMarkov}), and it is clearly not Markovian. On the other hand, it is not hard to see that this is the only way to build a left-monotone transport in $\M(\bmu)$.

\begin{figure}[b]
\begin{center}
\begin{tabular}{c}
\scalebox{1.1}{
\begin{tikzpicture}
\draw[very thick] (-2,3) -- (3,3) node[right] {$\mu_0$};
\draw[very thick] (-2,1.5) -- (3,1.5) node[right] {$\mu_1$};
\draw[very thick] (-2,0) -- (3,0) node[right] {$\mu_2$};

\draw[very thick] (0,3) node[minimum width=4pt,inner sep=0,draw,fill=white,circle] {} -- (0,1.5) node[minimum width=4pt,inner sep=0,draw,fill=white,circle] {};
\draw[very thick] (1,3) node[minimum width=4pt,inner sep=0,draw,fill=white,circle] {} -- (0,1.5) node[minimum width=4pt,inner sep=0,draw,fill=white,circle] {};
\draw[very thick] (1,3) node[minimum width=4pt,inner sep=0,draw,fill=white,circle] {} -- (2,1.5) node[minimum width=4pt,inner sep=0,draw,fill=white,circle] {};
\draw[very thick] (0,1.5) node[minimum width=4pt,inner sep=0,draw,fill=white,circle] {} -- (-1,0) node[minimum width=4pt,inner sep=0,draw,fill,circle] {};
\draw[very thick] (0,1.5) node[minimum width=4pt,inner sep=0,draw,fill=white,circle] {} -- (0,0) node[minimum width=4pt,inner sep=0,draw,fill,circle] {};
\draw[very thick] (0,1.5) node[minimum width=4pt,inner sep=0,draw,fill=white,circle] {} -- (1,0) node[minimum width=4pt,inner sep=0,draw,fill,circle] {};
\draw[very thick] (2,1.5) node[minimum width=4pt,inner sep=0,draw,fill=white,circle] {} -- (2,0) node[minimum width=4pt,inner sep=0,draw,fill,circle] {};
\end{tikzpicture}}
\end{tabular}
\end{center}
\caption{Support of the non-Markovian transport in Example~\ref{ex:NotMarkovian}.}
\label{fig:nonMarkov}
\end{figure}
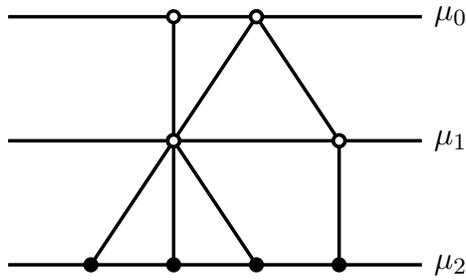
\end{example}


\section{Uniqueness of Left-Monotone Transports}\label{se:uniqueness}

In this section we consider the (non-)uniqueness of left-monotone transports. It turns out the presence of atoms in $\mu_{0}$ is important in this respect---let us start with the following simple observation.

\begin{remark}\label{rk:nonUniqueDirac}
  Let $\bmu=(\mu_{0},\dots,\mu_{n})$ be in convex order. If $\mu_{0}$ is a Dirac mass, then every $P\in\M(\bmu)$ is left-monotone. Indeed, $\M(\mu_{0},\mu_{t})$ is a singleton for every $1\leq t\leq n$, hence $P_{0t}$ must be the (one-step) left-monotone transport.
\end{remark}

Exploiting this observation, the following shows that left-monotone transports need not be unique when $n\geq2$.

\begin{example}
Let $\mu_0 = \delta_0$, $\mu_1 = \frac{1}{2}\delta_{-1} + \frac{1}{2}\delta_1$,
$\mu_2 = \frac{3}{8} \delta_{-2} + \frac{1}{4}\delta_0 + \frac{3}{8}\delta_2$.
By the remark, any element in $\M(\bmu)$ is left-monotone. Moreover, $\M(\bmu)$ is a continuum since  $\M(\mu_1,\mu_2)$ contains the
convex hull of the two measures
\begin{align*}
P_l &= \frac{1}{4}\delta_{(-1,-2)}+\frac{1}{4}\delta_{(-1,0)} + \frac{1}{8}\delta_{(1,-2)}+\frac{3}{8}\delta_{(1,2)},  \\
P_r &= \frac{3}{8}\delta_{(-1,-2)}+\frac{1}{8}\delta_{(-1,2)} + \frac{1}{4}\delta_{(1,0)}+\frac{1}{4}\delta_{(1,2)}.
\end{align*}
The corresponding supports are depicted in Figure~\ref{fig:nonunique}. 
\end{example}

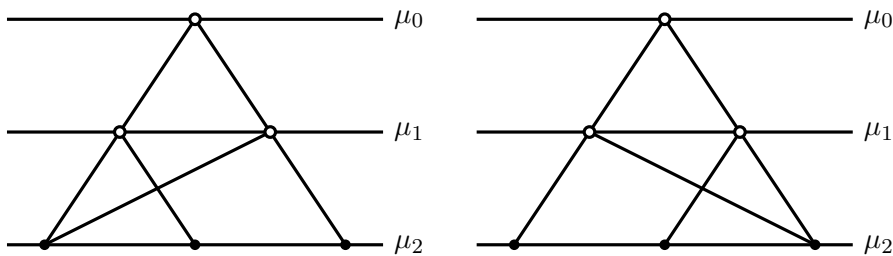
\begin{figure}[h]
\begin{center}
\begin{tabular}{cc}
\scalebox{1}{
\begin{tikzpicture}
\draw[very thick] (-2.5,3) -- (2.5,3) node[right] {$\mu_0$};
\draw[very thick] (-2.5,1.5) -- (2.5,1.5) node[right] {$\mu_1$};
\draw[very thick] (-2.5,0) -- (2.5,0) node[right] {$\mu_2$};

\draw[very thick] (0,3) -- (-1,1.5) {};
\draw[very thick] (0,3) node[minimum width=4pt,inner sep=0,draw,fill=white,circle] {} -- (1,1.5) {};
\draw[very thick] (-1,1.5) -- (-2,0) node[minimum width=4pt,inner sep=0,fill,circle] {};
\draw[very thick] (-1,1.5) node[minimum width=4pt,inner sep=0,draw,fill=white,circle] {} -- (0,0) node[minimum width=4pt,inner sep=0,fill,circle] {};
\draw[very thick] (1,1.5) -- (-2,0) {};
\draw[very thick] (1,1.5) node[minimum width=4pt,inner sep=0,draw,fill=white,circle] {} -- (2,0) node[minimum width=4pt,inner sep=0,fill,circle] {};
\end{tikzpicture}}
&
\scalebox{1}{
\begin{tikzpicture}
\draw[very thick] (-2.5,3) -- (2.5,3) node[right] {$\mu_0$};
\draw[very thick] (-2.5,1.5) -- (2.5,1.5) node[right] {$\mu_1$};
\draw[very thick] (-2.5,0) -- (2.5,0) node[right] {$\mu_2$};

\draw[very thick] (0,3) -- (-1,1.5) {};
\draw[very thick] (0,3) node[minimum width=4pt,inner sep=0,draw,fill=white,circle] {} -- (1,1.5) {};
\draw[very thick] (-1,1.5) -- (-2,0) node[minimum width=4pt,inner sep=0,fill,circle] {};
\draw[very thick] (-1,1.5) node [minimum width=4pt,inner sep=0,draw,fill=white,circle] {} -- (2,0) node[minimum width=4pt,inner sep=0,fill,circle] {};
\draw[very thick] (1,1.5) -- (0,0) node[minimum width=4pt,inner sep=0,fill,circle] {};
\draw[very thick] (1,1.5) node [minimum width=4pt,inner sep=0,draw,fill=white,circle] {} -- (2,0) {};
\end{tikzpicture}}
\end{tabular}
\end{center}
\caption{Supports of two left-monotone transports for the same marginals.}
\label{fig:nonunique}
\end{figure}

The example illustrates that non-uniqueness can typically be expected when $\mu_{0}$ has atoms. On the other hand, we have the following uniqueness result.

\begin{theorem}
\label{thm:Uniqueness}
Let $\bmu=(\mu_{0},\dots,\mu_{n})$ be in convex order. If $\mu_0$ is atomless, there exists a unique left-monotone transport $P\in\M(\bmu)$.
\end{theorem}

The remainder of this section is devoted to the proof.  Let us call a kernel $\kappa(x,dy)$ \emph{binomial} if for all $x \in \R$, the measure $\kappa(x,dy)$ consists of (at most) two point masses. A martingale transport will be called binomial if it can be disintegrated using only binomial kernels. We shall show that when $\mu_{0}$ is atomless, any left-monotone transport is a binomial martingale, and then conclude the uniqueness via a convexity argument.

The first step is the following set-theoretic result.

\begin{lemma}
\label{lem:graphAccumulation}
Let $k\geq 1$ be an integer and $\Gamma \subseteq \R^{t+1}$. For $\bx\in \R^t$, we denote by
$\Gamma_\bx := \{ y \in \R : (\bx,y) \in \Gamma \}$ the section at $\bx$. If the set
\[\{\bx \in \R^t : |\Gamma_\bx| \geq k\}\]
is uncountable, then it has an accumulation point. More precisely, there are $\bx=(x_{0},\dots,x_{t}) \in \R^t$ and $y_1 < \dots < y_k$ in $\Gamma_{\bx}$ such that for all $\epsilon > 0$ there exist $\bx'=(x'_{0},\dots,x'_{t}) \in \R^t$ and $y_1' < \dots < y'_k$ in $\Gamma_{\bx'}$ satisfying
\begin{enumerate}[(i)]
	\item $\|\bx - \bx'\| < \epsilon$,
	\item $x_0 < x_0'$,
	\item $\max_{i=1,\dots,k} |y_i - y_i'| < \epsilon$. 
\end{enumerate}
\end{lemma}

\begin{proof}
The proof is similar to the one of \cite[Lemma 3.2]{BeiglbockJuillet.12} and therefore omitted.
\end{proof}

The following statement on the binomial structure generalizes a result of~\cite{BeiglbockJuillet.12} for the one-step case and is of independent interest.

\begin{proposition}
\label{prop:leftMonotoneBinomial}
Let $\bmu=(\mu_{0},\dots,\mu_{n})$ be in convex order and let $\mu_0$ be atomless.
There exists a universally measurable set $\Gamma \subseteq \R^{n+1}$ such that every left-monotone transport $P \in \M(\bmu)$ is concentrated on $\Gamma$ and such that for all $1\leq t\leq n$ and $\bx \in \R^t$,
\begin{equation} 
\label{eqn:SuperBinomial}
|\{ y\in\R : (X_0,\dots,X_t)^{-1}(\bx,y) \cap \Gamma \neq \emptyset\}| \leq 2.
\end{equation}
In particular, every left-monotone transport $P \in \M(\bmu)$ is a binomial martingale.
\end{proposition}

\begin{proof}
Let $\Gamma$ be as in Theorem~\ref{th:leftMonotoneMain}; then every left-monotone $P \in \M(\bmu)$ is concentrated on $\Gamma$. Let $A_{t}$ be the set of all $\bx \in \R^t$ such that
\eqref{eqn:SuperBinomial} fails. Suppose that $A_{t}$ is uncountable; then Lemma~\ref{lem:graphAccumulation} yields points $\bx,\bx'$ such that
for some $y_1,y_2 \in \Gamma^{t}_\bx$ and $y \in \Gamma^{t}_{\bx'}$ we have
$y_1 < y < y_2$. This contradicts the left-monotonicity of $\Gamma$ (Definition~\ref{def:leftMonotoneSupport}), thus $A_{t}$ must be countable. Hence, $(X_{0},\dots,X_{t-1})^{-1}(A_{t})$ is Borel and $P$-null for all $P\in\M(\bmu)$, \MN{as $\mu_{0}$ is atomless}. The set $\Gamma'=\Gamma\setminus \cup_{t=1}^{n} (X_{0},\dots,X_{t-1})^{-1}(A_{t})$ then has the required properties.
\end{proof}

\begin{proof}[Proof of Theorem~\ref{thm:Uniqueness}]
We will prove this result using induction on $n$. For $n=1$ the result holds by Proposition~\ref{prop:leftMonotoneTransport}, with or without atoms.  To show the induction step, let $P'$ be the unique left-monotone transport in $\M(\mu_0,\dots,\mu_{n-1})$ and let $P_1 = P' \otimes \kappa_1$ and
$P_2 = P' \otimes \kappa_2$ be disintegrations of two $n$-step left-monotone transports. Then, 
\[\frac{P_1 + P_2}{2} = P' \otimes \frac{\kappa_1 + \kappa_2}{2}\]
is again left-monotone, and Proposition~\ref{prop:leftMonotoneBinomial} yields that $(\kappa_1 + \kappa_2)/2$
must be a binomial kernel $P'$-a.s. Using also the martingale property of $\kappa_{1}$ and $\kappa_{2}$, this can only be true if $\kappa_1 = \kappa_2$ holds $P'$-a.s., and therefore $P_1 = P_2$.
\end{proof}

\section{Free Intermediate Marginals}\label{se:freeIntermediate}

In this section we discuss a variant of our transport problem where the intermediate marginal constraints $\mu_{1},\dots,\mu_{n-1}$ are omitted; that is, only the first and last marginals $\mu_{0},\mu_{n}$ are prescribed. (One could similarly adapt the results to a case where some, but not all of the intermediate marginals are given.) 

The primal space will be denoted by $\M^n(\mu_0,\mu_n)$ and consists of all martingale measures $P$ on 
$\R^{n+1}$ such that $\mu_0 = P \circ (X_0)^{-1}$ and $\mu_n = P \circ (X_n)^{-1}$.
To make the connection with the previous sections, we note that 
\[\M^n(\mu_0,\mu_n) = \bigcup \M(\bmu)\]
where the union is taken over all vectors $\bmu=(\mu_{0},\mu_{1},\dots,\mu_{n-1},\mu_{n})$ in convex order.

\subsection{Polar Structure}

We first characterize the polar sets of $\M^n(\mu_0,\mu_n)$. To that end, we introduce an analogue of the irreducible components.

\begin{definition}\label{de:nStepComponent}
  Let $\mu_0 \leq_c \mu_n$ and let $(I_k,J_k)\subseteq \R^{2}$ be the corresponding irreducible domains in the sense of Proposition~\ref{prop:oneStepDecomposition}. The \emph{$n$-step components} of $\M^n(\mu_0,\mu_n)$ are the sets\footnote{A superscript $m$ indicates the $m$-fold Cartesian product; $\Delta_n$ is the diagonal in $\R^{n+1}$.}
  \begin{enumerate}
  \item $I_k^n \times J_k$, where $k\geq 1$,
  \item $I_0^{n+1} \cap \Delta_n$,
  \item $I_k^t \times \{p\}^{n-t+1}$, where $p\in J_{k}\setminus I_{k}$ and $1\leq t\leq n$, $k\geq 1$.
  \end{enumerate}
\end{definition}

The characterization then takes the following form.

\begin{theorem}[Polar Structure]
\label{thm:polarstructFree}
Let $\mu_0 \leq_c \mu_n$. A Borel set $B \subseteq \R^{n+1}$ is $\M^n(\mu_0,\mu_n)$-polar if
and only if there exist a $\mu_0$-nullset~$N_0$ and a $\mu_n$-nullset~$N_n$ such that
\[B \;\;\subseteq\;\; (N_0 \times \R^n) \;\cup\; (\R^n \times N_n) \;\cup\; \left(\bigcup V_j\right)^c\]
where the union runs over all $n$-step components $V_j$ of $\M^n(\mu_0,\mu_n)$.
\end{theorem}

It turns out that our previous results can be put to work to prove the theorem, by means of the following lemma which may be of independent interest.

\begin{lemma}
\label{lem:intermediateDominator}
Let $\mu \leq_c \nu$ be irreducible with domain $(I,J)$ and let $\rho$ be a probability
concentrated on $J$. Then, there exists a probability $\mu \leq_c \theta \leq_c \nu$ satisfying
$\theta \gg\rho$ such that $\mu \leq_c \theta$ and $\theta|_{I} \leq_c (\nu-\theta|_{J\setminus I})$ are both irreducible.
\end{lemma}

\begin{proof}
\emph{Step~1.} We first assume that $\rho = \delta_x$ for some $x \in J$ and show that there exists $\theta$ satisfying
\[\mu \leq_c \theta \leq_c \nu \quad \text{and} \quad \theta \gg \delta_x.\]
 If $\nu$ has an atom at $x$, we can choose $\theta = \nu$. Thus, we may assume that $\nu(\{x\})=0$ and in particular that $x\in I$.
Let $a$ be the common barycenter of $\mu$ and $\nu$ and suppose that $x < a$. For all $b\in\R$ and $0\leq c\leq \nu(\{b\})$, the measure 
\[\nu_{b,c}:=\nu|_{(-\infty,b)} + c \delta_b\]
satisfies $\nu_{b,c}\leq \nu$, and as $x<a$ there are unique $b,c$ such that $\bary(\nu_{b,c})=x$. Setting $\alpha=\nu_{b,c}$ and $\epsilon_0=\alpha(\R)$, we then have
$\epsilon_0\delta_x \leq_c \alpha\leq \nu$, and a similar construction yields this result for $x\geq a$. The existence of such $\alpha$ implies that
\[
  \epsilon  \delta_x \leq_{pc} \nu,\quad 0\leq \epsilon\leq \epsilon_0
\]
and thus the shadow $\shadow{\epsilon \delta_x}{\nu}$ is well-defined. This measure is given by the restriction of $\nu$ to an interval (possibly including fractions of atoms at the endpoints); cf.\ \cite[Example 4.7]{BeiglbockJuillet.12}. Moreover, the interval is bounded after possibly reducing the mass $\epsilon_{0}$. 
Thus, for all $\epsilon < \epsilon_0$, the difference of potential functions
\[u_{\shadow{\epsilon \delta_x}{\nu}} - u_{\epsilon \delta_x} \geq 0\]
vanishes outside a compact interval, and it converges uniformly to zero as $\epsilon \to 0$.

On the other hand, as $\mu \leq_c \nu$ is irreducible, the difference 
$u_\nu - u_\mu\geq0$ is uniformly bounded away from zero on compact subsets of $I$ and has nonzero derivative on $J\setminus I$. Together, it follows that
\begin{equation}\label{eq:ineqForMbl}
  u_\nu - u_{\shadow{\epsilon \delta_x}{\nu}} + u_{\epsilon \delta_x} \geq u_\mu
\end{equation}
for small enough $\epsilon>0$, so that
\[\theta := \nu - \shadow{\epsilon \delta_x}{\nu} + \epsilon \delta_x\]
satisfies $\mu \leq_c \theta \leq_c \nu$; moreover, $\theta \gg\delta_x$ as $\nu(\{x\})=0$.

\emph{Step~2.} 
We turn to the case of a general probability measure $\rho$ on $J$. By Step~1, we can find a measure $\theta_x$ for each $x\in J$ such that
\[\mu \leq_c \theta_x \leq_c \nu \quad \text{and} \quad \theta_x \gg \delta_x.\]
The map $x \mapsto \theta_x$ can easily be chosen to be measurable (\MN{by choosing the~$\epsilon$ for~\eqref{eq:ineqForMbl} in a measurable way}).
We can then define the probability measure 
\[\theta'(A) := \int_J \theta_x(A) \rho(dx),\quad A\in \B(\R)\]
which satisfies $\mu \leq_c \theta' \leq_c \nu$. Moreover, we have \FS{$\theta' \gg \rho$}; indeed, if $A\in \B(\R)$ is a $\theta'$-nullset, then $\theta_x(A) =0$ for $\rho$-a.e.\ $x$ and thus $\rho(A) = 0$ as $\theta_x \gg \delta_x$.

Finally, $\theta := (\mu + \theta' + \nu)/3$ shares these properties. As $u_{\mu}<u_{\nu}$ on $I$ due to irreducibility, we have $u_{\mu}<u_{\theta}<u_{\nu}$ on $I$ and it follows that $\mu \leq_c \theta$ and $\theta|_{I} \leq_c (\nu-\theta|_{J\setminus I})$ are irreducible.
\end{proof}

\begin{lemma}\label{le:polarFree}
Let $\mu_0 \leq_c \mu_n$ and let $\pi$ be a measure on $\R^{n+1}$ which is concentrated on an $n$-step component $V$ of $\M^n(\mu_0,\mu_n)$ and whose first and last marginals satisfy
\[\pi_0 \leq \mu_0, \quad \pi_n \leq \mu_n.\]
Then there exists $P \in \M^n(\mu_0,\mu_n)$ such that $P\gg\pi$.
\end{lemma}

\begin{proof}
  If $V=I_0^{n+1} \cap \Delta_n$, then $\pi$ must be an identical transport and we can take $P$ to be any element of $\M(\mu_{0},\mu_{0},\dots,\mu_{0},\mu_{n})$. Thus, we may assume that $V$ is of type~(i) or~(iii) in Definition~\ref{de:nStepComponent}, and then, by fixing $k\geq 1$, that $\mu_0 \leq_c \mu_n$ is irreducible with domain $(I,J)$.

Using Lemma~\ref{lem:intermediateDominator}, we can find intermediate marginals $\mu_{t}$ with 
\[\FS{\mu_0 \leq_c \mu_1 \leq_c \dots \leq_c \mu_{n-1} \leq_c \mu_n}\]
such that $\mu_t \gg \pi_t$ for all $1\leq t\leq n-1$, and each of the steps $\mu_{t-1} \leq_c \mu_t$, $1\leq t \leq n$ has a single irreducible domain given by $(I,J)$ as well as (possibly) a diagonal component on $J\setminus I$. We note that $V$ is an irreducible component of $\M(\mu_0, \mu_1, \dots, \mu_n)$ as introduced after Theorem~\ref{thm:polarstruct}.

Let $f_{t}=d\pi_t/d\mu_t$ be the Radon--Nikodym derivative of the marginal at date $t$. For $m\geq 1$, we define the measure $\pi^{m}\ll \pi$ by
\[
   \pi^{m}(dx_{0},\dots,dx_{n}) = 2^{-m} \left(\prod_{t = 1}^{n-1}\1_{f_t(x_t) \leq 2^m}\right)\pi(dx_{0},\dots,dx_{n}).
\]
Then, the marginals $\pi^{m}_{t}$ satisfy the stronger condition $\pi_t^m \leq \mu_t$ for $0\leq t\leq n$. 
Thus, we can apply Lemma~\ref{lem:polarstruct} to $\bmu = (\mu_0,\dots,\mu_n)$ and the
irreducible component $V$, to find $P^m \in \M(\bmu) \subseteq \M^n(\mu_0,\mu_n)$ such that $P^m\gg \pi^m$. Noting that $\sum_{m \geq 1} 2^{-m} \pi^m\gg \pi$, we see that $P := \sum_{m \geq 1} 2^{-m} P^m\gg \pi$ satisfies the requirements of the lemma.
\end{proof}

\begin{proof}[Proof of Theorem~\ref{thm:polarstructFree}]
  The result is deduced from Lemma~\ref{le:polarFree} by following the argument in the proof of Theorem~\ref{thm:polarstruct}.
\end{proof}

\subsection{Duality}

In this section we formulate a duality theorem for the transport problem with free intermediate marginals.

\begin{definition}
Let $f:\R^{n+1} \to [0,\infty]$. The \emph{primal problem} is
\[\S_{\mu_0,\mu_n}^n(f) := \sup_{P \in \M^n(\mu_0,\mu_n)} P(f) \in [0,\infty]\]
and the dual problem is
\[\I_{\mu_0,\mu_n}^n(f) := \inf_{(\phi,\psi,H) \in \D_{\mu_0,\mu_n}^n(f)} \mu_0(\phi) + \mu_n(\psi) \in [0,\infty],\]
where $\D_{\mu_0,\mu_n}^n(f)$ consists of all triplets $(\phi,\psi,H)$ such that
$(\phi,\psi) \in L^c(\mu_0,\mu_n)$ and $H = (H_1,\dots,H_n)$ is $\FF$-predictable with
\[\phi(X_0) + \psi(X_n) + (H \cdot X)_n \geq f \quad \M^n(\mu_0,\mu_n)\text{-q.s.}\]
i.e.\ the inequality holds $P$-a.s.\ for all $P\in\M^n(\mu_0,\mu_n)$.
\end{definition}

The analogue of Theorem~\ref{thm:duality} reads as follows.

\begin{theorem}[Duality]
\label{th:dualityFree}
Let $f:\R^{n+1} \to [0,\infty]$.
\begin{enumerate}[(i)]
\item If $f$ is upper semianalytic, then $\S_{\mu_0,\mu_n}^n(f) = \I_{\mu_0,\mu_n}^n(f) \in [0,\infty]$.
\item If $\I_{\mu_0,\mu_n}^n(f) < \infty$, there exists a dual optimizer 
$(\phi,\psi,H) \in \D_{\mu_0,\mu_n}^n(f)$.
\end{enumerate}
\end{theorem}

The main step for the proof is again a closedness result. We shall only discuss the case where $\mu_0 \leq_c \mu_n$ is irreducible; the extension to the general case can be obtained along the lines of Section~\ref{se:dualSpace}.

\begin{proposition}\label{pr:closednessFree}
Let $\mu_0 \leq_c \mu_n$ be irreducible and let $f^m : \R^{n+1} \to [0,\infty]$ be a sequence
of functions such that $f^m \to f$ pointwise. Moreover, let $(\phi^m,\psi^m,H^m) \in \D_{\mu_0,\mu_n}^n(f^m)$ be such that
$\sup_m \mu_0(\phi^m) + \mu_n(\psi^m) < \infty$. Then there exist $(\phi,\psi,H) \in \D_{\mu_0,\mu_n}^n(f)$ such that
\[ \mu_0(\phi) + \mu_n(\psi) \leq \liminf_{m \to \infty} \mu_0(\phi^m) + \mu_n(\psi^m).\]
\end{proposition}

\begin{proof}
Let $\mu_{t}$, $1\leq t\leq n-1$ be such that  $\bmu = (\mu_0,\dots,\mu_n)$ is in convex order and
$\mu_{t-1} \leq_c \mu_t$ is irreducible for all $1 \leq t \leq n$; such $\mu_{t}$ are easily constructed by prescribing their potential functions.
Setting $\bphi^m = (\phi^m,0,\dots,0,\psi^m)$ we have
$(\bphi^m, H^m) \in \D_{\bmu}^g(f^m)$ and can thus apply Proposition~\ref{prop:dualclosed}
to obtain $(\bphi,H) \in \D_{\bmu}^g(f)$. The construction in the proof of that proposition yields
$\phi_t \equiv 0$ for $1\leq t\leq n-1$. Therefore,
$(\phi_0,\phi_n,H) \in \D_{\mu_0,\mu_n}^n(f)$ and
\begin{align*}
\mu_0(\phi_0) + \mu_n(\phi_n) &= \bmu(\bphi) \leq \liminf_{m \to \infty} \bmu(\bphi^m) = \liminf_{m \to \infty} \mu_0(\phi^m) + \mu_n(\psi^m).
\end{align*}
\end{proof}

\begin{proof}[Proof of Theorem~\ref{th:dualityFree}]
  On the strength of Proposition~\ref{pr:closednessFree}, the proof is analogous to the one of Theorem~\ref{thm:duality}.
\end{proof}

\subsection{Monotone Transport}

The analogue of our result on left-monotone transports is somewhat degenerate: with unconstrained intermediate marginals, the corresponding coupling is the identical transport in the first $n-1$ steps and the (one-step) left-monotone transport in the last step. The full result runs as follows.

\begin{theorem}
Let $P \in \M^n(\mu_0,\mu_n)$. The following are equivalent:
\begin{enumerate}[(i)]
\item $P$ is a simultaneous optimizer
for $\S_{\mu_0,\mu_n}^n(f(X_0,X_t))$ for all smooth second-order Spence--Mirrlees functions $f$ and $1 \leq t \leq n$.
\item $P$ is concentrated on a left-monotone set $\Gamma\subset \R^{n+1}$ such that
\[\Gamma^{n-1} = \{(x,\dots,x) : x \in \Gamma^0\}.\]
\item For $0 \leq t \leq n-1$, we have $P \circ (X_t)^{-1} = \mu_0$  and
$P \circ (X_t,X_n)^{-1}$ is the (one-step) left-monotone transport in $\M(\mu_0,\mu_n)$.
\end{enumerate}
There exists a unique $P \in \M^n(\mu_0,\mu_n)$ satisfying (i)--(iii).
\end{theorem}

\begin{proof}
A transport $P$ as in~(iii) exists and is unique, because the
identical transport between equal marginals and the left-monotone transport in $\M(\mu_{0},\mu_{n})$ exist and are unique; cf.\ Proposition~\ref{prop:leftMonotoneTransport}.
The equivalence of~(ii) and~(iii) follows from the same proposition
and the fact that the only martingale transport from $\mu_{0}$ to $\mu_{0}$ is the identity.

Let $P\in \M^n(\mu_0,\mu_n)$ satisfy~(i). In particular, $P$ is then an optimizer for
$\S_{\mu_0,\mu_n}^n(f(X_0,X_n))$, which by Proposition~\ref{prop:leftMonotoneTransport} implies that $P_{0n}=P \circ (X_0,X_n)^{-1}$ is the (one-step) left-monotone transport in $\M(\mu_0,\mu_n)$.
For $t = 1,\dots,n-1$, $P$ is an optimizer for $\S_{\mu_0,\mu_n}^n(-\1_{\{X_0 \leq a\}}|X_t - b|)$, for
all $a,b \in \R$. This implies that $P_{0t}$ transports
$\mu_0|_{(-\infty,a]}$ to the minimal element of 
$\{\theta : \mu_0|_{(-\infty,a]} \leq_c \theta \leq_{pc} \mu_n\}$ in the sense of the convex order, which is $\theta =\mu_0|_{(-\infty,a]}$. Therefore, $P_{0t}$ must be the identical transport for $t = 1,\dots,n-1$ and all but the last marginal are equal to $\mu_0$.

Conversely, let $P\in \M^n(\mu_0,\mu_n)$ have the properties from~(iii). Then, $P$ is optimal for $\S_{\mu_0,\mu_n}^n(-\1_{\{X_0 \leq a\}}(X_t - b)^{+})$ for all $1\leq t\leq n$ and this can be extended to the optimality~(i) for smooth second-order Spence--Mirrlees functions as in the proof of Theorem~\ref{th:leftMonotoneOptimalitySpenceMirrlees}.
\end{proof}

\bibliography{stochfin}
\bibliographystyle{plain}

\end{document}